\documentclass[reqno]{amsart}
\usepackage{amsthm, amssymb, amsfonts}
\usepackage{lmodern}
\usepackage{color}
\usepackage{hyperref}
\usepackage[T5]{fontenc}
\usepackage{amscd,amssymb}
\usepackage[v2,cmtip]{xy}
\usepackage{mathrsfs}
\theoremstyle{plain}
\newtheorem{thm}{Theorem}[section]

\newtheorem{con}[thm]{Conjecture}

\theoremstyle{definition}

\newtheorem{rem}[thm]{Remark}

\theoremstyle{plain}
\newtheorem{thms}{Theorem}[subsection]
\newtheorem{props}[thms]{Proposition}
\newtheorem{lems}[thms]{Lemma}
\newtheorem{corls}[thms]{Corollary}
\newtheorem{cons}[thms]{Conjecture}
\theoremstyle{definition}
\newtheorem{defns}[thms]{Definition}
\newtheorem{rems}[thms]{Remark}
\newtheorem{notas}[thms]{Notation}
\allowdisplaybreaks
\begin{document} 
\title[A counter-example to Singer's conjecture]{A counter-example to Singer's conjecture\\ for the algebraic transfer}
\author{Nguy\~\ecircumflex n Sum}
\address{Department of Mathematics and Applications, S\`ai G\`on University, 273 An D\uhorn \ohorn ng V\uhorn \ohorn ng, District 5, H\`\ocircumflex\ Ch\'i Minh city, Viet Nam}

\email{nguyensum@sgu.edu.vn}


%-------------------------------------------------
\begin{abstract}
Write $P_k:= \mathbb F_2[x_1,x_2,\ldots ,x_k]$ for the polynomial algebra over the prime field $\mathbb F_2$ with two elements, in $k$ generators $x_1, x_2, \ldots , x_k$, each of degree 1. The polynomial algebra $P_k$ is considered as a module over the mod-2 Steenrod algebra, $\mathcal A$. 

Let $GL_k$ be the general linear group over the field $\mathbb F_2$. This group acts naturally on $P_k$ by matrix substitution. Since the two actions of $\mathcal A$ and $GL_k$ upon $P_k$ commute with each other, there is an inherit action of $GL_k$ on $\mathbb F_2{\otimes}_{\mathcal A}P_k$. Denote by $(\mathbb F_2{\otimes}_{\mathcal A}P_k)_n^{GL_k}$ the subspace of $\mathbb F_2{\otimes}_{\mathcal A}P_k$ consisting of all the $GL_k$-invariant classes of degree $n$. In 1989, Singer \cite{si1} defined the homological algebraic transfer
$$\varphi_k :\mbox{Tor}^{\mathcal A}_{k,n+k}(\mathbb F_2,\mathbb F_2) \longrightarrow (\mathbb F_2{\otimes}_{\mathcal A}P_k)_n^{GL_k},$$
where $\mbox{Tor}^{\mathcal{A}}_{k, k+n}(\mathbb{F}_2, \mathbb{F}_2)$  is the dual of Ext$_{\mathcal{A}}^{k,k+n}(\mathbb F_2,\mathbb F_2)$, the $E_2$ term of the Adams spectral sequence of spheres. In general, the transfer $\varphi_k$ is not a monomorphism and Singer made a conjecture that $\varphi_k$ is an epimorphism for any $k \geqslant 0$. The conjecture is studied by many authors. It is true for $k \leqslant 3$ but unknown for $k \geqslant 4$.

In this paper, by using a technique of the Peterson hit problem we prove that Singer's conjecture is not true for $k=5$ and the internal degree $n = 108$. 
\end{abstract}

\maketitle

%======================================
\section{Introduction}\label{s1} 
\setcounter{equation}{0}
Denote by $BS_k$ the classifying space of an elementary abelian 2-group $S_k$ of rank $k$. It is well-known that the cohomology algebra $P_k:= H^*(BS_k)$ is the polynomial algebra $\mathbb F_2[x_1,x_2,\ldots ,x_k]$  in $k$ variables $x_1, x_2, \ldots , x_k$, each of degree 1. Here the cohomology of the topological space $BS_k$ is taken with coefficients in the field $\mathbb F_2$ with two elements. 

Then, the algebra $P_k$ is considered as a module over the mod-2 Steenrod algebra, $\mathcal A$.  The action of $\mathcal A$ on $P_k$ is explicitly determined from the elementary properties of the Steenrod operations $Sq^r$ and the Cartan formula. For details, the readers can see Steenrod and Epstein~\cite{st}.

Let $GL_k = GL_k(\mathbb F_2)$ be the general linear group over the field $\mathbb F_2$. This group acts naturally on $P_k$ by matrix substitution. Since the two actions of $\mathcal A$ and $GL_k$ upon $P_k$ commute with each other, there is an action of $GL_k$ on $QP_k := \mathbb F_2{\otimes}_{\mathcal A}P_k$. 

For a nonnegative integer $n$, denote by $(P_k)_n$  the subspace of $P_k$ consisting of all degree $n$ homogeneous polynomials and by $(QP_k)_n$ the subspace of $QP_k$ consisting of all the classes represented by homogeneous polynomials of degree $n$ in $P_k$. Then  $(P_k)_n$ and $(QP_k)_n$ are $GL_k$-submodules of $P_k$ and $QP_k$ respectively. 

In 1989, Singer \cite{si1} defined the homological algebraic transfer
$$\varphi_k :\mbox{Tor}^{\mathcal A}_{k,n+k}(\mathbb F_2,\mathbb F_2) \longrightarrow (QP_k)_n^{GL_k},$$
where $(QP_k)_n^{GL_k}$ denote the subspace of $(QP_k)_n$ consisting of all the $GL_k$-invariant classes in $(QP_k)_n$ and $\mbox{Tor}^{\mathcal{A}}_{k, k+n}(\mathbb{F}_2, \mathbb{F}_2)$  is the dual of Ext$_{\mathcal{A}}^{k,k+n}(\mathbb F_2,\mathbb F_2)$, the $E_2$ term of the Adams spectral sequence of spheres.

The algebraic transfer was studied by Boardman \cite{bo}, Bruner-H\`a-H\uhorn ng \cite{br}, H\`a \cite{ha}, H\uhorn ng \cite{hu1, hu2, hu3}, Ch\ohorn n-H\`a \cite{cha,cha1}, H\uhorn ng-Qu\`ynh \cite{hq}, Minami \cite{mi}, Nam \cite{na2}, Ph\'uc \cite{p231,p232,p24}, Qu\`ynh \cite{qh}, and others.

Singer showed in \cite{si1} that $\varphi_k$ is an isomorphism for $k=1,\, 2$. Boardman showed in \cite{bo} that $\varphi_3$ is also an isomorphism. However, Singer showed in \cite{si1} that $\varphi_5$ is not a mononorphism in the degree 9. Then, H\uhorn ng proved in \cite{hu3} that for any $k\geqslant 4$,  $\varphi_k$ is not a mononorphism in infinitely many degrees. Singer made the following conjecture.

\begin{con}[Singer \cite{si1}]\label{sconj} The algebraic transfer $\varphi_k$ is an epimorphism for any $k \geqslant 0$.
\end{con}

Recently, Ph\'uc stated in \cite{p231} that the conjecture is also true for $k = 4$ but the proof is not explicit and the computations are false and incomplete.

The purpose of the paper is to give a negative answer to this conjecture. The following is the main result of the paper.
\begin{thm}\label{dlc} Singer's conjecture is not true for $k=5$.
\end{thm}

We prove this theorem by studying the Peterson hit problem of determining a minimal set of generators for the polynomial algebra $P_5$ as a module over the Steenrod algebra. More precisely, we explicitly determine a set of generators for the $\mathbb F_2$-vector space $(QP_5)_{108}$. Using this result we prove the following.
\begin{thm}\label{thm1} There is a nonzero class $[p]$ in the $\mathbb F_2$-vector space $(QP_5)_{108}^{GL_5}$ such that $[p] \in \langle [p_j] : 1 \leqslant j \leqslant 14\rangle$. Here, the polynomials $p_j$, $1 \leqslant j \leqslant 14$, are defined as in Subsection \ref{s52}.
\end{thm}
 From the results in Lin \cite{wl} and Chen \cite{che}, we have $\mbox{Ext}^{5,113}_{\mathcal A}(\mathbb F_2,\mathbb F_2) =0$. By passing to the dual, one gets $\mbox{Tor}_{5,113}^{\mathcal A}(\mathbb F_2,\mathbb F_2)=0$. Hence, Theorem \ref{thm1} implies that the homomorphism
$$\varphi_5: \mbox{Tor}_{5,113}^{\mathcal A}(\mathbb F_2,\mathbb F_2) \to (QP_5)_{108}^{GL_5}$$
is not an epimorphism. Theorem \ref{dlc} is proved.

In \cite[Theorem 2.10]{Tij}, by using a computer program based on an algorithm of the mathematics system SAGEMATH, T\'in proved a dimensional result that $\dim(QP_5)_{108} = 2171$. This result is also confirmed in Ph\'uc \cite{p24} by using a computer program based on an algorithm of the MAGMA. However, there are no any details of computations for a basis of $(QP_5)_{108}$.

To prove Theorem \ref{thm1}, we explicitly determine a set of generators for $(QP_5)_{108}$. We prove the following.

\begin{thm}\label{thm2} There exist a set of generators for $(QP_5)_{108}$ consisting of all the classes represented by 2184 monomials which are determined as in Section \ref{s51}.
\end{thm}

Since the basis of $(QP_5)_{108}$ has not yet been determined, we have been unable to explicitly determine the space $(QP_5)_{108}^{GL_5}$. We prove the existence of the polynomial $p$ by studying the $\mathbb F_2$-vector spaces $(\widetilde {QP}_5)_{n}$ and $(\widetilde {SF}_k)_n$ as defined below Lemma \ref{bdm}. By using Theorem \ref{thm2} we prove the following.
\begin{thm}\label{thm3} The space $(\widetilde {SF}_5)_{108} = \langle [p_j] : 1 \leqslant j \leqslant 14\rangle$ is an $GL_5$-submodule of $(QP_5)_{108}$ and $\dim(\widetilde {QP}_5)_{108} = 2170.$ 
\end{thm}

Since $(QP_5)_{108} \cong (\widetilde {QP}_5)_{108}\bigoplus (\widetilde {SF}_5)_{108}$ and $\dim({QP}_5)_{108} = 2171,$ from Theorem \ref{thm3} we obtain $\dim (\widetilde {SF}_5)_{108} = 1$. Hence, $(\widetilde {SF}_5)_{108} = \langle [p]\rangle$ with $p$ a nonzero class in $(QP_5)_{108}$. By Theorem \ref{thm3}, $(\widetilde {SF}_5)_{108}$ is an $GL_5$-module, hence $[p]$ is an $GL_5$-invariant. Theorem \ref{thm1} is proved.

\begin{rem} After the initial manuscript was posted, we discovered several errors in Lemma 3.2.5 that seriously affected the main results. Consequently, that manuscript was not submitted to any journal.
	
We also note that the results of the paper depend on $\dim({QP}_5)_{108} = 2171$ which was established by using the MAGMA and SAGEMATH computer programs. If it is correct, then there exists a unique index $j$, $1 \leqslant j \leqslant 14$, such that the monomials $a_{1456+j}$ in Subsection \ref{s51} is admissible and $[p] = [p_j].$ If $\dim({QP}_5)_{108} \ne 2171,$ then the result of this paper is no longer valid. Then, all results obtained by using the MAGMA and SAGEMATH computer programs are unreliable.
\end{rem}

Consider $n_d = 2^{d+3} + 2^{d+2} - 4$ with $d$ a positive integer. Based on the results in \cite{subs}, we have showed that $\dim(QP_5)_{n_2}^{GL_5} = 1$. Based on the result $\dim({QP}_5)_{n_3} = 1706$ which is confirmed in \cite{p24} and \cite{Tij} by using the MAGMA and SAGEMATH computer programs we showed that $(\widetilde {SF}_5)_{n_3} = \langle [\tilde p]\rangle$ with $[\tilde p]$ a nonzero class in $(QP_5)_{n_3}^{GL_5}$. It is easy to see that $(\widetilde {SF}_5)_{n_d} = 0$ for $d \geqslant 4$. We predict that the element $h_0g_d$ is not in the image of $\varphi_5^*$ for $d \geqslant 2$. Here, $g_d \in \mbox{Ext}^{4,n_d+4}_{\mathcal A}(\mathbb F_2,\mathbb F_2)$. This implies the following.

\begin{con} Singer conjecture for the fifth algebraic transfer is not true in the internal degrees $n_2 = 44$ and $n_3 = 92$ and it is true in the degree $n_d$ for $d \geqslant 4$.
\end{con}

The proofs of Theorems \ref{thm2} and \ref{thm3} are extremely complicated. We prove them by combining hand computation with the aids of the Microsoft Excel software.

The paper is organized as follows. In Section \ref{s2}, we recall some definitions and results on the admissible monomials in $P_k$, Singer's criterion on the hit monomials and establish some needed notations. Theorem \ref{thm2} is proved in Section \ref{s3} by explicitly determining a system of $\mathcal A$-generators for $P_5$ in degree $108$. The proof of Theorem \ref{thm3} is presented in Section \ref{s4}. A set of generators for $(QP_5)_{108}$ is explicitly presented in Section \ref{s5}. 

%====================================
\section{Background on the Peterson hit problem}\label{s2}
\setcounter{equation}{0}

The purpose of the section is to present some needed definitions and results on the weight vector of a monomial, admissible monomial, Singer's criterion for hit monomial from the work of Kameko~\cite{ka}, Singer \cite{si2} and the present author \cite{su2,su4}.

\subsection{The weight vector and the admissible monomials}\

\medskip
\begin{defns} Let $w =x_1^{b_1}x_2^{b_2}\ldots x_k^{b_k}$ be a monomial  in $P_k$. Denote $\nu_j(w) = b_j, 1 \leqslant j \leqslant k$. We define 
\begin{align*} 
\omega(w)&=(\omega_1(w),\omega_2(w),\ldots , \omega_i(w), \ldots),\ \
\sigma(x) = (\nu_1(w),\nu_2(w),\ldots ,\nu_k(w)),
\end{align*}
where
$\omega_i(w) = \sum_{1\leqslant j \leqslant k} \alpha_{i-1}(\nu_j(w)),\ i \geqslant 1.$
The sequence $\omega(w)$ is called the weight vector and $\sigma(w)$ is called the exponent vectors of $w$. 
	
Let $\omega= (\omega_1,\omega_2,\ldots , \omega_i, \ldots)$ be sequence of non-negative integers. The sequence $\omega$ is called a weight vector if $\omega_i = 0$ for $i \gg 0$. We define $\deg \omega = \sum_{i > 0}2^{i-1}\omega_i$, the length $\ell(\omega) = \max\{j : \omega_j >0\}$. Write $\omega= (\omega_1,\omega_2,\ldots , \omega_r)$ if $\omega_j = 0$ for $j > r$.

For weight vectors $\omega^{(1)} = (\omega_1^{(1)},\omega_2^{(1)},\ldots)$ and $\omega^{(2)} = (\omega_1^{(2)},\omega_2^{(2)},\ldots)$, we define the concatenation of weight vectors $\omega^{(1)}|\omega^{(2)} = (\omega_1^{(1)},\ldots,\omega_r^{(1)},\omega_1^{(2)},\omega_2^{(2)},\ldots)$ if $\ell(\omega^{(1)}) = r$ and $(a)|^s = (a)|(a)|\ldots|(a)$, ($s$ times of $(a)$'s), where $a,\, s$ are positive integers.
\end{defns}
The sets of weight vectors and exponent vectors respectively are equipped with the left lexicographical order.  
 
Denote by $P_k(\omega) = \langle u \in P_k: \deg u = \deg\omega,\mbox{ and } \omega(u)\leqslant \omega\rangle$ and  $P_k^-(\omega) = \langle u \in P_k(\omega): \omega(u) < \omega\rangle$. 

\begin{defns} Suppose $\omega$ is a weight vector and $f,\, h$ are two polynomials of the same degree in $P_k$. We define
	
i) $f \equiv h$ if $f+h \in \mathcal A^+P_k$ and $f$ is called hit if $f \equiv 0$.
	
ii) $f \equiv_{\omega} h$ if $f + h \in \mathcal A^+P_k+P_k^-(\omega)$. 
\end{defns}

Observe that $\equiv$ and $\equiv_{\omega}$ are equivalence relations. Set 
$$QP_k(\omega)= P_k(\omega)/ ((\mathcal A^+P_k\cap P_k(\omega))+P_k^-(\omega)).$$   

\begin{props}[\cite{su3}] If $\omega$ is a weight vector, then the space $QP_k(\omega)$ is an $GL_k$-module. 
\end{props}

\begin{defns} 
For $a,\, b$ monomials in $P_k$ with $\deg a = \deg b$, we define $a < b$ if one of the following conditions satisfies:  
	
i) $\omega (a) < \omega(b)$;
	
ii) $\omega (a) = \omega(b)$ and $\sigma(a) < \sigma(b).$
\end{defns}

\begin{defns}\label{dfnad}
Let $w$ be a monomial in $P_k$. The monomial $w$ is said to be inadmissible if there are monomials $u_1,u_2,\ldots, u_s$ such that $u_j<w$ for $j=1,2,\ldots , s$ and $w+ \sum_{j=1}^su_j \in \mathcal A^+P_k.$ 
If $w$ is not inadmissible, then it is called admissible.
\end{defns} 

We can see that, all the admissible monomials of degree $m$ in $P_k$ form a minimal set of $\mathcal{A}$-generators for $P_k$ in degree $m$.

We denote $\mathcal A_h$ the sub-Hopf algebra of $\mathcal A$ generated by the Steenrod squares $Sq^r$ with $0\leqslant r < 2^h$, and $\mathcal A_h^+ = \mathcal A^+\cap\mathcal A_h$. 

\begin{defns} 
Let $w$ be a monomial $x$ in $P_k$. The monomial $w$ is said to be strictly inadmissible if there are monomials $u_1,u_2,\ldots, u_s$ such that $u_j<x,$ for $j=1,2,\ldots , s$ and $w + \sum_{j=1}^s u_j\in \mathcal A_{r}^+P_k$ with $r = \max\{j : \omega_j(w) > 0\}$.
\end{defns}
We recall the following in Kameko \cite{ka} and our work \cite{su1}].
\begin{thms}\label{dlcb1}  
Let $u, v, w$ be monomials in $P_k$ such that $\omega_i(u) = 0$ for $i > t>0$, $\omega_i(w) = 0$ for $i > s > 0$. Then we have
	
{\rm i)}  If  $w$ is inadmissible, then the monomial $uw^{2^t}$ is also inadmissible.
	
{\rm ii)}  If $w$ is strictly inadmissible, then so is the monomial $wv^{2^{s}}$.
\end{thms} 

We need a result of Singer \cite{si2} on the hit monomials in $P_k$. 

\begin{defns}\label{spi}  Let $z$ be a monomial in $P_k$. If $\nu_j(z)=2^{t_u}-1$ with $t_u$ a non-negative integer for $u=1,2, \ldots , k$, then $z$ is called a spike. It is called the minimal spike if $t_1>t_2>\ldots >t_{h-1}\geqslant t_h>0$ and $t_j=0$ for $j>h$.
\end{defns}

\begin{thms}[Singer~{\cite{si2}}]\label{dlsig} Let $f$ be a monomial of degree $d$ in $P_k$ with $\mu(d) \leqslant k$ and let $z$ be the minimal spike of degree $d$. If $\omega(f) < \omega(z)$, then the monomial $f$ is hit.
\end{thms}

\subsection{Some other results and notations}\

\medskip
We set 
$
P_k^0 =\langle\{u\in P_k : \nu_t(u)=0 \mbox{ for some } t\}\rangle$ and 
$P_k^+ = \langle\{u\in P_k :  \nu_t(u)>0, \mbox{ for all } t\}\rangle. 
$ We see that $P_k^0$ and $P_k^+$ are the $\mathcal{A}$-submodules of $P_k$ and $QP_k =QP_k^0 \oplus  QP_k^+.$ Here, we denote $QP_k^0 = P_k^0/\mathcal A^+P_k^0$ and  $QP_k^+ = P_k^+/\mathcal A^+P_k^+$.

Set 
$\mathcal N_k =\{(j;J) : J=(h_1,h_2,\ldots,h_s),1 \leqslant  j < h_1 <  \ldots < h_s\leqslant  k,\ 0\leqslant s <k\}.$
For $(j;J) \in \mathcal N_k$, we define the homomorphism of $\mathcal A$-algebras $p_{(j;J)}: P_k \to P_{k-1}$  by substituting
\begin{equation}\label{ct23}
p_{(j;J)}(x_t) =\begin{cases} x_t, &\mbox{ if } 1 \leqslant t < j,\\
\sum_{u\in I}x_{u-1}, &\mbox{ if }  t = j,\\  
x_{t-1},&\mbox{ if } j< t \leqslant k.
\end{cases}
\end{equation}

\begin{lems}[\cite{sp}]\label{bdm} For any monomial $u$ in $P_k$, $p_{(j;J)}(u) \in P_{k-1}(\omega(u)).$
\end{lems}

By Lemma \ref{bdm}, if $u \in P_k(\omega)$, then $p_{(j;J)}(u) \in P_{k-1}(\omega)$ for any weight vector $\omega$. Hence, we get the homomorphisms  
\begin{align*}
&p_{(j;J)}^{(\omega)} :QP_k(\omega)\longrightarrow QP_{k-1}(\omega),\ \
p_{(j;J)}^{(d)} :(QP_k)_m\longrightarrow (QP_{k-1})_d,
\end{align*} 
where $d = \deg\omega$. We denote 
\begin{align*}
&\widetilde {SF}_k(\omega) = \bigcap_{(j;J)\in \mathcal N_k}  \mbox{Ker}(p_{(j;J)}^{(\omega)}) \mbox{ and } (\widetilde {SF}_k)_d = \bigcap_{(j;J)\in \mathcal N_k}  \mbox{Ker}(p_{(j;J)}^{(d)}), \\
&\widetilde{QP}_k(\omega) = QP_k(\omega)/\widetilde {SF}_k(\omega) \mbox{ and } (\widetilde{QP}_k)_d = (QP_k)_d/(\widetilde {SF}_k)_d,
\end{align*}
It is easy to see that $\widetilde {SF}_k(\omega)$ and $(\widetilde {SF}_k)_d$ are the subspaces of $QP_k^+(\omega)$ and $(QP_k^+)_d$ respectively.

From the results in Kameko \cite{ka} and our work \cite{su2}, we see that for $k \leqslant 4$, the spaces $\widetilde {SF}_k(\omega)$ and $(\widetilde {SF}_k)_d$ are the $GL_k$-submodules of $QP_k(\omega)$ and $(QP_k)_d$ respectively. Based on this result and the ones in Adams, Gunawardena and Miller~\cite{agm}, one can give the following prediction.

\begin{cons} For any weight vector $\omega$ of degree $d$, the spaces $\widetilde {SF}_k(\omega)$ and $(\widetilde {SF}_k)_d$ are the $GL_k$-submodules of $QP_k(\omega)$ and $(QP_k)_d$ respectively.
\end{cons}

\begin{notas} For any $h \in  P_k$, we denote $[h]$ the class in $QP_k$ represented by $h$. If $h \in  P_k(\omega)$, we denote $[h]_\omega$ the class in $QP_k(\omega)$ represented by $h$. 
For a subset $L$ of $P_k$, we denote $[L] = \{[h] : h \in L\}$. If $L \subset P_k(\omega)$, then we set $[L]_\omega = \{[h]_\omega : h \in L\}$. Denote by $|L|$ the cardinal of $L$.
	
Let $d$ be a nonnegative integer. Denote by $B_{k}(d)$ the set of all admissible monomials of degree $d$ in $P_k$. Set 
$B_{k}^0(d) = B_{k}(d)\cap P_k^0,\ B_{k}^+(d) = B_{k}(d)\cap P_k^+.$
For any weight vector $\omega$ of degree $d$, we denote 
$B_k(\omega) = B_{k}(d)\cap P_k(\omega),\ B_k^+(\omega) = B_{k}^+(d)\cap P_k(\omega),\ QP_k^+(\omega) := QP_k(\omega)\cap QP_k^+.$
\end{notas}

For $I= (s_1, s_2, \ldots, s_r),\, 1 \leqslant s_1 <\ldots < s_r \leqslant k$, we define a monomorphism $\theta_I: P_r \to P_k$ of $\mathcal A$-algebras by substituting 
\begin{equation}\label{ctbs}
\theta_I(x_v) = x_{s_v} \ \mbox{ for } \ 1 \leqslant v \leqslant r.
\end{equation} 
Clearly, for an arbitrary weight vector $\omega$ of degree $d$, 
\[Q\theta_I(P_r^+)(\omega) \cong  QP_r^+(\omega)\mbox{ and } (Q\theta_I(P_r^+))_d \cong (QP_r^+)_d\] 
for $1 \leqslant r \leqslant k$. Here, $Q\theta_I(P_r^+) = \theta_I(P_r^+)/\mathcal A^+\theta_I(P_r^+)$. Then we have
$$B_k(\omega) = \bigcup_{\mu(d) \leqslant r\leqslant k,\atop \ell(I)=r} \theta_I(B_r^+(\omega)),\ B_k(d) = \bigcup_{\mu(d) \leqslant r\leqslant k,\atop \ell(I)=r} \theta_I(B_r^+(d))$$

Set $J_r = (1,\ldots,\hat r,\ldots,k)$ for $1 \leqslant r \leqslant k$.

\begin{props}[Mothebe-Uys \cite{mo}]\label{mdmo} Suppose $r,\, h$ are integers such that $1 \leqslant r \leqslant k$. If $f\in B_{k-1}(d)$, then $x_r^{2^h-1}\theta_{J_r}(f)\in B_{k}(d+2^h-1)$.
\end{props}

%====================================
\subsection{A construction for the generators of $P_k$ }\

\begin{defns}[\cite{su2}]\label{dfn1} We set $x_{(I,r)} = x_{i_r}^{2^{s-1}+\ldots + 2^{s-r}}\prod_{r< t \leqslant s}x_{i_t}^{2^{s-t}}$ for any $(j;J) \in \mathcal N_k$ with $J = (i_1, \ldots,i_s),\ s >0$. 
For a monomial $w\in P_{k-1}$, we define  $\phi_{(j;J)}(w)\in P_k$ by setting
$$ \phi_{(j;J)}(w) = \begin{cases} \theta_{J_j}(w), &\mbox{if } s =  0,\\ 
(x_j^{2^r-1}\theta_{J_j}(w))/x_{(J,r)},
&\mbox{if there is $1 \leqslant r \leqslant s$ such that }\\ &\quad \nu_{i_1-1}(w)= \ldots = \nu_{i_{(r-1)}-1}(w)=2^{s} - 1,\\
&\quad \nu_{i_r-1}(w) > 2^{s} - 1,\\
&\quad \alpha_{s-t}(\nu_{i_r-1}(w)) = 1,\ \forall t,\ 1 \leqslant t \leqslant  r,\\
&\quad \alpha_{s-t}(\nu_{i_t-1} (w)) = 1,\ \forall t,\ r < t \leqslant r,\\
0, &\mbox{otherwise.}
\end{cases}$$
\end{defns}
For a subset $B \subset P_k$, we denote
\begin{align*}\Phi^0(B) = \bigcup_{1\leqslant j \leqslant k}\phi_{(j;\emptyset)}(B), \ \
\Phi^+(B) = \bigcup_{(j;J)\in \mathcal N_k, 0<\ell(J)\leqslant k-1}\phi_{(j;J)}(B)\setminus P_n^0,
\end{align*} 
and $\Phi(B) = \Phi^0(B)\bigcup \Phi^+(B)$. 
It is easy to see that for any minimal set of generators $B$ of $\mathcal A$-module $P_{k-1}$ in degree $d$, $\Phi^0(B)$ is a minimal set of generators for $\mathcal A$-module $P_k^0$ in degree $d$ and $\Phi^+(B) \subset P_k^+$.
\begin{thms}[{See \cite{su2,su5}}]\label{dl1} Let $n = (k-1)(2^{d}-1) + q.2^{d}$ with $d, q$ positive integers such that $1 \leqslant k-3 \leqslant \mu(q) \leqslant k-2$,  $\alpha(q + \mu(q)) = \mu(q)$. If $d \geqslant k-1$, then $B_k(n) = \Phi(B_{k-1}(n))$.
\end{thms}

In \cite[Pages 445-446]{su2} we have proved that from Theorem \ref{dl1} we get 
$$\dim(QP_k)_{n} = (2^k-1)\dim (QP_{k-1})_{n}.$$ Then we set $h_u = (k-1)(2^{d - u}-1)+ q.2^{d-u}$ for $0\leqslant u \leqslant d$, $h_0 = n$, $h_{d} = q$. By using a result in Kameko \cite[Theorem 1.1]{ka}, we get
\begin{align}\label{ctdc}(QP_{k-1})_{h_u} \cong (QP_{k-1})_{h_{u-1}},\ 1 \leqslant u \leqslant  d.
\end{align}
This means that $\dim (QP_{k})_{n} = (2^k-1)\dim (QP_{k-1})_{h_u}$ for $1 \leqslant u \leqslant  d$.
In particular, we have 

\begin{thms}[See \cite{su2}]\label{dl2} 
Let $n = (k-1)(2^{d}-1) + q.2^{d}$ with $d, q$ positive integers such that $1 \leqslant k-3 \leqslant \mu(q) \leqslant k-2$ and $\alpha(q + \mu(q)) = \mu(q)$. If $d \geqslant k-1$, then
$$\dim (QP_k)_n = (2^k-1)\dim (QP_{k-1})_q.$$
\end{thms}
By combining Theorem \ref{dl1} and \ref{dl2} we easily obtain the following.
\begin{corls}\label{hq} Let $\omega$ be a weight vector. If $d \geqslant k-1$, then $B_k((k-1)|^d|\omega) = \Phi(B_{k-1}((k-1)|^d|\omega))$ and $\dim QP_k((k-1)|^d|\omega) = (2^k-1)\dim QP_{k-1}(\omega)$.
\end{corls}

\begin{rems} Ph\'uc presents a formula in \cite{p24} as follows. Let $n_s := (k - 1)(2^s - 1)+ q.2^s$. If there is a non-negative integer $\zeta$ such that $\zeta < s$ and $1 \leqslant k - 3 \leqslant \mu(n_\zeta) = \alpha(n_\zeta +\mu(n_\zeta )) \leqslant k - 2$, then with the generic degrees of the form $(k - 1)(2^{s - \zeta+k-1} - 1) + n_\zeta.2^{s-\zeta+k-1}$, where $q = n_\zeta$ and $s \geqslant \zeta \geqslant 0$,
he has shown in \cite{p232} that for every $s \geqslant \zeta$,
\begin{equation}\label{ctdv}	
\dim(QP_k)_{(k - 1)(2^{s - \zeta+k-1} - 1) + n_\zeta.2^{s-\zeta+k-1}} = (2^k - 1)\dim(QP_{k-1})_{n_s}.
\end{equation}	
 
However, with $q = n_\zeta = (k - 1)(2^\zeta - 1)+ q.2^\zeta$, we get $\zeta = 0$. Then, $n = (k - 1)(2^{s - \zeta+k-1} - 1) + n_\zeta.2^{s-\zeta+q-1} = (k - 1)(2^{s +k-1} - 1) + q.2^{s+k-1}$. 
Thus, the author of \cite{p24} obtains the formula \eqref{ctdv} by plagiarizing from the proof of Theorem \ref{dl2} in \cite[Pages 445-446]{su2} with $d = s+k-1$ and $n_s = h_{k-1}$.

For the case $\zeta \geqslant 0$, in \cite{p232}, the author write $n_s := (k - 1)(2^s - 1)+ r.2^s$, so $n_\zeta = (k - 1)(2^\zeta - 1)+ r.2^\zeta$. Hence, $n_s = (k - 1)(2^{s-\zeta} - 1)+ n_\zeta.2^{s-\zeta}$. Thus, the formula  \eqref{ctdv} in Ph\'uc \cite[Page 15]{p232} is also plagiarized from the proof of Theorem \ref{dl2} in \cite{su2} with $q = n_\zeta$, $d = s-\zeta+k-1$ and $n_s = h_{k-1}$. So, Ph\'uc's explanations in \cite[Appendix]{pp24} are false. Kameko's isomorphism is used to prove \eqref{ctdc}. Therefore, there are no any new contributions of the author \cite{p232} for the formula \eqref{ctdv}. Recently, the correctness of the above comments are confirmed in \cite{ppbs}.

%Besides, there are many other details in \cite{p232} which are plagiarized from our works. For example, the proof of the case ``if'' in  \cite[Page 8]{p232} is copied from the proof of Lemma 2 in our work \cite[Pages 1087-1088]{su3}.
\end{rems}

\section{Proof of Theorem \ref{thm2}}\label{s3}
\setcounter{equation}{0} 

In this sction, we explicitly determine all admissible monomials of the degree 108 in $P_5$.

\subsection{The weight vector of admissible monomials in the space $(P_5)_{108}$}\

\medskip
First of all, we determine the weight vectors of the admissible monomials of the degree 108 in $P_5$.

\begin{lems}\label{bd1} If $x$ is an admissible monomial of degree $n = 2^{d+3} + 2^{d+2} +2^{d+1} -4$ in $P_5$ with an integer $d \geqslant 2$, then \begin{equation*}
\omega(x) = (4)|^{d}|(2)|^2|(1) \mbox{ or } \omega(x) = (4)|^{d+1}|(1)|^2 \mbox{ or } \omega(x) =(4)|^{d+1}|(3).
	\end{equation*}
\end{lems}
The result of the lemma is stated in \cite{p24} but there are no any details of the proof. We need a technical result for the proof of this lemma.  

\begin{lems}\label{bd2} Let $(i,j,t,u,v)$ be a permutation of $(1,2,3,4,5)$. The following monomials are strictly inadmissible:
	
\medskip
{\rm i)} $x_i^{2}x_jx_t^{3}x_u^{3}x_v^{3},\, x_i^{3}x_j^{4}x_t^{7}x_u^{7}x_v^{7},\, i<j.$

\smallskip
{\rm ii)} $x_i^{7}x_j^{8}x_t^{11}x_u^{11}x_v^{15}$,\ $x_i^{7}x_j^{9}x_t^{10}x_u^{11}x_v^{15}$,\ $x_i^{7}x_j^{9}x_t^{11}x_u^{11}x_v^{14}$,\, $x_i^{7}x_j^{11}x_t^{11}x_u^{13}x_v^{10}$,\,

\quad \ $x_i^{7}x_j^{11}x_t^{11}x_u^{11}x_v^{12}$.

\end{lems}
\begin{proof} Part i) is proved in \cite{sp2}. Note that all monomials in Part ii) are of weight vector $(4,4,2,4)$. We prove Part ii) for $x = x_1^{7}x_2^{11}x_3^{11}x_4^{11}x_5^{12}$. The computations for others are easy. By using the Cartan formula, we have
\begin{align*}
x &= Sq^1\big(x_i^{7}x_j^{7}x_t^{7}x_u^{11}x_v^{19} + x_i^{7}x_j^{7}x_t^{7}x_u^{19}x_v^{11} + x_i^{7}x_j^{11}x_t^{11}x_u^{9}x_v^{13}\big) +  Sq^2\big(x_i^{7}x_j^{7}x_t^{7}x_u^{10}x_v^{19}\\ 
&\quad + x_i^{7}x_j^{7}x_t^{7}x_u^{18}x_v^{11} + x_i^{7}x_j^{7}x_t^{11}x_u^{3}x_v^{22} + x_i^{7}x_j^{7}x_t^{19}x_u^{3}x_v^{14}\big) +  Sq^4\big(x_i^{11}x_j^{7}x_t^{7}x_u^{11}x_v^{12}\\ 
&\quad + x_i^{11}x_j^{7}x_t^{7}x_u^{10}x_v^{13} + x_i^{11}x_j^{7}x_t^{11}x_u^{5}x_v^{14} + x_i^{11}x_j^{7}x_t^{13}x_u^{3}x_v^{14}\big) +  Sq^8\big(x_i^{7}x_j^{7}x_t^{7}x_u^{11}x_v^{12}\\ 
&\quad + x_i^{7}x_j^{7}x_t^{7}x_u^{10}x_v^{13} + x_i^{7}x_j^{7}x_t^{11}x_u^{5}x_v^{14} + x_i^{7}x_j^{7}x_t^{13}x_u^{3}x_v^{14}\big) \  \mbox{ mod}(P_5^-(4,4,2,4)).  
\end{align*}
The above equality shows that the monomial $x$ is strictly inadmissible.
\end{proof}

\begin{proof}[Proof of Lemma \ref{bd1}] Let $x$ be an admissible monomial of degree $n = 2^{d+3} + 2^{d+2} +2^{d+1} -4$ in $P_5$ with $d$ a positive integer. By Lemma 2.14 in \cite{su1}, $\omega_i(x) = 4$ for $1 \leqslant i \leqslant d$. Hence, $x = \bar xy^{2^{d}}$ with $\bar x$ a monomial of weight vector $(4)|^{d}$ and $y$ a monomial of degree 10 in $P_5$. By a result in T\'in \cite{Tiau}, $\omega(y)$ is one of the weight vectors $(2,2,1)$, $(2,4)$, $(4,1,1)$, $(4,3)$. If $d\geqslant 2$ and $\omega(y) = (2,4)$, then either there is a monomial $w$ as given in Lemma \ref{bd2}(i) such that $x = y_1w^{2^r}y_2^{2^{d-r-s}}$ with $y_1$ a monomial of weight vector $(4)|^r$ and $y_2$ a monomial of weight vector $(4)|^{d-r-s}|(2)|(4)$, $2\leqslant s \leqslant 3$, $r+s \leqslant d$, or there is a monomial $u$ as given in Lemma \ref{bd2}(ii) such that $x = y_3u^{d-2}$ with $y_3$ a monomial of weight vector $(4)|^{d-2}$. By Theorem \ref{dlcb1}, $x$ is inadmissible. Hence, the lemma is proved.
\end{proof}

By applying Lemma \ref{bd1} with $d = 3$, we obtain the following.

\begin{corls} We have
$$(QP_5)_{108} \cong QP_5((4)|^3|(2)|^2|(1))\bigoplus QP_5((4)|^4|(1)|^2)\bigoplus QP_5((4)|^4|(3)).$$	
\end{corls}

According to \cite{su2}, 
\begin{align*}
&\dim QP_4((4)|^4|(1)|^2) = \dim QP_4((1)|^2)= 10,\\
&\dim QP_4((4)|^4|(3)) = \dim QP_4((3))= 4.
\end{align*}
Hence, by using Corollary \ref{hq} we get 
$$\dim QP_5((4)|^4|(1)|^2) = 31\times 10 = 310,\ \dim QP_5((4)|^4|(3)) = 31\times 4 = 124.$$ 
Thus, we need only to determined the space $QP_5((4)|^3|(2)|^2|(1))$.

\subsection{Computation of $QP_5((4)|^3|(2)|^2|(1))$}\

\medskip
We have $$B_5((4)|^3|(2)|^2|(1)) =  B_5^0((4)|^3|(2)|^2|(1))\bigcup B_5^+((4)|^3|(2)|^2|(1)),$$ 
where $B_5^0((4)|^3|(2)|^2|(1)) = \Phi^0(B_4((4)|^3|(2)|^2|(1))))$. By \cite{su2}, $B_4((4)|^3|(2)|^2|(1)))$ is the set of $56$ monomials $w_t,\ 1 \leqslant t \leqslant 56$, which are determined as in Section \ref{s5}. We have $|B_5^0((4)|^3|(2)|^2|(1))| = 5\times 56 = 280$.

We set $\mathcal B := \left\{x_i^{63}\theta_{J_i}(x): x \in B_4^+((3)|^3|(1)|^2), 1 \leqslant i \leqslant 5\right\}.$
According to \cite{su2}, we have $|B_4^+((3)|^3|(1)|^2)| = 66$. Then, by using Proposition \ref{mdmo}, we see that $\mathcal B$ is the set of 330 monomials determined as in Section \ref{s5} and $\mathcal B \subset B_5^+((4)|^3|(2)|^2|(1))$. A direct computation shows that  $\Phi^+(B_4((4)|^3|(2)|^2|(1)))$ is the set of 1403 monomials. 

In this subsection we prove the following.

\begin{props}\label{mdc} We have $B_5^+((4)|^3|(2)|^2|(1)) \subset \Phi^+(B_4((4)|^3|(2)|^2|(1)))\cup \mathcal C$, where $\mathcal C$ is the set of 67 monomials which are determined as follows:

\medskip
\centerline{\begin{tabular}{llll}
$x_1^{7}x_2^{7}x_3^{15}x_4^{23}x_5^{56}$& $x_1^{7}x_2^{7}x_3^{15}x_4^{55}x_5^{24}$& $x_1^{7}x_2^{15}x_3^{7}x_4^{23}x_5^{56}$& $x_1^{7}x_2^{15}x_3^{7}x_4^{55}x_5^{24} $\cr  $x_1^{7}x_2^{15}x_3^{15}x_4^{16}x_5^{55}$& $x_1^{7}x_2^{15}x_3^{15}x_4^{17}x_5^{54}$& $x_1^{7}x_2^{15}x_3^{15}x_4^{23}x_5^{48}$& $x_1^{7}x_2^{15}x_3^{15}x_4^{48}x_5^{23} $\cr  $x_1^{7}x_2^{15}x_3^{15}x_4^{49}x_5^{22}$& $x_1^{7}x_2^{15}x_3^{23}x_4^{7}x_5^{56}$& $x_1^{7}x_2^{15}x_3^{23}x_4^{9}x_5^{54}$& $x_1^{7}x_2^{15}x_3^{23}x_4^{25}x_5^{38} $\cr  $x_1^{7}x_2^{15}x_3^{23}x_4^{41}x_5^{22}$& $x_1^{7}x_2^{15}x_3^{23}x_4^{57}x_5^{6}$& $x_1^{7}x_2^{15}x_3^{55}x_4^{7}x_5^{24}$& $x_1^{7}x_2^{15}x_3^{55}x_4^{9}x_5^{22} $\cr  $x_1^{7}x_2^{15}x_3^{55}x_4^{25}x_5^{6}$& $x_1^{15}x_2^{7}x_3^{7}x_4^{23}x_5^{56}$& $x_1^{15}x_2^{7}x_3^{7}x_4^{55}x_5^{24}$& $x_1^{15}x_2^{7}x_3^{15}x_4^{16}x_5^{55} $\cr  $x_1^{15}x_2^{7}x_3^{15}x_4^{17}x_5^{54}$& $x_1^{15}x_2^{7}x_3^{15}x_4^{23}x_5^{48}$& $x_1^{15}x_2^{7}x_3^{15}x_4^{48}x_5^{23}$& $x_1^{15}x_2^{7}x_3^{15}x_4^{49}x_5^{22} $\cr  $x_1^{15}x_2^{7}x_3^{23}x_4^{7}x_5^{56}$& $x_1^{15}x_2^{7}x_3^{23}x_4^{9}x_5^{54}$& $x_1^{15}x_2^{7}x_3^{23}x_4^{25}x_5^{38}$& $x_1^{15}x_2^{7}x_3^{23}x_4^{33}x_5^{30} $\cr  $x_1^{15}x_2^{7}x_3^{23}x_4^{41}x_5^{22}$& $x_1^{15}x_2^{7}x_3^{23}x_4^{57}x_5^{6}$& $x_1^{15}x_2^{7}x_3^{55}x_4^{7}x_5^{24}$& $x_1^{15}x_2^{7}x_3^{55}x_4^{9}x_5^{22} $\cr  $x_1^{15}x_2^{7}x_3^{55}x_4^{25}x_5^{6}$& $x_1^{15}x_2^{15}x_3^{7}x_4^{16}x_5^{55}$& $x_1^{15}x_2^{15}x_3^{7}x_4^{17}x_5^{54}$& $x_1^{15}x_2^{15}x_3^{7}x_4^{23}x_5^{48} $\cr  $x_1^{15}x_2^{15}x_3^{7}x_4^{48}x_5^{23}$& $x_1^{15}x_2^{15}x_3^{7}x_4^{49}x_5^{22}$& $x_1^{15}x_2^{15}x_3^{17}x_4^{6}x_5^{55}$& $x_1^{15}x_2^{15}x_3^{17}x_4^{7}x_5^{54} $\cr  $x_1^{15}x_2^{15}x_3^{17}x_4^{22}x_5^{39}$& $x_1^{15}x_2^{15}x_3^{17}x_4^{23}x_5^{38}$& $x_1^{15}x_2^{15}x_3^{17}x_4^{38}x_5^{23}$& $x_1^{15}x_2^{15}x_3^{17}x_4^{39}x_5^{22} $\cr  $x_1^{15}x_2^{15}x_3^{17}x_4^{54}x_5^{7}$& $x_1^{15}x_2^{15}x_3^{17}x_4^{55}x_5^{6}$& $x_1^{15}x_2^{15}x_3^{23}x_4^{7}x_5^{48}$& $x_1^{15}x_2^{15}x_3^{23}x_4^{17}x_5^{38} $\cr  $x_1^{15}x_2^{15}x_3^{23}x_4^{23}x_5^{32}$& $x_1^{15}x_2^{15}x_3^{23}x_4^{32}x_5^{23}$& $x_1^{15}x_2^{15}x_3^{23}x_4^{49}x_5^{6}$& $x_1^{15}x_2^{15}x_3^{49}x_4^{6}x_5^{23} $\cr  $x_1^{15}x_2^{15}x_3^{49}x_4^{7}x_5^{22}$& $x_1^{15}x_2^{15}x_3^{49}x_4^{22}x_5^{7}$& $x_1^{15}x_2^{15}x_3^{49}x_4^{23}x_5^{6}$& $x_1^{15}x_2^{19}x_3^{5}x_4^{7}x_5^{62} $\cr  $x_1^{15}x_2^{19}x_3^{7}x_4^{5}x_5^{62}$& $x_1^{15}x_2^{19}x_3^{7}x_4^{13}x_5^{54}$& $x_1^{15}x_2^{19}x_3^{7}x_4^{61}x_5^{6}$& $x_1^{15}x_2^{23}x_3^{7}x_4^{7}x_5^{56} $\cr  $x_1^{15}x_2^{23}x_3^{7}x_4^{9}x_5^{54}$& $x_1^{15}x_2^{23}x_3^{7}x_4^{25}x_5^{38}$& $x_1^{15}x_2^{23}x_3^{7}x_4^{41}x_5^{22}$& $x_1^{15}x_2^{23}x_3^{7}x_4^{57}x_5^{6} $\cr  $x_1^{15}x_2^{55}x_3^{7}x_4^{7}x_5^{24}$& $x_1^{15}x_2^{55}x_3^{7}x_4^{9}x_5^{22}$& $x_1^{15}x_2^{55}x_3^{7}x_4^{25}x_5^{6}$& \cr   
\end{tabular}}

\medskip
The elements of $\Phi^+(B_4((4)|^3|(2)|^2|(1)))\cup \mathcal C$ are explicitly determined as in Subsection \ref{s51}. 
\end{props}

We need some technical lemmas for the proof of the proposition. 

\begin{lems}\label{bd3} Let $1 \leqslant i \leqslant 5$ and the homomorphism $\theta_{J_i}: P_4 \to P_5$ defined by \eqref{ctbs} with $k=5$ and $J_i = (1,\ldots,\hat i, \ldots,5)$.
	
\medskip
{\rm i)} The monomial $x_i^{15}\theta_{J_i}\big(x_1^{7}x_2^{7}x_3^{8}x_4^{7}\big)$ is strictly inadmissible.

\smallskip
{\rm ii)} If $x$ is one of the monomials $x_1^{7}x_2^{7}x_3^{15}x_4^{16}$, $x_1^{7}x_2^{15}x_3^{7}x_4^{16}$, $x_1^{7}x_2^{15}x_3^{16}x_4^{7}$, $x_1^{7}x_2^{15}x_3^{17}x_4^{6}$,  $x_1^{15}x_2^{7}x_3^{7}x_4^{16}$, $x_1^{15}x_2^{7}x_3^{16}x_4^{7}$, $x_1^{15}x_2^{7}x_3^{17}x_4^{6}$, $x_1^{15}x_2^{19}x_3^{5}x_4^{6}$, then $x_i^{31}\theta_{J_i}(x)$ is strictly inadmissible.

\smallskip
{\rm iii)} The following monomials are strictly inadmissible:

\medskip
\centerline{\begin{tabular}{llll}
$x_1^{7}x_2^{15}x_3^{16}x_4^{15}x_5^{23}$& $x_1^{7}x_2^{15}x_3^{17}x_4^{14}x_5^{23}$& $x_1^{7}x_2^{15}x_3^{17}x_4^{15}x_5^{22}$& $x_1^{7}x_2^{15}x_3^{17}x_4^{30}x_5^{7} $\cr  $x_1^{15}x_2^{7}x_3^{16}x_4^{15}x_5^{23}$& $x_1^{15}x_2^{7}x_3^{17}x_4^{14}x_5^{23}$& $x_1^{15}x_2^{7}x_3^{17}x_4^{15}x_5^{22}$& $x_1^{15}x_2^{7}x_3^{17}x_4^{30}x_5^{7} $\cr  $x_1^{15}x_2^{15}x_3^{16}x_4^{7}x_5^{23}$& $x_1^{15}x_2^{15}x_3^{16}x_4^{23}x_5^{7}$& $x_1^{15}x_2^{15}x_3^{23}x_4^{16}x_5^{7}$& $x_1^{15}x_2^{19}x_3^{5}x_4^{14}x_5^{23} $\cr  $x_1^{15}x_2^{19}x_3^{5}x_4^{15}x_5^{22}$& $x_1^{15}x_2^{19}x_3^{5}x_4^{30}x_5^{7}$& $x_1^{15}x_2^{19}x_3^{15}x_4^{5}x_5^{22}$& $x_1^{15}x_2^{19}x_3^{15}x_4^{21}x_5^{6} $\cr  $x_1^{15}x_2^{23}x_3^{7}x_4^{15}x_5^{16}$& $x_1^{15}x_2^{23}x_3^{15}x_4^{7}x_5^{16}$& $x_1^{15}x_2^{23}x_3^{15}x_4^{16}x_5^{7}$& $x_1^{15}x_2^{23}x_3^{15}x_4^{17}x_5^{6}.$\cr  
\end{tabular}} 
\end{lems}
\begin{proof} The proof of Parts i) and ii) are easy. We prove Part iii) for 6 monomials
	
\medskip
\centerline{\begin{tabular}{lll}
$W_{1} =  x_1^{7}x_2^{15}x_3^{16}x_4^{15}x_5^{23}$& $W_{2} =  x_1^{15}x_2^{7}x_3^{17}x_4^{15}x_5^{22}$& $W_{3} =  x_1^{15}x_2^{15}x_3^{23}x_4^{16}x_5^{7} $\cr  $W_{4} =  x_1^{15}x_2^{19}x_3^{15}x_4^{21}x_5^{6}$& $W_{5} =  x_1^{15}x_2^{23}x_3^{15}x_4^{7}x_5^{16}$& $W_{6} =  x_1^{15}x_2^{23}x_3^{15}x_4^{17}x_5^{6}. $\cr
\end{tabular}} 

\medskip
Note that the monomials in Part iii) are of weight vector $(4)|^3|(2)|^2$. By using the Cartan formula we get
\begin{align*}
W_1 &= x_1^{5}x_2^{15}x_3^{7}x_4^{19}x_5^{30} + x_1^{5}x_2^{15}x_3^{14}x_4^{19}x_5^{23} + x_1^{5}x_2^{15}x_3^{18}x_4^{15}x_5^{23} + x_1^{5}x_2^{19}x_3^{7}x_4^{15}x_5^{30}\\ 
&\quad + x_1^{5}x_2^{19}x_3^{14}x_4^{15}x_5^{23} + x_1^{7}x_2^{9}x_3^{7}x_4^{23}x_5^{30} + x_1^{7}x_2^{9}x_3^{14}x_4^{23}x_5^{23} + x_1^{7}x_2^{9}x_3^{22}x_4^{15}x_5^{23}\\ 
&\quad + x_1^{7}x_2^{15}x_3^{7}x_4^{17}x_5^{30} + x_1^{7}x_2^{15}x_3^{14}x_4^{17}x_5^{23} +  Sq^1\big(x_1^{3}x_2^{15}x_3^{13}x_4^{15}x_5^{29} + x_1^{7}x_2^{15}x_3^{9}x_4^{15}x_5^{29}\\ 
&\quad + x_1^{7}x_2^{15}x_3^{11}x_4^{15}x_5^{27} + x_1^{7}x_2^{15}x_3^{13}x_4^{15}x_5^{25}\big) +  Sq^2\big(x_1^{3}x_2^{15}x_3^{11}x_4^{15}x_5^{30}\\ 
&\quad + x_1^{3}x_2^{15}x_3^{14}x_4^{15}x_5^{27} + x_1^{7}x_2^{15}x_3^{7}x_4^{15}x_5^{30} + x_1^{7}x_2^{15}x_3^{10}x_4^{15}x_5^{27} + x_1^{7}x_2^{15}x_3^{11}x_4^{15}x_5^{26}\\ 
&\quad + x_1^{7}x_2^{15}x_3^{14}x_4^{15}x_5^{23}\big) +  Sq^4\big(x_1^{5}x_2^{15}x_3^{7}x_4^{15}x_5^{30} + x_1^{5}x_2^{15}x_3^{14}x_4^{15}x_5^{23}\big)\\ 
&\quad +  Sq^8\big(x_1^{7}x_2^{9}x_3^{7}x_4^{15}x_5^{30} + x_1^{7}x_2^{9}x_3^{14}x_4^{15}x_5^{23}\big) \  \mbox{ mod}(P_5^-((4)|^3|(2)|^2)),\\
W_2 &= x_1^{8}x_2^{7}x_3^{15}x_4^{23}x_5^{23} + x_1^{8}x_2^{7}x_3^{23}x_4^{15}x_5^{23} + x_1^{9}x_2^{7}x_3^{15}x_4^{23}x_5^{22} + x_1^{9}x_2^{7}x_3^{23}x_4^{15}x_5^{22}\\ 
&\quad + x_1^{11}x_2^{4}x_3^{15}x_4^{23}x_5^{23} + x_1^{11}x_2^{4}x_3^{23}x_4^{15}x_5^{23} + x_1^{11}x_2^{5}x_3^{15}x_4^{23}x_5^{22} + x_1^{11}x_2^{5}x_3^{23}x_4^{15}x_5^{22}\\ 
&\quad + x_1^{15}x_2^{4}x_3^{15}x_4^{19}x_5^{23} + x_1^{15}x_2^{4}x_3^{19}x_4^{15}x_5^{23} + x_1^{15}x_2^{5}x_3^{15}x_4^{19}x_5^{22} + x_1^{15}x_2^{5}x_3^{19}x_4^{15}x_5^{22}\\ 
&\quad + x_1^{15}x_2^{7}x_3^{15}x_4^{16}x_5^{23} + x_1^{15}x_2^{7}x_3^{15}x_4^{17}x_5^{22} + x_1^{15}x_2^{7}x_3^{16}x_4^{15}x_5^{23}\big)\\ 
&\quad +  Sq^1\big(x_1^{15}x_2^{3}x_3^{15}x_4^{15}x_5^{27} + x_1^{15}x_2^{7}x_3^{15}x_4^{15}x_5^{23}\big) +  Sq^2\big(x_1^{15}x_2^{3}x_3^{15}x_4^{15}x_5^{26}\\ 
&\quad + x_1^{15}x_2^{7}x_3^{15}x_4^{15}x_5^{22}\big) +  Sq^4\big(x_1^{15}x_2^{4}x_3^{15}x_4^{15}x_5^{23} + x_1^{15}x_2^{5}x_3^{15}x_4^{15}x_5^{22}\big)\\ 
&\quad +  Sq^8\big(x_1^{8}x_2^{7}x_3^{15}x_4^{15}x_5^{23} + x_1^{9}x_2^{7}x_3^{15}x_4^{15}x_5^{22} + x_1^{11}x_2^{4}x_3^{15}x_4^{15}x_5^{23}\\ 
&\quad + x_1^{11}x_2^{5}x_3^{15}x_4^{15}x_5^{22}\big) \  \mbox{ mod}(P_5^-((4)|^3|(2)|^2)),\\  
W_3 &= x_1^{9}x_2^{15}x_3^{23}x_4^{7}x_5^{22} + x_1^{9}x_2^{15}x_3^{23}x_4^{22}x_5^{7} + x_1^{9}x_2^{23}x_3^{23}x_4^{7}x_5^{14} + x_1^{9}x_2^{23}x_3^{23}x_4^{14}x_5^{7}\\ 
&\quad + x_1^{11}x_2^{7}x_3^{7}x_4^{21}x_5^{30} + x_1^{11}x_2^{7}x_3^{7}x_4^{30}x_5^{21} + x_1^{11}x_2^{7}x_3^{21}x_4^{7}x_5^{30} + x_1^{11}x_2^{7}x_3^{21}x_4^{30}x_5^{7}\\ 
&\quad + x_1^{11}x_2^{7}x_3^{29}x_4^{7}x_5^{22} + x_1^{11}x_2^{7}x_3^{29}x_4^{22}x_5^{7} + x_1^{11}x_2^{13}x_3^{23}x_4^{7}x_5^{22} + x_1^{11}x_2^{13}x_3^{23}x_4^{22}x_5^{7}\\ 
&\quad + x_1^{11}x_2^{15}x_3^{7}x_4^{21}x_5^{22} + x_1^{11}x_2^{15}x_3^{7}x_4^{22}x_5^{21} + x_1^{11}x_2^{21}x_3^{7}x_4^{7}x_5^{30} + x_1^{11}x_2^{21}x_3^{7}x_4^{30}x_5^{7}\\ 
&\quad + x_1^{11}x_2^{21}x_3^{23}x_4^{7}x_5^{14} + x_1^{11}x_2^{21}x_3^{23}x_4^{14}x_5^{7} + x_1^{11}x_2^{23}x_3^{7}x_4^{7}x_5^{28} + x_1^{11}x_2^{23}x_3^{7}x_4^{13}x_5^{22}\\ 
&\quad + x_1^{11}x_2^{23}x_3^{7}x_4^{22}x_5^{13} + x_1^{11}x_2^{23}x_3^{7}x_4^{28}x_5^{7} + x_1^{11}x_2^{23}x_3^{13}x_4^{7}x_5^{22} + x_1^{11}x_2^{23}x_3^{13}x_4^{22}x_5^{7}\\ 
&\quad + x_1^{11}x_2^{23}x_3^{21}x_4^{7}x_5^{14} + x_1^{11}x_2^{23}x_3^{21}x_4^{14}x_5^{7} + x_1^{15}x_2^{7}x_3^{7}x_4^{17}x_5^{30} + x_1^{15}x_2^{7}x_3^{7}x_4^{30}x_5^{17}\\ 
&\quad + x_1^{15}x_2^{7}x_3^{17}x_4^{7}x_5^{30} + x_1^{15}x_2^{7}x_3^{17}x_4^{30}x_5^{7} + x_1^{15}x_2^{7}x_3^{29}x_4^{7}x_5^{18} + x_1^{15}x_2^{7}x_3^{29}x_4^{18}x_5^{7}\\ 
&\quad + x_1^{15}x_2^{11}x_3^{7}x_4^{22}x_5^{21} + x_1^{15}x_2^{11}x_3^{21}x_4^{7}x_5^{22} + x_1^{15}x_2^{11}x_3^{21}x_4^{22}x_5^{7} + x_1^{15}x_2^{13}x_3^{23}x_4^{7}x_5^{18}\\ 
&\quad + x_1^{15}x_2^{13}x_3^{23}x_4^{18}x_5^{7} + x_1^{15}x_2^{15}x_3^{7}x_4^{17}x_5^{22} + x_1^{15}x_2^{15}x_3^{7}x_4^{22}x_5^{17} + x_1^{15}x_2^{15}x_3^{17}x_4^{7}x_5^{22}\\ 
&\quad + x_1^{15}x_2^{15}x_3^{17}x_4^{22}x_5^{7} + x_1^{15}x_2^{15}x_3^{21}x_4^{7}x_5^{18} + x_1^{15}x_2^{15}x_3^{21}x_4^{18}x_5^{7} + x_1^{15}x_2^{15}x_3^{23}x_4^{7}x_5^{16}\big)\\ 
&\quad +  Sq^1\big(x_1^{15}x_2^{7}x_3^{27}x_4^{13}x_5^{13} + x_1^{15}x_2^{11}x_3^{23}x_4^{13}x_5^{13} + x_1^{15}x_2^{15}x_3^{19}x_4^{13}x_5^{13}\\ 
&\quad + x_1^{15}x_2^{15}x_3^{23}x_4^{9}x_5^{13} + x_1^{15}x_2^{15}x_3^{23}x_4^{11}x_5^{11} + x_1^{15}x_2^{15}x_3^{23}x_4^{13}x_5^{9} + x_1^{15}x_2^{17}x_3^{7}x_4^{7}x_5^{29}\\ 
&\quad + x_1^{15}x_2^{17}x_3^{7}x_4^{29}x_5^{7} + x_1^{15}x_2^{19}x_3^{7}x_4^{7}x_5^{27} + x_1^{15}x_2^{19}x_3^{7}x_4^{13}x_5^{21} + x_1^{15}x_2^{19}x_3^{7}x_4^{27}x_5^{7}\big)\\ 
&\quad +  Sq^2\big(x_1^{15}x_2^{7}x_3^{11}x_4^{11}x_5^{30} + x_1^{15}x_2^{7}x_3^{11}x_4^{30}x_5^{11} + x_1^{15}x_2^{7}x_3^{27}x_4^{11}x_5^{14}\\ 
&\quad + x_1^{15}x_2^{7}x_3^{27}x_4^{14}x_5^{11} + x_1^{15}x_2^{11}x_3^{7}x_4^{11}x_5^{30} + x_1^{15}x_2^{11}x_3^{7}x_4^{30}x_5^{11} + x_1^{15}x_2^{11}x_3^{11}x_4^{7}x_5^{30}\\ 
&\quad + x_1^{15}x_2^{11}x_3^{11}x_4^{30}x_5^{7} + x_1^{15}x_2^{11}x_3^{23}x_4^{11}x_5^{14} + x_1^{15}x_2^{11}x_3^{23}x_4^{14}x_5^{11} + x_1^{15}x_2^{11}x_3^{27}x_4^{7}x_5^{14}\\ 
&\quad + x_1^{15}x_2^{11}x_3^{27}x_4^{14}x_5^{7} + x_1^{15}x_2^{15}x_3^{7}x_4^{11}x_5^{26} + x_1^{15}x_2^{15}x_3^{7}x_4^{26}x_5^{11} + x_1^{15}x_2^{15}x_3^{11}x_4^{7}x_5^{26}\\ 
&\quad + x_1^{15}x_2^{15}x_3^{11}x_4^{11}x_5^{22} + x_1^{15}x_2^{15}x_3^{11}x_4^{22}x_5^{11} + x_1^{15}x_2^{15}x_3^{11}x_4^{26}x_5^{7} + x_1^{15}x_2^{15}x_3^{19}x_4^{11}x_5^{14}\\ 
&\quad + x_1^{15}x_2^{15}x_3^{19}x_4^{14}x_5^{11} + x_1^{15}x_2^{15}x_3^{23}x_4^{7}x_5^{14} + x_1^{15}x_2^{15}x_3^{23}x_4^{10}x_5^{11} + x_1^{15}x_2^{15}x_3^{23}x_4^{11}x_5^{10}\\ 
&\quad + x_1^{15}x_2^{15}x_3^{23}x_4^{14}x_5^{7} + x_1^{15}x_2^{18}x_3^{7}x_4^{7}x_5^{27} + x_1^{15}x_2^{18}x_3^{7}x_4^{27}x_5^{7}\big) +  Sq^4\big(x_1^{15}x_2^{7}x_3^{7}x_4^{13}x_5^{30}\\ 
&\quad + x_1^{15}x_2^{7}x_3^{7}x_4^{30}x_5^{13} + x_1^{15}x_2^{7}x_3^{13}x_4^{7}x_5^{30} + x_1^{15}x_2^{7}x_3^{13}x_4^{30}x_5^{7} + x_1^{15}x_2^{7}x_3^{29}x_4^{7}x_5^{14}\\ 
&\quad + x_1^{15}x_2^{7}x_3^{29}x_4^{14}x_5^{7} + x_1^{15}x_2^{13}x_3^{7}x_4^{7}x_5^{30} + x_1^{15}x_2^{13}x_3^{7}x_4^{30}x_5^{7} + x_1^{15}x_2^{13}x_3^{23}x_4^{7}x_5^{14}\\ 
&\quad + x_1^{15}x_2^{13}x_3^{23}x_4^{14}x_5^{7} + x_1^{15}x_2^{15}x_3^{7}x_4^{7}x_5^{28} + x_1^{15}x_2^{15}x_3^{7}x_4^{13}x_5^{22} + x_1^{15}x_2^{15}x_3^{7}x_4^{22}x_5^{13}\\ 
&\quad + x_1^{15}x_2^{15}x_3^{7}x_4^{28}x_5^{7} + x_1^{15}x_2^{15}x_3^{13}x_4^{7}x_5^{22} + x_1^{15}x_2^{15}x_3^{13}x_4^{22}x_5^{7} + x_1^{15}x_2^{15}x_3^{21}x_4^{7}x_5^{14}\\ 
&\quad + x_1^{15}x_2^{15}x_3^{21}x_4^{14}x_5^{7}\big) +  Sq^8\big(x_1^{9}x_2^{15}x_3^{23}x_4^{7}x_5^{14} + x_1^{9}x_2^{15}x_3^{23}x_4^{14}x_5^{7}\\ 
&\quad + x_1^{11}x_2^{7}x_3^{7}x_4^{13}x_5^{30} + x_1^{11}x_2^{7}x_3^{7}x_4^{30}x_5^{13} + x_1^{11}x_2^{7}x_3^{13}x_4^{7}x_5^{30} + x_1^{11}x_2^{7}x_3^{13}x_4^{30}x_5^{7}\\ 
&\quad + x_1^{11}x_2^{7}x_3^{29}x_4^{7}x_5^{14} + x_1^{11}x_2^{7}x_3^{29}x_4^{14}x_5^{7} + x_1^{11}x_2^{13}x_3^{7}x_4^{7}x_5^{30} + x_1^{11}x_2^{13}x_3^{7}x_4^{30}x_5^{7}\\ 
&\quad + x_1^{11}x_2^{13}x_3^{23}x_4^{7}x_5^{14} + x_1^{11}x_2^{13}x_3^{23}x_4^{14}x_5^{7} + x_1^{11}x_2^{15}x_3^{7}x_4^{7}x_5^{28} + x_1^{11}x_2^{15}x_3^{7}x_4^{13}x_5^{22}\\ 
&\quad + x_1^{11}x_2^{15}x_3^{7}x_4^{22}x_5^{13} + x_1^{11}x_2^{15}x_3^{7}x_4^{28}x_5^{7} + x_1^{11}x_2^{15}x_3^{13}x_4^{7}x_5^{22} + x_1^{11}x_2^{15}x_3^{13}x_4^{22}x_5^{7}\\ 
&\quad + x_1^{11}x_2^{15}x_3^{21}x_4^{7}x_5^{14} + x_1^{11}x_2^{15}x_3^{21}x_4^{14}x_5^{7} + x_1^{23}x_2^{11}x_3^{7}x_4^{14}x_5^{13} + x_1^{23}x_2^{11}x_3^{13}x_4^{7}x_5^{14}\\ 
&\quad + x_1^{23}x_2^{11}x_3^{13}x_4^{14}x_5^{7}\big) +  Sq^{16}\big(x_1^{15}x_2^{11}x_3^{7}x_4^{14}x_5^{13} + x_1^{15}x_2^{11}x_3^{13}x_4^{7}x_5^{14}\\ 
&\quad + x_1^{15}x_2^{11}x_3^{13}x_4^{14}x_5^{7}\big) \mbox{ mod}(P_5^-((4)|^3|(2)|^2)),\\  
W_4 &= x_1^{8}x_2^{15}x_3^{23}x_4^{23}x_5^{7} + x_1^{8}x_2^{23}x_3^{15}x_4^{23}x_5^{7} + x_1^{9}x_2^{15}x_3^{23}x_4^{22}x_5^{7} + x_1^{9}x_2^{15}x_3^{23}x_4^{23}x_5^{6}\\ 
&\quad + x_1^{9}x_2^{23}x_3^{15}x_4^{22}x_5^{7} + x_1^{9}x_2^{23}x_3^{15}x_4^{23}x_5^{6} + x_1^{11}x_2^{15}x_3^{23}x_4^{21}x_5^{6} + x_1^{11}x_2^{23}x_3^{15}x_4^{21}x_5^{6}\\ 
&\quad + x_1^{15}x_2^{15}x_3^{16}x_4^{23}x_5^{7} + x_1^{15}x_2^{15}x_3^{17}x_4^{22}x_5^{7} + x_1^{15}x_2^{15}x_3^{17}x_4^{23}x_5^{6} + x_1^{15}x_2^{15}x_3^{19}x_4^{21}x_5^{6}\\ 
&\quad + x_1^{15}x_2^{16}x_3^{15}x_4^{23}x_5^{7} + x_1^{15}x_2^{17}x_3^{15}x_4^{22}x_5^{7} + x_1^{15}x_2^{17}x_3^{15}x_4^{23}x_5^{6}\\ 
&\quad + Sq^1\big(x_1^{15}x_2^{15}x_3^{15}x_4^{21}x_5^{9} + x_1^{15}x_2^{15}x_3^{15}x_4^{23}x_5^{7}\big)\\ 
&\quad +  Sq^2\big(x_1^{15}x_2^{15}x_3^{15}x_4^{22}x_5^{7} + x_1^{15}x_2^{15}x_3^{15}x_4^{23}x_5^{6}\big) +  Sq^4\big(x_1^{15}x_2^{15}x_3^{15}x_4^{21}x_5^{6}\big)\\ 
&\quad +  Sq^8\big(x_1^{8}x_2^{15}x_3^{15}x_4^{23}x_5^{7} + x_1^{9}x_2^{15}x_3^{15}x_4^{22}x_5^{7} + x_1^{9}x_2^{15}x_3^{15}x_4^{23}x_5^{6}\\ 
&\quad + x_1^{11}x_2^{15}x_3^{15}x_4^{21}x_5^{6}\big) \ \mbox{ mod}(P_5^-((4)|^3|(2)|^2)),\\  
W_5 &= x_1^{8}x_2^{23}x_3^{15}x_4^{7}x_5^{23} + x_1^{8}x_2^{23}x_3^{23}x_4^{7}x_5^{15} + x_1^{9}x_2^{22}x_3^{15}x_4^{7}x_5^{23} + x_1^{9}x_2^{22}x_3^{23}x_4^{7}x_5^{15}\\ 
&\quad + x_1^{9}x_2^{23}x_3^{15}x_4^{6}x_5^{23} + x_1^{9}x_2^{23}x_3^{23}x_4^{6}x_5^{15} + x_1^{11}x_2^{21}x_3^{15}x_4^{6}x_5^{23} + x_1^{11}x_2^{21}x_3^{23}x_4^{6}x_5^{15}\\ 
&\quad + x_1^{15}x_2^{15}x_3^{16}x_4^{7}x_5^{23} + x_1^{15}x_2^{15}x_3^{17}x_4^{6}x_5^{23} + x_1^{15}x_2^{21}x_3^{15}x_4^{6}x_5^{19} + x_1^{15}x_2^{21}x_3^{19}x_4^{6}x_5^{15}\\ 
&\quad + x_1^{15}x_2^{22}x_3^{15}x_4^{7}x_5^{17} + x_1^{15}x_2^{22}x_3^{17}x_4^{7}x_5^{15} + x_1^{15}x_2^{23}x_3^{8}x_4^{7}x_5^{23} + x_1^{15}x_2^{23}x_3^{9}x_4^{6}x_5^{23}\\ 
&\quad + x_1^{15}x_2^{23}x_3^{15}x_4^{6}x_5^{17} + Sq^1\big(x_1^{15}x_2^{21}x_3^{15}x_4^{9}x_5^{15} + x_1^{15}x_2^{23}x_3^{15}x_4^{7}x_5^{15}\big)\\ 
&\quad +  Sq^2\big(x_1^{15}x_2^{22}x_3^{15}x_4^{7}x_5^{15} + x_1^{15}x_2^{23}x_3^{15}x_4^{6}x_5^{15}\big) +  Sq^4\big(x_1^{15}x_2^{21}x_3^{15}x_4^{6}x_5^{15}\big)\\ 
&\quad +  Sq^8\big(x_1^{8}x_2^{23}x_3^{15}x_4^{7}x_5^{15} + x_1^{9}x_2^{22}x_3^{15}x_4^{7}x_5^{15} + x_1^{9}x_2^{23}x_3^{15}x_4^{6}x_5^{15}\\ 
&\quad + x_1^{11}x_2^{21}x_3^{15}x_4^{6}x_5^{15} + x_1^{23}x_2^{15}x_3^{8}x_4^{7}x_5^{15} + x_1^{23}x_2^{15}x_3^{9}x_4^{6}x_5^{15}\big)\\ 
&\quad +  Sq^{16}\big(x_1^{15}x_2^{15}x_3^{8}x_4^{7}x_5^{15} + x_1^{15}x_2^{15}x_3^{9}x_4^{6}x_5^{15}\big) \ \mbox{ mod}(P_5^-((4)|^3|(2)|^2)),\\
W_6 &= x_1^{8}x_2^{23}x_3^{15}x_4^{23}x_5^{7} + x_1^{8}x_2^{23}x_3^{23}x_4^{15}x_5^{7} + x_1^{9}x_2^{22}x_3^{15}x_4^{23}x_5^{7} + x_1^{9}x_2^{22}x_3^{23}x_4^{15}x_5^{7}\\ 
&\quad + x_1^{9}x_2^{23}x_3^{15}x_4^{23}x_5^{6} + x_1^{9}x_2^{23}x_3^{23}x_4^{15}x_5^{6} + x_1^{11}x_2^{21}x_3^{15}x_4^{23}x_5^{6} + x_1^{11}x_2^{21}x_3^{23}x_4^{15}x_5^{6}\\ 
&\quad + x_1^{15}x_2^{15}x_3^{16}x_4^{23}x_5^{7} + x_1^{15}x_2^{15}x_3^{17}x_4^{23}x_5^{6} + x_1^{15}x_2^{21}x_3^{15}x_4^{19}x_5^{6} + x_1^{15}x_2^{21}x_3^{19}x_4^{15}x_5^{6}\\ 
&\quad + x_1^{15}x_2^{22}x_3^{15}x_4^{17}x_5^{7} + x_1^{15}x_2^{22}x_3^{17}x_4^{15}x_5^{7} + x_1^{15}x_2^{23}x_3^{8}x_4^{23}x_5^{7} + x_1^{15}x_2^{23}x_3^{9}x_4^{23}x_5^{6}\\ 
&\quad + x_1^{15}x_2^{23}x_3^{15}x_4^{16}x_5^{7} +  Sq^1\big(x_1^{15}x_2^{21}x_3^{15}x_4^{15}x_5^{9} + x_1^{15}x_2^{23}x_3^{15}x_4^{15}x_5^{7}\big)\\ 
&\quad +  Sq^2\big(x_1^{15}x_2^{22}x_3^{15}x_4^{15}x_5^{7} + x_1^{15}x_2^{23}x_3^{15}x_4^{15}x_5^{6}\big) +  Sq^4\big(x_1^{15}x_2^{21}x_3^{15}x_4^{15}x_5^{6}\big)\\ 
&\quad +  Sq^8\big(x_1^{8}x_2^{23}x_3^{15}x_4^{15}x_5^{7} + x_1^{9}x_2^{22}x_3^{15}x_4^{15}x_5^{7} + x_1^{9}x_2^{23}x_3^{15}x_4^{15}x_5^{6}\\ 
&\quad + x_1^{11}x_2^{21}x_3^{15}x_4^{15}x_5^{6} + x_1^{23}x_2^{15}x_3^{8}x_4^{15}x_5^{7} + x_1^{23}x_2^{15}x_3^{9}x_4^{15}x_5^{6}\big)\\ 
&\quad +  Sq^{16}\big(x_1^{15}x_2^{15}x_3^{8}x_4^{15}x_5^{7} + x_1^{15}x_2^{15}x_3^{9}x_4^{15}x_5^{6}\big) \ \mbox{ mod}(P_5^-((4)|^3|(2)|^2)).
\end{align*}
The above equalities show that the monomials $W_j,\, 1 \leqslant j \leqslant 6$, are strictly inadmissible. The proof for other monomials are carried out by a similar computation.
\end{proof}

The following is proved in \cite{p24}.
\begin{lems}\label{bd4} The following monomials are strictly inadmissible: 
	
\medskip
\centerline{\begin{tabular}{lll}
$ x_1^{3}x_2^{4}x_3x_4^{9}x_5^{7}$& $ x_1^{3}x_2^{4}x_3^{9}x_4x_5^{7}$& $ x_1^{3}x_2^{4}x_3^{9}x_4^{7}x_5 $\cr  $ x_1^{3}x_2^{12}x_3x_4x_5^{7}$& $ x_1^{3}x_2^{12}x_3x_4^{7}x_5$& $ x_1^{3}x_2^{12}x_3^{7}x_4x_5 $.\cr 
\end{tabular}} 
\end{lems}

\begin{lems}\label{bd5} The following monomials are strictly inadmissible:
	
\medskip
\centerline{\begin{tabular}{llll}
$ x_1^{3}x_2^{7}x_3^{11}x_4^{16}x_5^{15}$& $ x_1^{3}x_2^{7}x_3^{11}x_4^{19}x_5^{12}$& $ x_1^{3}x_2^{7}x_3^{24}x_4^{3}x_5^{15}$& $ x_1^{3}x_2^{15}x_3^{16}x_4^{7}x_5^{11} $\cr  $ x_1^{7}x_2^{3}x_3^{8}x_4^{19}x_5^{15}$& $ x_1^{7}x_2^{3}x_3^{11}x_4^{16}x_5^{15}$& $ x_1^{7}x_2^{3}x_3^{11}x_4^{19}x_5^{12}$& $ x_1^{7}x_2^{3}x_3^{24}x_4^{3}x_5^{15} $\cr  $ x_1^{7}x_2^{11}x_3^{3}x_4^{16}x_5^{15}$& $ x_1^{7}x_2^{11}x_3^{3}x_4^{19}x_5^{12}$& $ x_1^{7}x_2^{11}x_3^{16}x_4^{3}x_5^{15}$& $ x_1^{7}x_2^{11}x_3^{16}x_4^{7}x_5^{11} $\cr  $ x_1^{7}x_2^{11}x_3^{16}x_4^{15}x_5^{3}$& $ x_1^{7}x_2^{11}x_3^{17}x_4^{2}x_5^{15}$& $ x_1^{7}x_2^{11}x_3^{17}x_4^{14}x_5^{3}$& $ x_1^{7}x_2^{11}x_3^{19}x_4^{12}x_5^{3} $\cr  $ x_1^{15}x_2^{3}x_3^{16}x_4^{7}x_5^{11}$& $ x_1^{15}x_2^{16}x_3^{3}x_4^{7}x_5^{11}$& $ x_1^{15}x_2^{16}x_3^{7}x_4^{3}x_5^{11}$& $ x_1^{15}x_2^{16}x_3^{7}x_4^{11}x_5^{3} $\cr  $ x_1^{15}x_2^{17}x_3^{2}x_4^{7}x_5^{11}$& $ x_1^{15}x_2^{17}x_3^{6}x_4^{3}x_5^{11}$& $ x_1^{15}x_2^{17}x_3^{6}x_4^{11}x_5^{3}$& $ x_1^{15}x_2^{17}x_3^{7}x_4^{2}x_5^{11} $\cr  $ x_1^{15}x_2^{17}x_3^{7}x_4^{10}x_5^{3}$.& & &\cr   
\end{tabular}} 
\end{lems}

\begin{proof} We prove the lemma for 8 monomials
	
\medskip
\centerline{\begin{tabular}{llll}
$Z_{1} =  x_1^{3}x_2^{7}x_3^{24}x_4^{3}x_5^{15}$& $Z_{2} =  x_1^{7}x_2^{3}x_3^{8}x_4^{19}x_5^{15}$& $Z_{3} =  x_1^{7}x_2^{3}x_3^{11}x_4^{19}x_5^{12} $\cr  $Z_{4} =  x_1^{7}x_2^{11}x_3^{16}x_4^{7}x_5^{11}$& $Z_{5} =  x_1^{7}x_2^{11}x_3^{17}x_4^{14}x_5^{3}$& $Z_{6} =  x_1^{15}x_2^{3}x_3^{16}x_4^{7}x_5^{11} $\cr  $Z_{7} =  x_1^{15}x_2^{17}x_3^{6}x_4^{3}x_5^{11}$& $Z_{8} =  x_1^{15}x_2^{17}x_3^{7}x_4^{10}x_5^{3}$.& \cr
\end{tabular}} 

\medskip
The others can be proved by a similar computation. By a direct computation using the Cartan formula we get
\begin{align*}
Z_1 &= x_1^{2}x_2^{7}x_3^{15}x_4^{3}x_5^{25} + x_1^{2}x_2^{7}x_3^{15}x_4^{9}x_5^{19} + x_1^{2}x_2^{7}x_3^{19}x_4^{9}x_5^{15} + x_1^{2}x_2^{7}x_3^{25}x_4^{3}x_5^{15}\\
&\quad + x_1^{2}x_2^{9}x_3^{15}x_4^{3}x_5^{23} + x_1^{2}x_2^{9}x_3^{23}x_4^{3}x_5^{15} + x_1^{3}x_2^{4}x_3^{15}x_4^{3}x_5^{27} + x_1^{3}x_2^{4}x_3^{27}x_4^{3}x_5^{15}\\
&\quad + x_1^{3}x_2^{7}x_3^{15}x_4^{3}x_5^{24} + x_1^{3}x_2^{7}x_3^{15}x_4^{8}x_5^{19} + x_1^{3}x_2^{7}x_3^{19}x_4^{8}x_5^{15}\big) +  Sq^1\big(x_1^{3}x_2^{7}x_3^{15}x_4^{3}x_5^{23}\\
&\quad + x_1^{3}x_2^{7}x_3^{23}x_4^{3}x_5^{15}\big) +  Sq^2\big(x_1^{2}x_2^{7}x_3^{15}x_4^{3}x_5^{23} + x_1^{2}x_2^{7}x_3^{23}x_4^{3}x_5^{15}\big)\\
&\quad +  Sq^4\big(x_1^{2}x_2^{11}x_3^{15}x_4^{5}x_5^{15} + x_1^{3}x_2^{4}x_3^{15}x_4^{3}x_5^{23} + x_1^{3}x_2^{4}x_3^{23}x_4^{3}x_5^{15} + x_1^{3}x_2^{11}x_3^{15}x_4^{4}x_5^{15}\big)\\
&\quad +  Sq^8\big(x_1^{2}x_2^{7}x_3^{15}x_4^{5}x_5^{15} + x_1^{3}x_2^{7}x_3^{15}x_4^{4}x_5^{15}\big) \ \mbox{ mod}(P_5^-((4)|^2|(2)|^2|(1)),\\ 
Z_2  & = x_1^{4}x_2^{3}x_3^{3}x_4^{15}x_5^{27} + x_1^{4}x_2^{3}x_3^{3}x_4^{27}x_5^{15} + x_1^{5}x_2^{2}x_3^{3}x_4^{15}x_5^{27} + x_1^{5}x_2^{2}x_3^{3}x_4^{27}x_5^{15}\\
&\quad + x_1^{7}x_2^{2}x_3^{3}x_4^{15}x_5^{25} + x_1^{7}x_2^{2}x_3^{3}x_4^{25}x_5^{15} + x_1^{7}x_2^{2}x_3^{9}x_4^{15}x_5^{19} + x_1^{7}x_2^{2}x_3^{9}x_4^{19}x_5^{15}\\
&\quad + x_1^{7}x_2^{3}x_3^{3}x_4^{15}x_5^{24} + x_1^{7}x_2^{3}x_3^{3}x_4^{24}x_5^{15} + x_1^{7}x_2^{3}x_3^{8}x_4^{15}x_5^{19} +  Sq^1\big(x_1^{7}x_2^{3}x_3^{3}x_4^{15}x_5^{23}\\
&\quad + x_1^{7}x_2^{3}x_3^{3}x_4^{23}x_5^{15}\big) +  Sq^2\big(x_1^{7}x_2^{2}x_3^{3}x_4^{15}x_5^{23} + x_1^{7}x_2^{2}x_3^{3}x_4^{23}x_5^{15}\big)\\
&\quad +  Sq^4\big(x_1^{4}x_2^{3}x_3^{3}x_4^{15}x_5^{23} + x_1^{4}x_2^{3}x_3^{3}x_4^{23}x_5^{15} + x_1^{5}x_2^{2}x_3^{3}x_4^{15}x_5^{23}\\
&\quad + x_1^{5}x_2^{2}x_3^{3}x_4^{23}x_5^{15} + x_1^{11}x_2^{2}x_3^{5}x_4^{15}x_5^{15} + x_1^{11}x_2^{3}x_3^{4}x_4^{15}x_5^{15}\big)\\
&\quad +  Sq^8\big(x_1^{7}x_2^{2}x_3^{5}x_4^{15}x_5^{15} + x_1^{7}x_2^{3}x_3^{4}x_4^{15}x_5^{15}\big) \ \mbox{ mod}(P_5^-((4)|^2|(2)|^2|(1)),\\
Z_3 &= x_1^{4}x_2^{3}x_3^{7}x_4^{27}x_5^{11} + x_1^{4}x_2^{3}x_3^{11}x_4^{15}x_5^{19} + x_1^{4}x_2^{3}x_3^{11}x_4^{23}x_5^{11} + x_1^{5}x_2^{2}x_3^{7}x_4^{27}x_5^{11}\\
&\quad + x_1^{5}x_2^{2}x_3^{11}x_4^{15}x_5^{19} + x_1^{5}x_2^{2}x_3^{11}x_4^{23}x_5^{11} + x_1^{7}x_2^{2}x_3^{7}x_4^{25}x_5^{11} + x_1^{7}x_2^{2}x_3^{9}x_4^{15}x_5^{19}\\
&\quad + x_1^{7}x_2^{2}x_3^{9}x_4^{23}x_5^{11} + x_1^{7}x_2^{2}x_3^{11}x_4^{15}x_5^{17} + x_1^{7}x_2^{2}x_3^{11}x_4^{19}x_5^{13} + x_1^{7}x_2^{3}x_3^{7}x_4^{24}x_5^{11}\\
&\quad + x_1^{7}x_2^{3}x_3^{8}x_4^{15}x_5^{19} + x_1^{7}x_2^{3}x_3^{8}x_4^{23}x_5^{11} + x_1^{7}x_2^{3}x_3^{11}x_4^{15}x_5^{16} +  Sq^1\big(x_1^{7}x_2^{3}x_3^{7}x_4^{15}x_5^{19}\\
&\quad + x_1^{7}x_2^{3}x_3^{7}x_4^{23}x_5^{11}\big) +  Sq^2\big(x_1^{7}x_2^{2}x_3^{7}x_4^{15}x_5^{19} + x_1^{7}x_2^{2}x_3^{7}x_4^{23}x_5^{11}\big)\\
&\quad +  Sq^4\big(x_1^{4}x_2^{3}x_3^{7}x_4^{15}x_5^{19} + x_1^{4}x_2^{3}x_3^{7}x_4^{23}x_5^{11} + x_1^{5}x_2^{2}x_3^{7}x_4^{15}x_5^{19}\\
&\quad + x_1^{5}x_2^{2}x_3^{7}x_4^{23}x_5^{11} + x_1^{11}x_2^{2}x_3^{7}x_4^{15}x_5^{13} + x_1^{11}x_2^{3}x_3^{7}x_4^{15}x_5^{12}\big)\\
&\quad +  Sq^8\big(x_1^{7}x_2^{2}x_3^{7}x_4^{15}x_5^{13} + x_1^{7}x_2^{3}x_3^{7}x_4^{15}x_5^{12}\big) \ \mbox{ mod}(P_5^-((4)|^2|(2)|^2|(1)),\\  
Z_4 &= x_1^{4}x_2^{7}x_3^{11}x_4^{11}x_5^{19} + x_1^{4}x_2^{7}x_3^{11}x_4^{19}x_5^{11} + x_1^{4}x_2^{11}x_3^{7}x_4^{11}x_5^{19} + x_1^{4}x_2^{11}x_3^{7}x_4^{19}x_5^{11}\\
&\quad + x_1^{5}x_2^{3}x_3^{7}x_4^{26}x_5^{11} + x_1^{5}x_2^{3}x_3^{11}x_4^{11}x_5^{22} + x_1^{5}x_2^{3}x_3^{11}x_4^{14}x_5^{19} + x_1^{5}x_2^{3}x_3^{11}x_4^{19}x_5^{14}\\
&\quad + x_1^{5}x_2^{3}x_3^{11}x_4^{22}x_5^{11} + x_1^{5}x_2^{3}x_3^{14}x_4^{11}x_5^{19} + x_1^{5}x_2^{3}x_3^{22}x_4^{11}x_5^{11} + x_1^{5}x_2^{3}x_3^{26}x_4^{7}x_5^{11}\\
&\quad + x_1^{5}x_2^{7}x_3^{3}x_4^{11}x_5^{26} + x_1^{5}x_2^{7}x_3^{10}x_4^{11}x_5^{19} + x_1^{5}x_2^{7}x_3^{18}x_4^{11}x_5^{11} + x_1^{5}x_2^{11}x_3^{3}x_4^{11}x_5^{22}\\
&\quad + x_1^{5}x_2^{11}x_3^{3}x_4^{19}x_5^{14} + x_1^{5}x_2^{11}x_3^{10}x_4^{7}x_5^{19} + x_1^{5}x_2^{11}x_3^{18}x_4^{7}x_5^{11} + x_1^{7}x_2^{3}x_3^{7}x_4^{11}x_5^{24}\\
&\quad + x_1^{7}x_2^{3}x_3^{7}x_4^{24}x_5^{11} + x_1^{7}x_2^{3}x_3^{9}x_4^{11}x_5^{22} + x_1^{7}x_2^{3}x_3^{9}x_4^{14}x_5^{19} + x_1^{7}x_2^{3}x_3^{9}x_4^{19}x_5^{14}\\
&\quad + x_1^{7}x_2^{3}x_3^{9}x_4^{22}x_5^{11} + x_1^{7}x_2^{3}x_3^{11}x_4^{13}x_5^{18} + x_1^{7}x_2^{3}x_3^{11}x_4^{14}x_5^{17} + x_1^{7}x_2^{3}x_3^{11}x_4^{17}x_5^{14}\\
&\quad + x_1^{7}x_2^{3}x_3^{11}x_4^{18}x_5^{13} + x_1^{7}x_2^{3}x_3^{14}x_4^{9}x_5^{19} + x_1^{7}x_2^{3}x_3^{14}x_4^{11}x_5^{17} + x_1^{7}x_2^{3}x_3^{18}x_4^{11}x_5^{13}\\
&\quad + x_1^{7}x_2^{3}x_3^{22}x_4^{9}x_5^{11} + x_1^{7}x_2^{3}x_3^{24}x_4^{7}x_5^{11} + x_1^{7}x_2^{5}x_3^{11}x_4^{11}x_5^{18} + x_1^{7}x_2^{5}x_3^{11}x_4^{18}x_5^{11}\\
&\quad + x_1^{7}x_2^{5}x_3^{18}x_4^{11}x_5^{11} + x_1^{7}x_2^{7}x_3^{3}x_4^{11}x_5^{24} + x_1^{7}x_2^{7}x_3^{8}x_4^{11}x_5^{19} + x_1^{7}x_2^{7}x_3^{8}x_4^{19}x_5^{11}\\
&\quad + x_1^{7}x_2^{7}x_3^{10}x_4^{9}x_5^{19} + x_1^{7}x_2^{7}x_3^{11}x_4^{11}x_5^{16} + x_1^{7}x_2^{7}x_3^{11}x_4^{16}x_5^{11} + x_1^{7}x_2^{7}x_3^{16}x_4^{11}x_5^{11}\\
&\quad + x_1^{7}x_2^{7}x_3^{18}x_4^{9}x_5^{11} + x_1^{7}x_2^{8}x_3^{7}x_4^{11}x_5^{19} + x_1^{7}x_2^{8}x_3^{7}x_4^{19}x_5^{11} + x_1^{7}x_2^{9}x_3^{3}x_4^{11}x_5^{22}\\
&\quad + x_1^{7}x_2^{9}x_3^{3}x_4^{19}x_5^{14} + x_1^{7}x_2^{9}x_3^{7}x_4^{11}x_5^{18} + x_1^{7}x_2^{9}x_3^{7}x_4^{18}x_5^{11} + x_1^{7}x_2^{9}x_3^{10}x_4^{7}x_5^{19}\\
&\quad + x_1^{7}x_2^{11}x_3^{3}x_4^{13}x_5^{18} + x_1^{7}x_2^{11}x_3^{3}x_4^{17}x_5^{14} + x_1^{7}x_2^{11}x_3^{5}x_4^{11}x_5^{19} + x_1^{7}x_2^{11}x_3^{7}x_4^{11}x_5^{16}\\
&\quad + x_1^{7}x_2^{11}x_3^{7}x_4^{16}x_5^{11} + x_1^{7}x_2^{11}x_3^{9}x_4^{7}x_5^{18} +  Sq^1\big(x_1^{7}x_2^{7}x_3^{7}x_4^{11}x_5^{19} + x_1^{7}x_2^{7}x_3^{7}x_4^{19}x_5^{11}\\
&\quad + x_1^{7}x_2^{7}x_3^{9}x_4^{7}x_5^{21} + x_1^{7}x_2^{7}x_3^{17}x_4^{7}x_5^{13} + x_1^{7}x_2^{9}x_3^{11}x_4^{11}x_5^{13} + x_1^{7}x_2^{9}x_3^{11}x_4^{13}x_5^{11}\\
&\quad + x_1^{7}x_2^{9}x_3^{13}x_4^{11}x_5^{11} + x_1^{7}x_2^{11}x_3^{11}x_4^{11}x_5^{11}\big) +  Sq^2\big(x_1^{7}x_2^{3}x_3^{7}x_4^{11}x_5^{22} + x_1^{7}x_2^{3}x_3^{7}x_4^{14}x_5^{19}\\
&\quad + x_1^{7}x_2^{3}x_3^{7}x_4^{19}x_5^{14} + x_1^{7}x_2^{3}x_3^{7}x_4^{22}x_5^{11} + x_1^{7}x_2^{3}x_3^{14}x_4^{7}x_5^{19} + x_1^{7}x_2^{3}x_3^{22}x_4^{7}x_5^{11}\\
&\quad + x_1^{7}x_2^{7}x_3^{3}x_4^{11}x_5^{22} + x_1^{7}x_2^{7}x_3^{3}x_4^{19}x_5^{14} + x_1^{7}x_2^{7}x_3^{10}x_4^{7}x_5^{19} + x_1^{7}x_2^{7}x_3^{18}x_4^{7}x_5^{11}\\
&\quad + x_1^{7}x_2^{10}x_3^{11}x_4^{11}x_5^{11}\big) +  Sq^4\big(x_1^{4}x_2^{7}x_3^{7}x_4^{11}x_5^{19} + x_1^{4}x_2^{7}x_3^{7}x_4^{19}x_5^{11} + x_1^{5}x_2^{3}x_3^{7}x_4^{11}x_5^{22}\\
&\quad + x_1^{5}x_2^{3}x_3^{7}x_4^{14}x_5^{19} + x_1^{5}x_2^{3}x_3^{7}x_4^{19}x_5^{14} + x_1^{5}x_2^{3}x_3^{7}x_4^{22}x_5^{11} + x_1^{5}x_2^{3}x_3^{14}x_4^{7}x_5^{19}\\
&\quad + x_1^{5}x_2^{3}x_3^{22}x_4^{7}x_5^{11} + x_1^{5}x_2^{7}x_3^{3}x_4^{11}x_5^{22} + x_1^{5}x_2^{7}x_3^{3}x_4^{19}x_5^{14} + x_1^{5}x_2^{7}x_3^{10}x_4^{7}x_5^{19}\\
&\quad + x_1^{5}x_2^{7}x_3^{18}x_4^{7}x_5^{11} + x_1^{11}x_2^{3}x_3^{7}x_4^{13}x_5^{14} + x_1^{11}x_2^{3}x_3^{7}x_4^{14}x_5^{13} + x_1^{11}x_2^{3}x_3^{14}x_4^{7}x_5^{13}\\
&\quad + x_1^{11}x_2^{5}x_3^{7}x_4^{11}x_5^{14} + x_1^{11}x_2^{5}x_3^{7}x_4^{14}x_5^{11} + x_1^{11}x_2^{5}x_3^{14}x_4^{7}x_5^{11} + x_1^{11}x_2^{7}x_3^{3}x_4^{13}x_5^{14}\\
&\quad + x_1^{11}x_2^{7}x_3^{5}x_4^{11}x_5^{14} + x_1^{11}x_2^{7}x_3^{7}x_4^{11}x_5^{12} + x_1^{11}x_2^{7}x_3^{7}x_4^{12}x_5^{11} + x_1^{11}x_2^{7}x_3^{9}x_4^{7}x_5^{14}\\
&\quad + x_1^{11}x_2^{7}x_3^{12}x_4^{7}x_5^{11}\big) +  Sq^8\big(x_1^{7}x_2^{3}x_3^{7}x_4^{13}x_5^{14} + x_1^{7}x_2^{3}x_3^{7}x_4^{14}x_5^{13}\\
&\quad + x_1^{7}x_2^{3}x_3^{14}x_4^{7}x_5^{13} + x_1^{7}x_2^{5}x_3^{7}x_4^{11}x_5^{14} + x_1^{7}x_2^{5}x_3^{7}x_4^{14}x_5^{11} + x_1^{7}x_2^{5}x_3^{14}x_4^{7}x_5^{11}\\
&\quad + x_1^{7}x_2^{7}x_3^{3}x_4^{13}x_5^{14} + x_1^{7}x_2^{7}x_3^{5}x_4^{11}x_5^{14} + x_1^{7}x_2^{7}x_3^{7}x_4^{11}x_5^{12} + x_1^{7}x_2^{7}x_3^{7}x_4^{12}x_5^{11}\\
&\quad + x_1^{7}x_2^{7}x_3^{9}x_4^{7}x_5^{14} + x_1^{7}x_2^{7}x_3^{12}x_4^{7}x_5^{11}\big) \ \mbox{ mod}(P_5^-((4)|^2|(2)|^2|(1)),\\  
Z_5 &= x_1^{4}x_2^{7}x_3^{11}x_4^{11}x_5^{19} + x_1^{4}x_2^{7}x_3^{11}x_4^{19}x_5^{11} + x_1^{4}x_2^{11}x_3^{7}x_4^{11}x_5^{19} + x_1^{4}x_2^{11}x_3^{7}x_4^{19}x_5^{11}\\
&\quad + x_1^{5}x_2^{7}x_3^{11}x_4^{3}x_5^{26} + x_1^{5}x_2^{7}x_3^{11}x_4^{10}x_5^{19} + x_1^{5}x_2^{7}x_3^{11}x_4^{11}x_5^{18} + x_1^{5}x_2^{7}x_3^{11}x_4^{18}x_5^{11}\\
&\quad + x_1^{5}x_2^{7}x_3^{11}x_4^{19}x_5^{10} + x_1^{5}x_2^{7}x_3^{11}x_4^{26}x_5^{3} + x_1^{5}x_2^{11}x_3^{7}x_4^{10}x_5^{19} + x_1^{5}x_2^{11}x_3^{7}x_4^{11}x_5^{18}\\
&\quad + x_1^{5}x_2^{11}x_3^{7}x_4^{18}x_5^{11} + x_1^{5}x_2^{11}x_3^{7}x_4^{19}x_5^{10} + x_1^{5}x_2^{11}x_3^{11}x_4^{3}x_5^{22} + x_1^{5}x_2^{11}x_3^{11}x_4^{22}x_5^{3}\\
&\quad + x_1^{5}x_2^{11}x_3^{19}x_4^{3}x_5^{14} + x_1^{5}x_2^{11}x_3^{19}x_4^{14}x_5^{3} + x_1^{7}x_2^{7}x_3^{8}x_4^{11}x_5^{19} + x_1^{7}x_2^{7}x_3^{8}x_4^{19}x_5^{11}\\
&\quad + x_1^{7}x_2^{7}x_3^{9}x_4^{10}x_5^{19} + x_1^{7}x_2^{7}x_3^{9}x_4^{11}x_5^{18} + x_1^{7}x_2^{7}x_3^{9}x_4^{18}x_5^{11} + x_1^{7}x_2^{7}x_3^{9}x_4^{19}x_5^{10}\\
&\quad + x_1^{7}x_2^{7}x_3^{11}x_4^{3}x_5^{24} + x_1^{7}x_2^{7}x_3^{11}x_4^{24}x_5^{3} + x_1^{7}x_2^{8}x_3^{7}x_4^{11}x_5^{19} + x_1^{7}x_2^{8}x_3^{7}x_4^{19}x_5^{11}\\
&\quad + x_1^{7}x_2^{9}x_3^{7}x_4^{10}x_5^{19} + x_1^{7}x_2^{9}x_3^{7}x_4^{11}x_5^{18} + x_1^{7}x_2^{9}x_3^{7}x_4^{18}x_5^{11} + x_1^{7}x_2^{9}x_3^{7}x_4^{19}x_5^{10}\\
&\quad + x_1^{7}x_2^{9}x_3^{11}x_4^{3}x_5^{22} + x_1^{7}x_2^{9}x_3^{11}x_4^{22}x_5^{3} + x_1^{7}x_2^{9}x_3^{19}x_4^{3}x_5^{14} + x_1^{7}x_2^{9}x_3^{19}x_4^{14}x_5^{3}\\
&\quad + x_1^{7}x_2^{11}x_3^{7}x_4^{9}x_5^{18} + x_1^{7}x_2^{11}x_3^{7}x_4^{18}x_5^{9} + x_1^{7}x_2^{11}x_3^{11}x_4^{5}x_5^{18} + x_1^{7}x_2^{11}x_3^{11}x_4^{18}x_5^{5}\\
&\quad + x_1^{7}x_2^{11}x_3^{13}x_4^{3}x_5^{18} + x_1^{7}x_2^{11}x_3^{13}x_4^{18}x_5^{3} + x_1^{7}x_2^{11}x_3^{17}x_4^{3}x_5^{14} +  Sq^1\big(x_1^{7}x_2^{7}x_3^{7}x_4^{11}x_5^{19}\\
&\quad + x_1^{7}x_2^{7}x_3^{7}x_4^{19}x_5^{11} + x_1^{7}x_2^{7}x_3^{7}x_4^{9}x_5^{21} + x_1^{7}x_2^{7}x_3^{7}x_4^{13}x_5^{17} + x_1^{7}x_2^{7}x_3^{7}x_4^{17}x_5^{13}\\
&\quad + x_1^{7}x_2^{7}x_3^{7}x_4^{21}x_5^{9}\big) +  Sq^2\big(x_1^{7}x_2^{7}x_3^{7}x_4^{10}x_5^{19} + x_1^{7}x_2^{7}x_3^{7}x_4^{11}x_5^{18} + x_1^{7}x_2^{7}x_3^{7}x_4^{18}x_5^{11}\\
&\quad + x_1^{7}x_2^{7}x_3^{7}x_4^{19}x_5^{10} + x_1^{7}x_2^{7}x_3^{11}x_4^{3}x_5^{22} + x_1^{7}x_2^{7}x_3^{11}x_4^{22}x_5^{3} + x_1^{7}x_2^{7}x_3^{19}x_4^{3}x_5^{14}\\
&\quad + x_1^{7}x_2^{7}x_3^{19}x_4^{14}x_5^{3}\big) +  Sq^4\big(x_1^{4}x_2^{7}x_3^{7}x_4^{11}x_5^{19} + x_1^{4}x_2^{7}x_3^{7}x_4^{19}x_5^{11} + x_1^{5}x_2^{7}x_3^{7}x_4^{10}x_5^{19}\\
&\quad + x_1^{5}x_2^{7}x_3^{7}x_4^{11}x_5^{18} + x_1^{5}x_2^{7}x_3^{7}x_4^{18}x_5^{11} + x_1^{5}x_2^{7}x_3^{7}x_4^{19}x_5^{10} + x_1^{5}x_2^{7}x_3^{11}x_4^{3}x_5^{22}\\
&\quad + x_1^{5}x_2^{7}x_3^{11}x_4^{22}x_5^{3} + x_1^{5}x_2^{7}x_3^{19}x_4^{3}x_5^{14} + x_1^{5}x_2^{7}x_3^{19}x_4^{14}x_5^{3} + x_1^{11}x_2^{7}x_3^{7}x_4^{9}x_5^{14}\\
&\quad + x_1^{11}x_2^{7}x_3^{7}x_4^{14}x_5^{9} + x_1^{11}x_2^{7}x_3^{11}x_4^{5}x_5^{14} + x_1^{11}x_2^{7}x_3^{11}x_4^{14}x_5^{5} + x_1^{11}x_2^{7}x_3^{13}x_4^{3}x_5^{14}\\
&\quad + x_1^{11}x_2^{7}x_3^{13}x_4^{14}x_5^{3}\big) +  Sq^8\big(x_1^{7}x_2^{7}x_3^{7}x_4^{9}x_5^{14} + x_1^{7}x_2^{7}x_3^{7}x_4^{14}x_5^{9} + x_1^{7}x_2^{7}x_3^{11}x_4^{5}x_5^{14}\\
&\quad + x_1^{7}x_2^{7}x_3^{11}x_4^{14}x_5^{5} + x_1^{7}x_2^{7}x_3^{13}x_4^{3}x_5^{14} + x_1^{7}x_2^{7}x_3^{13}x_4^{14}x_5^{3}\big) \ \mbox{ mod}(P_5^-((4)|^2|(2)|^2|(1)),\\
Z_6 &= x_1^{11}x_2^{2}x_3^{7}x_4^{11}x_5^{21} + x_1^{11}x_2^{2}x_3^{7}x_4^{13}x_5^{19} + x_1^{11}x_2^{2}x_3^{7}x_4^{19}x_5^{13} + x_1^{11}x_2^{2}x_3^{7}x_4^{21}x_5^{11}\\
&\quad + x_1^{11}x_2^{2}x_3^{13}x_4^{7}x_5^{19} + x_1^{11}x_2^{2}x_3^{21}x_4^{7}x_5^{11} + x_1^{11}x_2^{3}x_3^{7}x_4^{11}x_5^{20} + x_1^{11}x_2^{3}x_3^{7}x_4^{12}x_5^{19}\\
&\quad + x_1^{11}x_2^{3}x_3^{7}x_4^{19}x_5^{12} + x_1^{11}x_2^{3}x_3^{7}x_4^{20}x_5^{11} + x_1^{11}x_2^{3}x_3^{12}x_4^{7}x_5^{19} + x_1^{11}x_2^{3}x_3^{20}x_4^{7}x_5^{11}\\
&\quad + x_1^{15}x_2^{2}x_3^{7}x_4^{11}x_5^{17} + x_1^{15}x_2^{2}x_3^{7}x_4^{17}x_5^{11} + x_1^{15}x_2^{2}x_3^{17}x_4^{7}x_5^{11} + x_1^{15}x_2^{3}x_3^{7}x_4^{11}x_5^{16}\\
&\quad + x_1^{15}x_2^{3}x_3^{7}x_4^{16}x_5^{11} +  Sq^1\big(x_1^{15}x_2^{3}x_3^{11}x_4^{11}x_5^{11}\big) +  Sq^2\big(x_1^{15}x_2^{2}x_3^{11}x_4^{11}x_5^{11}\big)\\
&\quad +  Sq^4\big(x_1^{15}x_2^{2}x_3^{7}x_4^{11}x_5^{13} + x_1^{15}x_2^{2}x_3^{7}x_4^{13}x_5^{11} + x_1^{15}x_2^{2}x_3^{13}x_4^{7}x_5^{11}\\
&\quad + x_1^{15}x_2^{3}x_3^{7}x_4^{11}x_5^{12} + x_1^{15}x_2^{3}x_3^{7}x_4^{12}x_5^{11} + x_1^{15}x_2^{3}x_3^{12}x_4^{7}x_5^{11}\big)\\
&\quad +  Sq^8\big(x_1^{11}x_2^{2}x_3^{7}x_4^{11}x_5^{13} + x_1^{11}x_2^{2}x_3^{7}x_4^{13}x_5^{11} + x_1^{11}x_2^{2}x_3^{13}x_4^{7}x_5^{11} + x_1^{11}x_2^{3}x_3^{7}x_4^{11}x_5^{12}\\
&\quad + x_1^{11}x_2^{3}x_3^{7}x_4^{12}x_5^{11} + x_1^{11}x_2^{3}x_3^{12}x_4^{7}x_5^{11}\big) \ \mbox{ mod}(P_5^-((4)|^2|(2)|^2|(1)),\\
Z_7 &= x_1^{11}x_2^{7}x_3^{3}x_4^{10}x_5^{21} + x_1^{11}x_2^{7}x_3^{3}x_4^{12}x_5^{19} + x_1^{11}x_2^{7}x_3^{3}x_4^{18}x_5^{13} + x_1^{11}x_2^{7}x_3^{3}x_4^{20}x_5^{11}\\
&\quad + x_1^{11}x_2^{7}x_3^{10}x_4^{3}x_5^{21} + x_1^{11}x_2^{7}x_3^{12}x_4^{3}x_5^{19} + x_1^{11}x_2^{7}x_3^{18}x_4^{3}x_5^{13} + x_1^{11}x_2^{7}x_3^{20}x_4^{3}x_5^{11}\\
&\quad + x_1^{11}x_2^{11}x_3^{5}x_4^{6}x_5^{19} + x_1^{11}x_2^{13}x_3^{3}x_4^{6}x_5^{19} + x_1^{11}x_2^{13}x_3^{6}x_4^{3}x_5^{19} + x_1^{11}x_2^{19}x_3^{5}x_4^{6}x_5^{11}\\
&\quad + x_1^{11}x_2^{21}x_3^{3}x_4^{6}x_5^{11} + x_1^{11}x_2^{21}x_3^{6}x_4^{3}x_5^{11} + x_1^{15}x_2^{7}x_3^{3}x_4^{10}x_5^{17} + x_1^{15}x_2^{7}x_3^{3}x_4^{16}x_5^{11}\\
&\quad + x_1^{15}x_2^{7}x_3^{10}x_4^{3}x_5^{17} + x_1^{15}x_2^{7}x_3^{16}x_4^{3}x_5^{11} + x_1^{15}x_2^{17}x_3^{3}x_4^{6}x_5^{11} +  Sq^1\big(x_1^{15}x_2^{11}x_3^{9}x_4^{5}x_5^{11}\big)\\
&\quad +  Sq^2\big(x_1^{15}x_2^{11}x_3^{3}x_4^{10}x_5^{11} + x_1^{15}x_2^{11}x_3^{10}x_4^{3}x_5^{11}\big) +  Sq^4\big(x_1^{15}x_2^{7}x_3^{3}x_4^{10}x_5^{13}\\
&\quad + x_1^{15}x_2^{7}x_3^{3}x_4^{12}x_5^{11} + x_1^{15}x_2^{7}x_3^{10}x_4^{3}x_5^{13} + x_1^{15}x_2^{7}x_3^{12}x_4^{3}x_5^{11} + x_1^{15}x_2^{11}x_3^{5}x_4^{6}x_5^{11}\\
&\quad + x_1^{15}x_2^{13}x_3^{3}x_4^{6}x_5^{11} + x_1^{15}x_2^{13}x_3^{6}x_4^{3}x_5^{11} +  Sq^8\big(x_1^{11}x_2^{7}x_3^{3}x_4^{10}x_5^{13} + x_1^{11}x_2^{7}x_3^{3}x_4^{12}x_5^{11}\\
&\quad + x_1^{11}x_2^{7}x_3^{10}x_4^{3}x_5^{13} + x_1^{11}x_2^{7}x_3^{12}x_4^{3}x_5^{11} + x_1^{11}x_2^{11}x_3^{5}x_4^{6}x_5^{11} + x_1^{11}x_2^{13}x_3^{3}x_4^{6}x_5^{11}\\
&\quad + x_1^{11}x_2^{13}x_3^{6}x_4^{3}x_5^{11}\big) \ \mbox{ mod}(P_5^-((4)|^2|(2)|^2|(1)),\\
Z_8 &= x_1^{11}x_2^{7}x_3^{11}x_4^{3}x_5^{20} + x_1^{11}x_2^{7}x_3^{11}x_4^{20}x_5^{3} + x_1^{11}x_2^{7}x_3^{13}x_4^{3}x_5^{18} + x_1^{11}x_2^{7}x_3^{13}x_4^{18}x_5^{3}\\
&\quad + x_1^{11}x_2^{7}x_3^{19}x_4^{3}x_5^{12} + x_1^{11}x_2^{7}x_3^{19}x_4^{12}x_5^{3} + x_1^{11}x_2^{7}x_3^{21}x_4^{3}x_5^{10} + x_1^{11}x_2^{7}x_3^{21}x_4^{10}x_5^{3}\\
&\quad + x_1^{11}x_2^{11}x_3^{19}x_4^{5}x_5^{6} + x_1^{11}x_2^{13}x_3^{7}x_4^{3}x_5^{18} + x_1^{11}x_2^{13}x_3^{7}x_4^{18}x_5^{3} + x_1^{11}x_2^{19}x_3^{11}x_4^{5}x_5^{6}\\
&\quad + x_1^{11}x_2^{21}x_3^{7}x_4^{3}x_5^{10} + x_1^{11}x_2^{21}x_3^{7}x_4^{10}x_5^{3} + x_1^{15}x_2^{7}x_3^{11}x_4^{3}x_5^{16} + x_1^{15}x_2^{7}x_3^{11}x_4^{16}x_5^{3}\\
&\quad + x_1^{15}x_2^{7}x_3^{17}x_4^{3}x_5^{10} + x_1^{15}x_2^{7}x_3^{17}x_4^{10}x_5^{3} + x_1^{15}x_2^{17}x_3^{7}x_4^{3}x_5^{10} +  Sq^1\big(x_1^{15}x_2^{11}x_3^{11}x_4^{9}x_5^{5}\big)\\
&\quad +  Sq^2\big(x_1^{15}x_2^{11}x_3^{11}x_4^{3}x_5^{10} + x_1^{15}x_2^{11}x_3^{11}x_4^{10}x_5^{3}\big) +  Sq^4\big(x_1^{15}x_2^{7}x_3^{11}x_4^{3}x_5^{12} + x_1^{15}x_2^{7}x_3^{11}x_4^{12}x_5^{3}\\
&\quad + x_1^{15}x_2^{7}x_3^{13}x_4^{3}x_5^{10} + x_1^{15}x_2^{7}x_3^{13}x_4^{10}x_5^{3} + x_1^{15}x_2^{11}x_3^{11}x_4^{5}x_5^{6} + x_1^{15}x_2^{13}x_3^{7}x_4^{3}x_5^{10}\\
&\quad + x_1^{15}x_2^{13}x_3^{7}x_4^{10}x_5^{3}\big) +  Sq^8\big(x_1^{11}x_2^{7}x_3^{11}x_4^{3}x_5^{12} + x_1^{11}x_2^{7}x_3^{11}x_4^{12}x_5^{3}\\
&\quad + x_1^{11}x_2^{7}x_3^{13}x_4^{3}x_5^{10} + x_1^{11}x_2^{7}x_3^{13}x_4^{10}x_5^{3} + x_1^{11}x_2^{11}x_3^{11}x_4^{5}x_5^{6}\\
&\quad + x_1^{11}x_2^{13}x_3^{7}x_4^{3}x_5^{10} + x_1^{11}x_2^{13}x_3^{7}x_4^{10}x_5^{3}\big) \ \mbox{ mod}(P_5^-((4)|^2|(2)|^2|(1)).
\end{align*}
These equalities imply the monomials $Z_j,\, 1 \leqslant j \leqslant 8$, are strictly inadmissible.
\end{proof}

\begin{lems}\label{bd6}\
	
\medskip
{\rm i)} If $(i,j,t,u,v)$ is an arbitrary permutation of $(1,2,3,4,5)$, then the monomials $x_i^{7}x_j^{7}x_t^{31}x_u^{31}x_v^{32}$,  $x_i^{7}x_j^{15}x_t^{23}x_u^{31}x_v^{32}$, $x_i^{15}x_j^{15}x_t^{23}x_u^{23}x_v^{32}$ are strictly inadmissible except for monomials $x_1^{15}x_2^{15}x_3^{23}x_4^{23}x_5^{32}$ and $x_1^{15}x_2^{15}x_3^{23}x_4^{32}x_5^{23}$.
	
\smallskip
{\rm ii)} The following monomials are strictly inadmissible:
	
\medskip
\centerline{\begin{tabular}{llll}
$x_1^{7}x_2^{15}x_3^{15}x_4^{55}x_5^{16}$& $x_1^{7}x_2^{15}x_3^{17}x_4^{7}x_5^{62}$& $x_1^{7}x_2^{15}x_3^{17}x_4^{39}x_5^{30}$& $x_1^{7}x_2^{15}x_3^{23}x_4^{15}x_5^{48} $\cr  $x_1^{7}x_2^{15}x_3^{23}x_4^{47}x_5^{16}$& $x_1^{7}x_2^{15}x_3^{31}x_4^{33}x_5^{22}$& $x_1^{7}x_2^{15}x_3^{55}x_4^{15}x_5^{16}$& $x_1^{7}x_2^{31}x_3^{15}x_4^{33}x_5^{22} $\cr  $x_1^{7}x_2^{31}x_3^{31}x_4^{33}x_5^{6}$& $x_1^{7}x_2^{31}x_3^{39}x_4^{9}x_5^{22}$& $x_1^{15}x_2^{7}x_3^{15}x_4^{55}x_5^{16}$& $x_1^{15}x_2^{7}x_3^{17}x_4^{7}x_5^{62} $\cr  $x_1^{15}x_2^{7}x_3^{23}x_4^{15}x_5^{48}$& $x_1^{15}x_2^{7}x_3^{23}x_4^{32}x_5^{31}$& $x_1^{15}x_2^{7}x_3^{23}x_4^{47}x_5^{16}$& $x_1^{15}x_2^{7}x_3^{31}x_4^{33}x_5^{22} $\cr  $x_1^{15}x_2^{7}x_3^{55}x_4^{15}x_5^{16}$& $x_1^{15}x_2^{15}x_3^{7}x_4^{55}x_5^{16}$& $x_1^{15}x_2^{15}x_3^{23}x_4^{33}x_5^{22}$& $x_1^{15}x_2^{15}x_3^{23}x_4^{39}x_5^{16} $\cr  $x_1^{15}x_2^{15}x_3^{55}x_4^{7}x_5^{16}$& $x_1^{15}x_2^{15}x_3^{55}x_4^{17}x_5^{6}$& $x_1^{15}x_2^{19}x_3^{39}x_4^{13}x_5^{22}$& $x_1^{15}x_2^{23}x_3^{7}x_4^{32}x_5^{31} $\cr  $x_1^{15}x_2^{23}x_3^{7}x_4^{33}x_5^{30}$& $x_1^{15}x_2^{23}x_3^{15}x_4^{33}x_5^{22}$& $x_1^{15}x_2^{23}x_3^{31}x_4^{33}x_5^{6}$& $x_1^{15}x_2^{23}x_3^{39}x_4^{9}x_5^{22} $\cr  $x_1^{15}x_2^{31}x_3^{7}x_4^{33}x_5^{22}$& $x_1^{15}x_2^{31}x_3^{23}x_4^{33}x_5^{6}$& $x_1^{15}x_2^{31}x_3^{33}x_4^{6}x_5^{23}$& $x_1^{15}x_2^{31}x_3^{33}x_4^{7}x_5^{22} $\cr  $x_1^{15}x_2^{31}x_3^{33}x_4^{22}x_5^{7}$& $x_1^{15}x_2^{31}x_3^{33}x_4^{23}x_5^{6}$& $x_1^{15}x_2^{31}x_3^{35}x_4^{5}x_5^{22}$& $x_1^{15}x_2^{31}x_3^{35}x_4^{21}x_5^{6} $\cr  $x_1^{15}x_2^{51}x_3^{7}x_4^{13}x_5^{22}$& $x_1^{31}x_2^{7}x_3^{15}x_4^{33}x_5^{22}$& $x_1^{31}x_2^{7}x_3^{31}x_4^{33}x_5^{6}$& $x_1^{31}x_2^{7}x_3^{39}x_4^{9}x_5^{22} $\cr  $x_1^{31}x_2^{15}x_3^{7}x_4^{33}x_5^{22}$& $x_1^{31}x_2^{15}x_3^{23}x_4^{33}x_5^{6}$& $x_1^{31}x_2^{15}x_3^{33}x_4^{6}x_5^{23}$& $x_1^{31}x_2^{15}x_3^{33}x_4^{7}x_5^{22} $\cr  $x_1^{31}x_2^{15}x_3^{33}x_4^{22}x_5^{7}$& $x_1^{31}x_2^{15}x_3^{33}x_4^{23}x_5^{6}$& $x_1^{31}x_2^{15}x_3^{35}x_4^{5}x_5^{22}$& $x_1^{31}x_2^{15}x_3^{35}x_4^{21}x_5^{6} $\cr  $x_1^{31}x_2^{31}x_3^{7}x_4^{33}x_5^{6}$& $x_1^{31}x_2^{31}x_3^{33}x_4^{6}x_5^{7}$& $x_1^{31}x_2^{31}x_3^{33}x_4^{7}x_5^{6}$& $x_1^{31}x_2^{31}x_3^{35}x_4^{5}x_5^{6} $\cr  $x_1^{31}x_2^{35}x_3^{7}x_4^{13}x_5^{22}$& $x_1^{31}x_2^{39}x_3^{7}x_4^{9}x_5^{22}$& $x_1^{31}x_2^{39}x_3^{11}x_4^{5}x_5^{22}$& $x_1^{31}x_2^{39}x_3^{11}x_4^{21}x_5^{6} $\cr 
\end{tabular}} 
\end{lems}

\begin{proof} The monomials in this lemma are of weight vector $(4)|^3|(2)|^2|(1)$.  We prove Part i) for 
$U_{1} =  x_1^{7}x_2^{15}x_3^{23}x_4^{31}x_5^{32}$, $U_{2} =  x_1^{15}x_2^{23}x_3^{15}x_4^{32}x_5^{23}$, $U_{3} =  x_1^{31}x_2^{15}x_3^{32}x_4^{23}x_5^{7}$. By using the Cartan formula we have
\begin{align*}
U_1 &= x_1^{5}x_2^{11}x_3^{22}x_4^{31}x_5^{39} + x_1^{5}x_2^{11}x_3^{22}x_4^{39}x_5^{31} + x_1^{5}x_2^{15}x_3^{22}x_4^{31}x_5^{35} + x_1^{5}x_2^{15}x_3^{22}x_4^{35}x_5^{31}\\ 
&\quad + x_1^{6}x_2^{9}x_3^{23}x_4^{31}x_5^{39} + x_1^{6}x_2^{9}x_3^{23}x_4^{39}x_5^{31} + x_1^{6}x_2^{15}x_3^{23}x_4^{31}x_5^{33} + x_1^{6}x_2^{15}x_3^{23}x_4^{33}x_5^{31}\\ 
&\quad + x_1^{7}x_2^{8}x_3^{23}x_4^{31}x_5^{39} + x_1^{7}x_2^{8}x_3^{23}x_4^{39}x_5^{31} + x_1^{7}x_2^{9}x_3^{22}x_4^{31}x_5^{39} + x_1^{7}x_2^{9}x_3^{22}x_4^{39}x_5^{31}\\ 
&\quad + x_1^{7}x_2^{15}x_3^{22}x_4^{31}x_5^{33} + x_1^{7}x_2^{15}x_3^{22}x_4^{33}x_5^{31} + x_1^{7}x_2^{15}x_3^{23}x_4^{16}x_5^{47} +  Sq^1\big(x_1^{5}x_2^{15}x_3^{25}x_4^{31}x_5^{31}\\ 
&\quad + x_1^{7}x_2^{15}x_3^{23}x_4^{31}x_5^{31}\big) +  Sq^2\big(x_1^{7}x_2^{15}x_3^{22}x_4^{31}x_5^{31}\big) +  Sq^4\big(x_1^{5}x_2^{15}x_3^{22}x_4^{31}x_5^{31}\big)\\ 
&\quad +  Sq^8\big(x_1^{5}x_2^{11}x_3^{22}x_4^{31}x_5^{31} + x_1^{6}x_2^{9}x_3^{23}x_4^{31}x_5^{31} + x_1^{7}x_2^{8}x_3^{23}x_4^{31}x_5^{31} + x_1^{7}x_2^{9}x_3^{22}x_4^{31}x_5^{31}\big)\\ 
&\quad +  Sq^{16}\big(x_1^{7}x_2^{15}x_3^{23}x_4^{16}x_5^{31}\big) \ \mbox{ mod}(P_5^-((4)|^3|(2)|^2|(1)),\\
U_2 &= x_1^{5}x_2^{15}x_3^{23}x_4^{27}x_5^{38} + x_1^{5}x_2^{15}x_3^{23}x_4^{30}x_5^{35} + x_1^{5}x_2^{15}x_3^{23}x_4^{35}x_5^{30} + x_1^{5}x_2^{15}x_3^{23}x_4^{38}x_5^{27}\\ 
&\quad + x_1^{5}x_2^{15}x_3^{27}x_4^{23}x_5^{38} + x_1^{5}x_2^{15}x_3^{27}x_4^{38}x_5^{23} + x_1^{5}x_2^{15}x_3^{35}x_4^{23}x_5^{30} + x_1^{5}x_2^{15}x_3^{35}x_4^{30}x_5^{23}\\ 
&\quad + x_1^{5}x_2^{27}x_3^{15}x_4^{23}x_5^{38} + x_1^{5}x_2^{27}x_3^{15}x_4^{38}x_5^{23} + x_1^{5}x_2^{35}x_3^{15}x_4^{23}x_5^{30} + x_1^{5}x_2^{35}x_3^{15}x_4^{30}x_5^{23}\\ 
&\quad + x_1^{7}x_2^{13}x_3^{23}x_4^{27}x_5^{38} + x_1^{7}x_2^{13}x_3^{23}x_4^{30}x_5^{35} + x_1^{7}x_2^{13}x_3^{23}x_4^{35}x_5^{30} + x_1^{7}x_2^{13}x_3^{23}x_4^{38}x_5^{27}\\ 
&\quad + x_1^{7}x_2^{13}x_3^{27}x_4^{23}x_5^{38} + x_1^{7}x_2^{13}x_3^{27}x_4^{38}x_5^{23} + x_1^{7}x_2^{13}x_3^{35}x_4^{23}x_5^{30} + x_1^{7}x_2^{13}x_3^{35}x_4^{30}x_5^{23}\\ 
&\quad + x_1^{7}x_2^{25}x_3^{15}x_4^{23}x_5^{38} + x_1^{7}x_2^{25}x_3^{15}x_4^{38}x_5^{23} + x_1^{7}x_2^{33}x_3^{15}x_4^{23}x_5^{30} + x_1^{7}x_2^{33}x_3^{15}x_4^{30}x_5^{23}\\ 
&\quad + x_1^{9}x_2^{23}x_3^{15}x_4^{23}x_5^{38} + x_1^{9}x_2^{23}x_3^{15}x_4^{38}x_5^{23} + x_1^{11}x_2^{13}x_3^{7}x_4^{23}x_5^{54} + x_1^{11}x_2^{13}x_3^{7}x_4^{54}x_5^{23}\\ 
&\quad + x_1^{11}x_2^{15}x_3^{5}x_4^{23}x_5^{54} + x_1^{11}x_2^{15}x_3^{5}x_4^{54}x_5^{23} + x_1^{11}x_2^{21}x_3^{7}x_4^{23}x_5^{46} + x_1^{11}x_2^{21}x_3^{7}x_4^{30}x_5^{39}\\ 
&\quad + x_1^{11}x_2^{21}x_3^{7}x_4^{39}x_5^{30} + x_1^{11}x_2^{21}x_3^{7}x_4^{46}x_5^{23} + x_1^{11}x_2^{21}x_3^{15}x_4^{23}x_5^{38} + x_1^{11}x_2^{21}x_3^{15}x_4^{38}x_5^{23}\\ 
&\quad + x_1^{11}x_2^{23}x_3^{5}x_4^{23}x_5^{46} + x_1^{11}x_2^{23}x_3^{5}x_4^{30}x_5^{39} + x_1^{11}x_2^{23}x_3^{5}x_4^{39}x_5^{30} + x_1^{11}x_2^{23}x_3^{5}x_4^{46}x_5^{23}\\ 
&\quad + x_1^{15}x_2^{13}x_3^{7}x_4^{23}x_5^{50} + x_1^{15}x_2^{13}x_3^{7}x_4^{50}x_5^{23} + x_1^{15}x_2^{13}x_3^{19}x_4^{23}x_5^{38} + x_1^{15}x_2^{13}x_3^{19}x_4^{38}x_5^{23}\\ 
&\quad + x_1^{15}x_2^{15}x_3^{5}x_4^{23}x_5^{50} + x_1^{15}x_2^{15}x_3^{5}x_4^{50}x_5^{23} + x_1^{15}x_2^{15}x_3^{17}x_4^{23}x_5^{38} + x_1^{15}x_2^{15}x_3^{17}x_4^{38}x_5^{23}\\ 
&\quad + x_1^{15}x_2^{17}x_3^{7}x_4^{23}x_5^{46} + x_1^{15}x_2^{17}x_3^{7}x_4^{30}x_5^{39} + x_1^{15}x_2^{17}x_3^{7}x_4^{39}x_5^{30} + x_1^{15}x_2^{17}x_3^{7}x_4^{46}x_5^{23}\\ 
&\quad + x_1^{15}x_2^{19}x_3^{5}x_4^{23}x_5^{46} + x_1^{15}x_2^{19}x_3^{5}x_4^{30}x_5^{39} + x_1^{15}x_2^{19}x_3^{5}x_4^{39}x_5^{30} + x_1^{15}x_2^{19}x_3^{5}x_4^{46}x_5^{23}\\ 
&\quad + x_1^{15}x_2^{21}x_3^{7}x_4^{27}x_5^{38} + x_1^{15}x_2^{21}x_3^{7}x_4^{30}x_5^{35} + x_1^{15}x_2^{21}x_3^{7}x_4^{35}x_5^{30} + x_1^{15}x_2^{21}x_3^{7}x_4^{38}x_5^{27}\\ 
&\quad + x_1^{15}x_2^{21}x_3^{11}x_4^{23}x_5^{38} + x_1^{15}x_2^{21}x_3^{11}x_4^{38}x_5^{23} + x_1^{15}x_2^{21}x_3^{15}x_4^{23}x_5^{34} + x_1^{15}x_2^{21}x_3^{15}x_4^{34}x_5^{23}\\ 
&\quad + x_1^{15}x_2^{23}x_3^{5}x_4^{27}x_5^{38} + x_1^{15}x_2^{23}x_3^{5}x_4^{30}x_5^{35} + x_1^{15}x_2^{23}x_3^{5}x_4^{35}x_5^{30} + x_1^{15}x_2^{23}x_3^{5}x_4^{38}x_5^{27}\\ 
&\quad + x_1^{15}x_2^{23}x_3^{9}x_4^{23}x_5^{38} + x_1^{15}x_2^{23}x_3^{9}x_4^{38}x_5^{23} + x_1^{15}x_2^{23}x_3^{15}x_4^{23}x_5^{32} +  Sq^1\big(x_1^{15}x_2^{19}x_3^{15}x_4^{29}x_5^{29}\\ 
&\quad + x_1^{15}x_2^{23}x_3^{15}x_4^{25}x_5^{29} + x_1^{15}x_2^{23}x_3^{15}x_4^{27}x_5^{27} + x_1^{15}x_2^{23}x_3^{15}x_4^{29}x_5^{25}\big)\\ 
&\quad +  Sq^2\big(x_1^{15}x_2^{19}x_3^{15}x_4^{27}x_5^{30} + x_1^{15}x_2^{19}x_3^{15}x_4^{30}x_5^{27} + x_1^{15}x_2^{23}x_3^{15}x_4^{23}x_5^{30}\\ 
&\quad + x_1^{15}x_2^{23}x_3^{15}x_4^{26}x_5^{27} + x_1^{15}x_2^{23}x_3^{15}x_4^{27}x_5^{26} + x_1^{15}x_2^{23}x_3^{15}x_4^{30}x_5^{23}\big)\\ 
&\quad +  Sq^4\big(x_1^{5}x_2^{23}x_3^{23}x_4^{23}x_5^{30} + x_1^{5}x_2^{23}x_3^{23}x_4^{30}x_5^{23} + x_1^{7}x_2^{21}x_3^{23}x_4^{23}x_5^{30} + x_1^{7}x_2^{21}x_3^{23}x_4^{30}x_5^{23}\\ 
&\quad + x_1^{15}x_2^{13}x_3^{7}x_4^{23}x_5^{46} + x_1^{15}x_2^{13}x_3^{7}x_4^{30}x_5^{39} + x_1^{15}x_2^{13}x_3^{7}x_4^{39}x_5^{30} + x_1^{15}x_2^{13}x_3^{7}x_4^{46}x_5^{23}\\ 
&\quad + x_1^{15}x_2^{15}x_3^{5}x_4^{23}x_5^{46} + x_1^{15}x_2^{15}x_3^{5}x_4^{30}x_5^{39} + x_1^{15}x_2^{15}x_3^{5}x_4^{39}x_5^{30} + x_1^{15}x_2^{15}x_3^{5}x_4^{46}x_5^{23}\\ 
&\quad + x_1^{15}x_2^{21}x_3^{15}x_4^{23}x_5^{30} + x_1^{15}x_2^{21}x_3^{15}x_4^{30}x_5^{23}\big) +  Sq^8\big(x_1^{5}x_2^{15}x_3^{23}x_4^{27}x_5^{30}\\ 
&\quad + x_1^{5}x_2^{15}x_3^{23}x_4^{30}x_5^{27} + x_1^{5}x_2^{15}x_3^{27}x_4^{23}x_5^{30} + x_1^{5}x_2^{15}x_3^{27}x_4^{30}x_5^{23} + x_1^{5}x_2^{27}x_3^{15}x_4^{23}x_5^{30}\\ 
&\quad + x_1^{5}x_2^{27}x_3^{15}x_4^{30}x_5^{23} + x_1^{7}x_2^{13}x_3^{23}x_4^{27}x_5^{30} + x_1^{7}x_2^{13}x_3^{23}x_4^{30}x_5^{27} + x_1^{7}x_2^{13}x_3^{27}x_4^{23}x_5^{30}\\ 
&\quad + x_1^{7}x_2^{13}x_3^{27}x_4^{30}x_5^{23} + x_1^{7}x_2^{25}x_3^{15}x_4^{23}x_5^{30} + x_1^{7}x_2^{25}x_3^{15}x_4^{30}x_5^{23} + x_1^{9}x_2^{23}x_3^{15}x_4^{23}x_5^{30}\\ 
&\quad + x_1^{9}x_2^{23}x_3^{15}x_4^{30}x_5^{23} + x_1^{11}x_2^{13}x_3^{7}x_4^{23}x_5^{46} + x_1^{11}x_2^{13}x_3^{7}x_4^{30}x_5^{39} + x_1^{11}x_2^{13}x_3^{7}x_4^{39}x_5^{30}\\ 
&\quad + x_1^{11}x_2^{13}x_3^{7}x_4^{46}x_5^{23} + x_1^{11}x_2^{15}x_3^{5}x_4^{23}x_5^{46} + x_1^{11}x_2^{15}x_3^{5}x_4^{30}x_5^{39} + x_1^{11}x_2^{15}x_3^{5}x_4^{39}x_5^{30}\\ 
&\quad + x_1^{11}x_2^{15}x_3^{5}x_4^{46}x_5^{23} + x_1^{11}x_2^{21}x_3^{15}x_4^{23}x_5^{30} + x_1^{11}x_2^{21}x_3^{15}x_4^{30}x_5^{23} + x_1^{23}x_2^{13}x_3^{7}x_4^{27}x_5^{30}\\ 
&\quad + x_1^{23}x_2^{13}x_3^{7}x_4^{30}x_5^{27} + x_1^{23}x_2^{13}x_3^{11}x_4^{23}x_5^{30} + x_1^{23}x_2^{13}x_3^{11}x_4^{30}x_5^{23} + x_1^{23}x_2^{15}x_3^{5}x_4^{27}x_5^{30}\\ 
&\quad + x_1^{23}x_2^{15}x_3^{5}x_4^{30}x_5^{27} + x_1^{23}x_2^{15}x_3^{9}x_4^{23}x_5^{30} + x_1^{23}x_2^{15}x_3^{9}x_4^{30}x_5^{23}\big)\\ 
&\quad +  Sq^{16}\big(x_1^{15}x_2^{13}x_3^{7}x_4^{27}x_5^{30} + x_1^{15}x_2^{13}x_3^{7}x_4^{30}x_5^{27} + x_1^{15}x_2^{13}x_3^{11}x_4^{23}x_5^{30}\\ 
&\quad + x_1^{15}x_2^{13}x_3^{11}x_4^{30}x_5^{23} + x_1^{15}x_2^{15}x_3^{5}x_4^{27}x_5^{30} + x_1^{15}x_2^{15}x_3^{5}x_4^{30}x_5^{27}\\ 
&\quad + x_1^{15}x_2^{15}x_3^{9}x_4^{23}x_5^{30} + x_1^{15}x_2^{15}x_3^{9}x_4^{30}x_5^{23}\big) \ \mbox{ mod}(P_5^-((4)|^3|(2)|^2|(1)),\\
U_3 &= x_1^{19}x_2^{15}x_3^{7}x_4^{23}x_5^{44} + x_1^{19}x_2^{15}x_3^{7}x_4^{29}x_5^{38} + x_1^{19}x_2^{15}x_3^{7}x_4^{39}x_5^{28} + x_1^{19}x_2^{15}x_3^{7}x_4^{45}x_5^{22}\\ 
&\quad + x_1^{19}x_2^{15}x_3^{13}x_4^{23}x_5^{38} + x_1^{19}x_2^{15}x_3^{13}x_4^{39}x_5^{22} + x_1^{19}x_2^{15}x_3^{22}x_4^{39}x_5^{13} + x_1^{19}x_2^{15}x_3^{22}x_4^{45}x_5^{7}\\ 
&\quad + x_1^{19}x_2^{15}x_3^{28}x_4^{39}x_5^{7} + x_1^{19}x_2^{15}x_3^{38}x_4^{23}x_5^{13} + x_1^{19}x_2^{15}x_3^{38}x_4^{29}x_5^{7} + x_1^{19}x_2^{15}x_3^{44}x_4^{23}x_5^{7}\\ 
&\quad + x_1^{23}x_2^{11}x_3^{7}x_4^{23}x_5^{44} + x_1^{23}x_2^{11}x_3^{7}x_4^{29}x_5^{38} + x_1^{23}x_2^{11}x_3^{7}x_4^{39}x_5^{28} + x_1^{23}x_2^{11}x_3^{7}x_4^{45}x_5^{22}\\ 
&\quad + x_1^{23}x_2^{11}x_3^{13}x_4^{23}x_5^{38} + x_1^{23}x_2^{11}x_3^{13}x_4^{39}x_5^{22} + x_1^{23}x_2^{11}x_3^{22}x_4^{39}x_5^{13} + x_1^{23}x_2^{11}x_3^{22}x_4^{45}x_5^{7}\\ 
&\quad + x_1^{23}x_2^{11}x_3^{28}x_4^{39}x_5^{7} + x_1^{23}x_2^{11}x_3^{38}x_4^{23}x_5^{13} + x_1^{23}x_2^{11}x_3^{38}x_4^{29}x_5^{7} + x_1^{23}x_2^{11}x_3^{44}x_4^{23}x_5^{7}\\ 
&\quad + x_1^{31}x_2^{11}x_3^{7}x_4^{23}x_5^{36} + x_1^{31}x_2^{11}x_3^{7}x_4^{37}x_5^{22} + x_1^{31}x_2^{11}x_3^{22}x_4^{37}x_5^{7} + x_1^{31}x_2^{11}x_3^{36}x_4^{23}x_5^{7}\\ 
&\quad + x_1^{31}x_2^{15}x_3^{7}x_4^{23}x_5^{32} + x_1^{31}x_2^{15}x_3^{7}x_4^{33}x_5^{22} + x_1^{31}x_2^{15}x_3^{22}x_4^{33}x_5^{7}\\ 
&\quad +  Sq^1\big(x_1^{31}x_2^{11}x_3^{21}x_4^{23}x_5^{21} + x_1^{31}x_2^{15}x_3^{17}x_4^{23}x_5^{21} + x_1^{31}x_2^{15}x_3^{19}x_4^{23}x_5^{19}\\ 
&\quad + x_1^{31}x_2^{15}x_3^{21}x_4^{23}x_5^{17}\big) +  Sq^2\big(x_1^{31}x_2^{15}x_3^{7}x_4^{27}x_5^{26} + x_1^{31}x_2^{15}x_3^{11}x_4^{23}x_5^{26}\\ 
&\quad + x_1^{31}x_2^{15}x_3^{11}x_4^{27}x_5^{22} + x_1^{31}x_2^{15}x_3^{18}x_4^{23}x_5^{19} + x_1^{31}x_2^{15}x_3^{19}x_4^{23}x_5^{18} + x_1^{31}x_2^{15}x_3^{22}x_4^{27}x_5^{11}\\ 
&\quad + x_1^{31}x_2^{15}x_3^{26}x_4^{23}x_5^{11} + x_1^{31}x_2^{15}x_3^{26}x_4^{27}x_5^{7}\big) +  Sq^4\big(x_1^{31}x_2^{15}x_3^{7}x_4^{23}x_5^{28}\\ 
&\quad + x_1^{31}x_2^{15}x_3^{7}x_4^{29}x_5^{22} + x_1^{31}x_2^{15}x_3^{13}x_4^{23}x_5^{22} + x_1^{31}x_2^{15}x_3^{22}x_4^{23}x_5^{13} + x_1^{31}x_2^{15}x_3^{22}x_4^{29}x_5^{7}\\ 
&\quad + x_1^{31}x_2^{15}x_3^{28}x_4^{23}x_5^{7}\big) +  Sq^8\big(x_1^{31}x_2^{11}x_3^{7}x_4^{23}x_5^{28} + x_1^{31}x_2^{11}x_3^{7}x_4^{29}x_5^{22}\\ 
&\quad + x_1^{31}x_2^{11}x_3^{13}x_4^{23}x_5^{22} + x_1^{31}x_2^{11}x_3^{22}x_4^{23}x_5^{13} + x_1^{31}x_2^{11}x_3^{22}x_4^{29}x_5^{7} + x_1^{31}x_2^{11}x_3^{28}x_4^{23}x_5^{7}\big)\\ 
&\quad +  Sq^{16}\big(x_1^{19}x_2^{15}x_3^{7}x_4^{23}x_5^{28} + x_1^{19}x_2^{15}x_3^{7}x_4^{29}x_5^{22} + x_1^{19}x_2^{15}x_3^{13}x_4^{23}x_5^{22}\\ 
&\quad + x_1^{19}x_2^{15}x_3^{22}x_4^{23}x_5^{13} + x_1^{19}x_2^{15}x_3^{22}x_4^{29}x_5^{7} + x_1^{19}x_2^{15}x_3^{28}x_4^{23}x_5^{7} + x_1^{23}x_2^{11}x_3^{7}x_4^{23}x_5^{28}\\ 
&\quad + x_1^{23}x_2^{11}x_3^{7}x_4^{29}x_5^{22} + x_1^{23}x_2^{11}x_3^{13}x_4^{23}x_5^{22} + x_1^{23}x_2^{11}x_3^{22}x_4^{23}x_5^{13}\\ 
&\quad + x_1^{23}x_2^{11}x_3^{22}x_4^{29}x_5^{7} + x_1^{23}x_2^{11}x_3^{28}x_4^{23}x_5^{7}\big) \ \mbox{ mod}(P_5^-((4)|^3|(2)|^2|(1)).
\end{align*}
Thus, the monomials $U_1, U_2,\, U_3$ are strictly inadmissible.

We prove Part ii) for $T_{1} =  x_1^{7}x_2^{15}x_3^{23}x_4^{15}x_5^{48}$, $T_{2} =  x_1^{7}x_2^{15}x_3^{31}x_4^{33}x_5^{22}$, $T_{3} =  x_1^{15}x_2^{31}x_3^{33}x_4^{6}x_5^{23}$, $T_{4} =  x_1^{31}x_2^{15}x_3^{35}x_4^{5}x_5^{22} $,  $T_{5} =  x_1^{31}x_2^{39}x_3^{11}x_4^{5}x_5^{22}$. A direct computation using the Cartan formula shows
\begin{align*}
T_1 &= x_1^{5}x_2^{15}x_3^{22}x_4^{15}x_5^{51} + x_1^{5}x_2^{15}x_3^{22}x_4^{19}x_5^{47} + x_1^{5}x_2^{19}x_3^{22}x_4^{15}x_5^{47} + x_1^{6}x_2^{15}x_3^{23}x_4^{15}x_5^{49}\\ 
&\quad + x_1^{6}x_2^{15}x_3^{23}x_4^{17}x_5^{47} + x_1^{6}x_2^{17}x_3^{23}x_4^{15}x_5^{47} + x_1^{7}x_2^{8}x_3^{23}x_4^{15}x_5^{55} + x_1^{7}x_2^{8}x_3^{23}x_4^{23}x_5^{47}\\ 
&\quad + x_1^{7}x_2^{9}x_3^{22}x_4^{15}x_5^{55} + x_1^{7}x_2^{9}x_3^{22}x_4^{23}x_5^{47} + x_1^{7}x_2^{15}x_3^{15}x_4^{16}x_5^{55} + x_1^{7}x_2^{15}x_3^{22}x_4^{15}x_5^{49}\\ 
&\quad + x_1^{7}x_2^{15}x_3^{22}x_4^{17}x_5^{47} + x_1^{7}x_2^{15}x_3^{23}x_4^{8}x_5^{55} +  Sq^1\big(x_1^{5}x_2^{15}x_3^{25}x_4^{15}x_5^{47}\\ 
&\quad + x_1^{7}x_2^{15}x_3^{23}x_4^{15}x_5^{47}\big) +  Sq^2(x_1^{7}x_2^{15}x_3^{22}x_4^{15}x_5^{47}\big) +  Sq^4\big(x_1^{5}x_2^{15}x_3^{22}x_4^{15}x_5^{47}\big)\\ 
&\quad +  Sq^8\big(x_1^{7}x_2^{8}x_3^{23}x_4^{15}x_5^{47} + x_1^{7}x_2^{9}x_3^{22}x_4^{15}x_5^{47} + x_1^{7}x_2^{23}x_3^{15}x_4^{8}x_5^{47}\big)\\ 
&\quad +  Sq^{16}\big(x_1^{7}x_2^{15}x_3^{15}x_4^{8}x_5^{47}\big) \ \mbox{ mod}(P_5^-((4)|^3|(2)|^2|(1)),\\
T_2 &= x_1^{5}x_2^{11}x_3^{31}x_4^{39}x_5^{22} + x_1^{5}x_2^{11}x_3^{39}x_4^{31}x_5^{22} + x_1^{5}x_2^{15}x_3^{31}x_4^{35}x_5^{22} + x_1^{5}x_2^{15}x_3^{35}x_4^{31}x_5^{22}\\ 
&\quad + x_1^{6}x_2^{9}x_3^{31}x_4^{39}x_5^{23} + x_1^{6}x_2^{9}x_3^{39}x_4^{31}x_5^{23} + x_1^{6}x_2^{15}x_3^{31}x_4^{33}x_5^{23} + x_1^{6}x_2^{15}x_3^{33}x_4^{31}x_5^{23}\\ 
&\quad + x_1^{7}x_2^{8}x_3^{31}x_4^{39}x_5^{23} + x_1^{7}x_2^{8}x_3^{39}x_4^{31}x_5^{23} + x_1^{7}x_2^{9}x_3^{31}x_4^{39}x_5^{22} + x_1^{7}x_2^{9}x_3^{39}x_4^{31}x_5^{22}\\ 
&\quad + x_1^{7}x_2^{15}x_3^{16}x_4^{31}x_5^{39} + x_1^{7}x_2^{15}x_3^{16}x_4^{47}x_5^{23} + x_1^{7}x_2^{15}x_3^{17}x_4^{31}x_5^{38} + x_1^{7}x_2^{15}x_3^{17}x_4^{47}x_5^{22}\\ 
&\quad + x_1^{7}x_2^{15}x_3^{31}x_4^{32}x_5^{23} +  Sq^1\big(x_1^{5}x_2^{15}x_3^{31}x_4^{31}x_5^{25} + x_1^{7}x_2^{15}x_3^{31}x_4^{31}x_5^{23}\big)\\ 
&\quad +  Sq^2\big(x_1^{6}x_2^{15}x_3^{31}x_4^{31}x_5^{23} + x_1^{7}x_2^{15}x_3^{31}x_4^{31}x_5^{22}\big) +  Sq^4\big(x_1^{5}x_2^{15}x_3^{31}x_4^{31}x_5^{22}\big)\\ 
&\quad +  Sq^8x_1^{5}x_2^{11}x_3^{31}x_4^{31}x_5^{22} + x_1^{6}x_2^{9}x_3^{31}x_4^{31}x_5^{23} + x_1^{7}x_2^{8}x_3^{31}x_4^{31}x_5^{23} + x_1^{7}x_2^{9}x_3^{31}x_4^{31}x_5^{22}\big)\\ 
&\quad +  Sq^{16}\big(x_1^{7}x_2^{15}x_3^{16}x_4^{31}x_5^{23} + x_1^{7}x_2^{15}x_3^{17}x_4^{31}x_5^{22}\big) \ \mbox{ mod}(P_5^-((4)|^3|(2)|^2|(1)),\\
T_3 &= x_1^{5}x_2^{31}x_3^{15}x_4^{22}x_5^{35} + x_1^{5}x_2^{31}x_3^{15}x_4^{34}x_5^{23} + x_1^{5}x_2^{31}x_3^{35}x_4^{14}x_5^{23} + x_1^{5}x_2^{39}x_3^{15}x_4^{22}x_5^{27}\\ 
&\quad + x_1^{5}x_2^{39}x_3^{15}x_4^{26}x_5^{23} + x_1^{5}x_2^{39}x_3^{27}x_4^{14}x_5^{23} + x_1^{7}x_2^{31}x_3^{15}x_4^{20}x_5^{35} + x_1^{7}x_2^{31}x_3^{15}x_4^{32}x_5^{23}\\ 
&\quad + x_1^{7}x_2^{31}x_3^{35}x_4^{12}x_5^{23} + x_1^{7}x_2^{39}x_3^{15}x_4^{20}x_5^{27} + x_1^{7}x_2^{39}x_3^{15}x_4^{24}x_5^{23} + x_1^{7}x_2^{39}x_3^{27}x_4^{12}x_5^{23}\\ 
&\quad + x_1^{9}x_2^{31}x_3^{23}x_4^{7}x_5^{38} + x_1^{9}x_2^{39}x_3^{23}x_4^{7}x_5^{30} + x_1^{9}x_2^{39}x_3^{23}x_4^{14}x_5^{23} + x_1^{11}x_2^{31}x_3^{7}x_4^{22}x_5^{37}\\ 
&\quad + x_1^{11}x_2^{31}x_3^{7}x_4^{36}x_5^{23} + x_1^{11}x_2^{31}x_3^{21}x_4^{7}x_5^{38} + x_1^{11}x_2^{31}x_3^{23}x_4^{6}x_5^{37} + x_1^{11}x_2^{31}x_3^{37}x_4^{6}x_5^{23}\\ 
&\quad + x_1^{11}x_2^{39}x_3^{7}x_4^{22}x_5^{29} + x_1^{11}x_2^{39}x_3^{7}x_4^{28}x_5^{23} + x_1^{11}x_2^{39}x_3^{13}x_4^{22}x_5^{23} + x_1^{11}x_2^{39}x_3^{21}x_4^{7}x_5^{30}\\ 
&\quad + x_1^{11}x_2^{39}x_3^{21}x_4^{14}x_5^{23} + x_1^{11}x_2^{39}x_3^{23}x_4^{6}x_5^{29} + x_1^{11}x_2^{39}x_3^{23}x_4^{12}x_5^{23} + x_1^{11}x_2^{39}x_3^{29}x_4^{6}x_5^{23}\\ 
&\quad + x_1^{15}x_2^{17}x_3^{23}x_4^{7}x_5^{46} + x_1^{15}x_2^{17}x_3^{23}x_4^{14}x_5^{39} + x_1^{15}x_2^{17}x_3^{39}x_4^{7}x_5^{30} + x_1^{15}x_2^{17}x_3^{39}x_4^{14}x_5^{23}\\ 
&\quad + x_1^{15}x_2^{19}x_3^{7}x_4^{22}x_5^{45} + x_1^{15}x_2^{19}x_3^{7}x_4^{28}x_5^{39} + x_1^{15}x_2^{19}x_3^{7}x_4^{38}x_5^{29} + x_1^{15}x_2^{19}x_3^{7}x_4^{44}x_5^{23}\\ 
&\quad + x_1^{15}x_2^{19}x_3^{13}x_4^{22}x_5^{39} + x_1^{15}x_2^{19}x_3^{13}x_4^{38}x_5^{23} + x_1^{15}x_2^{19}x_3^{21}x_4^{7}x_5^{46} + x_1^{15}x_2^{19}x_3^{21}x_4^{14}x_5^{39}\\ 
&\quad + x_1^{15}x_2^{19}x_3^{23}x_4^{6}x_5^{45} + x_1^{15}x_2^{19}x_3^{23}x_4^{12}x_5^{39} + x_1^{15}x_2^{19}x_3^{29}x_4^{6}x_5^{39} + x_1^{15}x_2^{19}x_3^{37}x_4^{7}x_5^{30}\\ 
&\quad + x_1^{15}x_2^{19}x_3^{37}x_4^{14}x_5^{23} + x_1^{15}x_2^{19}x_3^{39}x_4^{6}x_5^{29} + x_1^{15}x_2^{19}x_3^{39}x_4^{12}x_5^{23} + x_1^{15}x_2^{19}x_3^{45}x_4^{6}x_5^{23}\\ 
&\quad + x_1^{15}x_2^{31}x_3^{7}x_4^{22}x_5^{33} + x_1^{15}x_2^{31}x_3^{7}x_4^{32}x_5^{23} + x_1^{15}x_2^{31}x_3^{21}x_4^{7}x_5^{34} + x_1^{15}x_2^{31}x_3^{23}x_4^{6}x_5^{33}\\ 
&\quad + x_1^{15}x_2^{31}x_3^{23}x_4^{7}x_5^{32} +  Sq^1\big(x_1^{15}x_2^{31}x_3^{17}x_4^{21}x_5^{23} + x_1^{15}x_2^{31}x_3^{19}x_4^{13}x_5^{29}\\ 
&\quad + x_1^{15}x_2^{31}x_3^{19}x_4^{19}x_5^{23} + x_1^{15}x_2^{31}x_3^{23}x_4^{9}x_5^{29} + x_1^{15}x_2^{31}x_3^{23}x_4^{11}x_5^{27} + x_1^{15}x_2^{31}x_3^{23}x_4^{13}x_5^{25}\big)\\ 
&\quad +  Sq^2\big(x_1^{15}x_2^{31}x_3^{7}x_4^{26}x_5^{27} + x_1^{15}x_2^{31}x_3^{11}x_4^{22}x_5^{27} + x_1^{15}x_2^{31}x_3^{11}x_4^{26}x_5^{23}\\ 
&\quad + x_1^{15}x_2^{31}x_3^{18}x_4^{19}x_5^{23} + x_1^{15}x_2^{31}x_3^{19}x_4^{11}x_5^{30} + x_1^{15}x_2^{31}x_3^{19}x_4^{14}x_5^{27} + x_1^{15}x_2^{31}x_3^{19}x_4^{18}x_5^{23}\\ 
&\quad + x_1^{15}x_2^{31}x_3^{23}x_4^{7}x_5^{30} + x_1^{15}x_2^{31}x_3^{23}x_4^{10}x_5^{27} + x_1^{15}x_2^{31}x_3^{23}x_4^{10}x_5^{27} + x_1^{15}x_2^{31}x_3^{23}x_4^{11}x_5^{26}\\ 
&\quad + x_1^{15}x_2^{31}x_3^{23}x_4^{14}x_5^{23} + x_1^{15}x_2^{31}x_3^{27}x_4^{6}x_5^{27} + x_1^{15}x_2^{31}x_3^{27}x_4^{10}x_5^{23}\big)\\ 
&\quad +  Sq^4\big(x_1^{5}x_2^{31}x_3^{23}x_4^{22}x_5^{23} + x_1^{7}x_2^{31}x_3^{23}x_4^{20}x_5^{23} + x_1^{15}x_2^{31}x_3^{7}x_4^{22}x_5^{29}\\ 
&\quad + x_1^{15}x_2^{31}x_3^{7}x_4^{28}x_5^{23} + x_1^{15}x_2^{31}x_3^{13}x_4^{22}x_5^{23} + x_1^{15}x_2^{31}x_3^{21}x_4^{7}x_5^{30} + x_1^{15}x_2^{31}x_3^{21}x_4^{14}x_5^{23}\\ 
&\quad + x_1^{15}x_2^{31}x_3^{23}x_4^{6}x_5^{29} + x_1^{15}x_2^{31}x_3^{23}x_4^{12}x_5^{23} + x_1^{15}x_2^{31}x_3^{29}x_4^{6}x_5^{23}\big)\\ 
&\quad +  Sq^8\big(x_1^{5}x_2^{31}x_3^{15}x_4^{22}x_5^{27} + x_1^{5}x_2^{31}x_3^{15}x_4^{26}x_5^{23} + x_1^{5}x_2^{31}x_3^{27}x_4^{14}x_5^{23} + x_1^{7}x_2^{31}x_3^{15}x_4^{20}x_5^{27}\\ 
&\quad + x_1^{7}x_2^{31}x_3^{15}x_4^{24}x_5^{23} + x_1^{7}x_2^{31}x_3^{27}x_4^{12}x_5^{23} + x_1^{9}x_2^{31}x_3^{23}x_4^{7}x_5^{30} + x_1^{9}x_2^{31}x_3^{23}x_4^{14}x_5^{23}\\ 
&\quad + x_1^{11}x_2^{31}x_3^{7}x_4^{22}x_5^{29} + x_1^{11}x_2^{31}x_3^{7}x_4^{28}x_5^{23} + x_1^{11}x_2^{31}x_3^{13}x_4^{22}x_5^{23} + x_1^{11}x_2^{31}x_3^{21}x_4^{7}x_5^{30}\\ 
&\quad + x_1^{11}x_2^{31}x_3^{21}x_4^{14}x_5^{23} + x_1^{11}x_2^{31}x_3^{23}x_4^{6}x_5^{29} + x_1^{11}x_2^{31}x_3^{23}x_4^{12}x_5^{23} + x_1^{11}x_2^{31}x_3^{29}x_4^{6}x_5^{23}\big)\\ 
&\quad +  Sq^{16}\big(x_1^{15}x_2^{17}x_3^{23}x_4^{7}x_5^{30} + x_1^{15}x_2^{17}x_3^{23}x_4^{14}x_5^{23} + x_1^{15}x_2^{19}x_3^{7}x_4^{22}x_5^{29}\\ 
&\quad + x_1^{15}x_2^{19}x_3^{7}x_4^{28}x_5^{23} + x_1^{15}x_2^{19}x_3^{13}x_4^{22}x_5^{23} + x_1^{15}x_2^{19}x_3^{21}x_4^{7}x_5^{30}\\ 
&\quad + x_1^{15}x_2^{19}x_3^{21}x_4^{14}x_5^{23} + x_1^{15}x_2^{19}x_3^{23}x_4^{6}x_5^{29} + x_1^{15}x_2^{19}x_3^{23}x_4^{12}x_5^{23}\\ 
&\quad + x_1^{15}x_2^{19}x_3^{29}x_4^{6}x_5^{23}\big) \ \mbox{ mod}(P_5^-((4)|^3|(2)|^2|(1)),\\
T_4 & = x_1^{16}x_2^{15}x_3^{31}x_4^{7}x_5^{39} + x_1^{16}x_2^{15}x_3^{47}x_4^{7}x_5^{23} + x_1^{17}x_2^{15}x_3^{31}x_4^{6}x_5^{39} + x_1^{17}x_2^{15}x_3^{31}x_4^{7}x_5^{38}\\ 
&\quad + x_1^{17}x_2^{15}x_3^{47}x_4^{6}x_5^{23} + x_1^{17}x_2^{15}x_3^{47}x_4^{7}x_5^{22} + x_1^{19}x_2^{15}x_3^{31}x_4^{5}x_5^{38} + x_1^{19}x_2^{15}x_3^{47}x_4^{5}x_5^{22}\\ 
&\quad + x_1^{23}x_2^{8}x_3^{31}x_4^{7}x_5^{39} + x_1^{23}x_2^{8}x_3^{47}x_4^{7}x_5^{23} + x_1^{23}x_2^{9}x_3^{31}x_4^{6}x_5^{39} + x_1^{23}x_2^{9}x_3^{31}x_4^{7}x_5^{38}\\ 
&\quad + x_1^{23}x_2^{9}x_3^{47}x_4^{6}x_5^{23} + x_1^{23}x_2^{9}x_3^{47}x_4^{7}x_5^{22} + x_1^{23}x_2^{11}x_3^{31}x_4^{5}x_5^{38} + x_1^{23}x_2^{11}x_3^{47}x_4^{5}x_5^{22}\\ 
&\quad + x_1^{31}x_2^{8}x_3^{39}x_4^{7}x_5^{23} + x_1^{31}x_2^{9}x_3^{39}x_4^{6}x_5^{23} + x_1^{31}x_2^{9}x_3^{39}x_4^{7}x_5^{22} + x_1^{31}x_2^{11}x_3^{39}x_4^{5}x_5^{22}\\ 
&\quad + x_1^{31}x_2^{15}x_3^{32}x_4^{7}x_5^{23} + x_1^{31}x_2^{15}x_3^{33}x_4^{6}x_5^{23} + x_1^{31}x_2^{15}x_3^{33}x_4^{7}x_5^{22}\\ 
&\quad +  Sq^1\big(x_1^{31}x_2^{15}x_3^{31}x_4^{5}x_5^{25} + x_1^{31}x_2^{15}x_3^{31}x_4^{7}x_5^{23}\big) +  Sq^2\big(x_1^{31}x_2^{15}x_3^{31}x_4^{6}x_5^{23}\\ 
&\quad + x_1^{31}x_2^{15}x_3^{31}x_4^{7}x_5^{22}\big) +  sq^4\big(x_1^{31}x_2^{15}x_3^{31}x_4^{5}x_5^{22}\big) +  Sq^8\big(x_1^{31}x_2^{8}x_3^{31}x_4^{7}x_5^{23}\\ 
&\quad + x_1^{31}x_2^{9}x_3^{31}x_4^{6}x_5^{23} + x_1^{31}x_2^{9}x_3^{31}x_4^{7}x_5^{22} + x_1^{31}x_2^{11}x_3^{31}x_4^{5}x_5^{22}\big)\\ 
&\quad +  Sq^{16}\big(x_1^{16}x_2^{15}x_3^{31}x_4^{7}x_5^{23} + x_1^{17}x_2^{15}x_3^{31}x_4^{6}x_5^{23} + x_1^{17}x_2^{15}x_3^{31}x_4^{7}x_5^{22}\\ 
&\quad + x_1^{19}x_2^{15}x_3^{31}x_4^{5}x_5^{22} + x_1^{23}x_2^{8}x_3^{31}x_4^{7}x_5^{23} + x_1^{23}x_2^{9}x_3^{31}x_4^{6}x_5^{23}\\ 
&\quad + x_1^{23}x_2^{9}x_3^{31}x_4^{7}x_5^{22} + x_1^{23}x_2^{11}x_3^{31}x_4^{5}x_5^{22}\big) \ \mbox{ mod}(P_5^-((4)|^3|(2)|^2|(1)),\\
T_5 &= x_1^{16}x_2^{31}x_3^{15}x_4^{7}x_5^{39} + x_1^{16}x_2^{47}x_3^{15}x_4^{7}x_5^{23} + x_1^{17}x_2^{31}x_3^{15}x_4^{6}x_5^{39} + x_1^{17}x_2^{31}x_3^{15}x_4^{7}x_5^{38}\\ 
&\quad + x_1^{17}x_2^{47}x_3^{15}x_4^{6}x_5^{23} + x_1^{17}x_2^{47}x_3^{15}x_4^{7}x_5^{22} + x_1^{19}x_2^{31}x_3^{15}x_4^{5}x_5^{38} + x_1^{19}x_2^{47}x_3^{15}x_4^{5}x_5^{22}\\ 
&\quad + x_1^{23}x_2^{31}x_3^{8}x_4^{7}x_5^{39} + x_1^{23}x_2^{31}x_3^{9}x_4^{6}x_5^{39} + x_1^{23}x_2^{31}x_3^{9}x_4^{7}x_5^{38} + x_1^{23}x_2^{31}x_3^{11}x_4^{5}x_5^{38}\\ 
&\quad + x_1^{23}x_2^{47}x_3^{8}x_4^{7}x_5^{23} + x_1^{23}x_2^{47}x_3^{9}x_4^{6}x_5^{23} + x_1^{23}x_2^{47}x_3^{9}x_4^{7}x_5^{22} + x_1^{23}x_2^{47}x_3^{11}x_4^{5}x_5^{22}\\ 
&\quad + x_1^{31}x_2^{32}x_3^{15}x_4^{7}x_5^{23} + x_1^{31}x_2^{33}x_3^{15}x_4^{6}x_5^{23} + x_1^{31}x_2^{33}x_3^{15}x_4^{7}x_5^{22} + x_1^{31}x_2^{35}x_3^{15}x_4^{5}x_5^{22}\\ 
&\quad + x_1^{31}x_2^{39}x_3^{8}x_4^{7}x_5^{23} + x_1^{31}x_2^{39}x_3^{9}x_4^{6}x_5^{23} + x_1^{31}x_2^{39}x_3^{9}x_4^{7}x_5^{22}\\ 
&\quad +  Sq^1\big(x_1^{31}x_2^{31}x_3^{15}x_4^{5}x_5^{25} + x_1^{31}x_2^{31}x_3^{15}x_4^{7}x_5^{23}\big) +  Sq^2\big(x_1^{31}x_2^{31}x_3^{15}x_4^{6}x_5^{23}\\ 
&\quad + x_1^{31}x_2^{31}x_3^{15}x_4^{7}x_5^{22}\big) +  Sq^4\big(x_1^{31}x_2^{31}x_3^{15}x_4^{5}x_5^{22}\big) +  Sq^8\big(x_1^{31}x_2^{31}x_3^{8}x_4^{7}x_5^{23}\\ 
&\quad + x_1^{31}x_2^{31}x_3^{9}x_4^{6}x_5^{23} + x_1^{31}x_2^{31}x_3^{9}x_4^{7}x_5^{22} + x_1^{31}x_2^{31}x_3^{11}x_4^{5}x_5^{22}\big)\\ 
&\quad +  Sq^{16}\big(x_1^{16}x_2^{31}x_3^{15}x_4^{7}x_5^{23} + x_1^{17}x_2^{31}x_3^{15}x_4^{6}x_5^{23} + x_1^{17}x_2^{31}x_3^{15}x_4^{7}x_5^{22}\\ 
&\quad + x_1^{19}x_2^{31}x_3^{15}x_4^{5}x_5^{22} + x_1^{23}x_2^{31}x_3^{8}x_4^{7}x_5^{23} + x_1^{23}x_2^{31}x_3^{9}x_4^{6}x_5^{23}\\ 
&\quad + x_1^{23}x_2^{31}x_3^{9}x_4^{7}x_5^{22} + x_1^{23}x_2^{31}x_3^{11}x_4^{5}x_5^{22}\big) \ \mbox{ mod}(P_5^-((4)|^3|(2)|^2|(1)).
\end{align*}
Hence, the monomials $T_j,\, 1\leqslant j \leqslant 5$, are strictly inadmissible.
\end{proof}
\begin{proof}[Proof of Proposition \ref{mdc}]
By a simple computation we have 
$$\Phi^+(B_4((4)|^3|(2)|^2|(1)))\cup \mathcal C = \{a_s : 1 \leqslant s \leqslant 1470\},$$ 
where the monomials $a_s$ are listed as in the Subsection \ref{s51}.

Let $x$ be an admissible monomial of weight vector $(4)|^3|(2)|^2|(1)$ and $x \in P_5^+$. By a direct computation we see that if $x \notin \Phi^+(B_4((4)|^3|(2)|^2|(1)))\cup \mathcal C$, then either $x$ is one of the monomials as listed in Lemma \ref{bd6} or there is a monomials as listed in one of Lemmas \ref{bd2}(i), \ref{bd3}, \ref{bd4}, \ref{bd5} such that $x = yw^{2^r}z^{2^{r+u}}$ with $1\leqslant r \leqslant 5$, $u \geqslant 0$, $ r+u \leqslant 5$ and $y$, $z$ suitable monomials in $P_5$. By Theorem \ref{dlcb1}, $x$ is inadmissible and we have a contradiction. Hence, 
$$B_5^+((4)|^3(2)|^2|(1)) \subset \Phi^+(B_4((4)|^3|(2)|^2|(1)))\cup \mathcal C.$$ 
The proposition is proved. 
\end{proof}

%===============================

\section{Proof of Theorem \ref{thm3}}\label{s4}

From the proof of Proposition 3.3 in \cite{su2} we get $\widetilde {SF}_{5}((4)|^4|(1)|^2) = 0$ and $\widetilde {SF}_{5}((4)|^4|(3)) = 0$. Hence, we need only to determined $\widetilde {SF}_{5}((4)|^3|(2)|^2|(1)) \subset QP_5^+((4)|^3|(2)|^2|(1)).$ 

We can easily verify that $[p_j] \in \widetilde {SF}_{5}((4)|^3|(2)|^2|(1))$ for $1 \leqslant j \leqslant 14$. The leading monomial of $p_j$ is $a_{1456+j}$. 
Since $B_5^+((4)|^3|(2)|^2|(1)) \subset \{[a_t] :1\leqslant t \leqslant 1470\}$, if $\xi \in \widetilde {SF}_{5}((4)|^3|(2)|^2|(1))$, then there are $\delta_j \in \mathbb F_2$, $1 \leqslant j \leqslant 14$, such that
\begin{equation}\label{ctd612}
\theta = \xi + \sum_{1\leqslant j \leqslant 14}\delta_jp_{j} \equiv \sum_{1\leqslant t \leqslant 1456}\gamma_ta_{t},
\end{equation}
where $\gamma_t \in \mathbb F_2$, $1\leqslant t \leqslant 1470$. Then, we have $\theta \in \widetilde {SF}_{5}((4)|^3|(2)|^2|(1))$ and $p_{(i;I)}(\theta) \equiv 0$ for all $(i;I) \in \mathcal N_5$. 

By a direct computation from the relations $p_{(i;I)}(\theta) \equiv 0$ for all $(i;I) \in \mathcal N_5$ we obtain $\gamma_t = 0$ for $1\leqslant t \leqslant 1456$. Hence, $\theta \equiv 0$ and $\xi \equiv \sum_{1\leqslant j \leqslant 14}\delta_jp_{j}$. This implies $$(\widetilde {SF}_{5})_{108} = \widetilde {SF}_{5}((4)|^3|(2)|^2|(1)) = \langle [p_{j}] : 1\leqslant j \leqslant 14\rangle$$ 
and $\widetilde {QP}_5^+((4)|^3|(2)|^2|(1)) \cong \langle [a_{t}] : 1\leqslant t \leqslant 1456\rangle.$ So, we get
$$\dim (\widetilde {QP}_5)_{108} = 124 + 310 +280 + 1456 = 2170.$$

We can easily verify that $p_{(i;I)}(\rho_u(p_j)) \equiv 0$ for all $(i;I) \in \mathcal N_5$. Hence $\rho_u([p_j]) \in (\widetilde {SF}_{5})_{108}$ for $1\leqslant u \leqslant 5$ and $1\leqslant j \leqslant 14$. This means that $(\widetilde {SF}_{5})_{108}$ is an $GL_5$-submodule of $(QP_5)_{108}$. 

Below are the detailed calculations. 

Let $w_s,\, 1\leqslant s \leqslant 56$, be as in Subsection \ref{s51} and the homomorphism $p_{(i;I)}:P_5\to P_4$ which is defined by \eqref{ct23} for $k=5$. For simplicity, we denote $\gamma_{\mathbb J} = \sum_{t \in \mathbb J}\gamma_t$ for any $\mathbb J \subset \{i\in \mathbb N:1\leqslant i \leqslant 1456\}$. 

By applying $p_{(1;j)}$, $2 \leqslant j \leqslant 5,$ to (\ref{ctd612}) and using Theorem \ref{dlsig}, we obtain 
\begin{align*}
p_{(1;2)}(\theta) &\equiv \gamma_{1}w_{1} + \gamma_{2}w_{2} + \gamma_{3}w_{3} + \gamma_{331}w_{4} + \gamma_{332}w_{5} + \gamma_{4}w_{6} + \gamma_{5}w_{7} + \gamma_{333}w_{8}\\
&\quad + \gamma_{334}w_{9} + \gamma_{335}w_{10} + \gamma_{6}w_{11} + \gamma_{7}w_{12} + \gamma_{8}w_{13} + \gamma_{9}w_{14} + \gamma_{28}w_{15}\\
&\quad + \gamma_{360}w_{16} + \gamma_{361}w_{17} + \gamma_{29}w_{18} + \gamma_{362}w_{19} + \gamma_{363}w_{20} + \gamma_{30}w_{21}\\
&\quad + \gamma_{364}w_{22} + \gamma_{365}w_{23} + \gamma_{366}w_{24} + \gamma_{367}w_{25} + \gamma_{31}w_{26} + \gamma_{368}w_{27}\\
&\quad + \gamma_{369}w_{28} + \gamma_{370}w_{29} + \gamma_{371}w_{30} + \gamma_{372}w_{31} + \gamma_{373}w_{32} + \gamma_{374}w_{33}\\
&\quad + \gamma_{32}w_{34} + \gamma_{33}w_{35} + \gamma_{46}w_{36} + \gamma_{426}w_{37} + \gamma_{427}w_{38} + \gamma_{428}w_{39}\\
&\quad + \gamma_{47}w_{40} + \gamma_{429}w_{41} + \gamma_{430}w_{42} + \gamma_{431}w_{43} + \gamma_{432}w_{44} + \gamma_{433}w_{45}\\
&\quad + \gamma_{434}w_{46} + \gamma_{435}w_{47} + \gamma_{436}w_{48} + \gamma_{437}w_{49} + \gamma_{48}w_{50} + \gamma_{477}w_{51}\\
&\quad + \gamma_{478}w_{52} + \gamma_{479}w_{53} + \gamma_{480}w_{54} + \gamma_{481}w_{55} + \gamma_{482}w_{56}  \equiv 0,\\ 
p_{(1;3)}(\theta) &\equiv \gamma_{10}w_{1} + \gamma_{11}w_{2} + \gamma_{14}w_{3} + \gamma_{338}w_{4} + \gamma_{339}w_{5} + \gamma_{15}w_{6} + \gamma_{18}w_{7}\\
&\quad + \gamma_{346}w_{8} + \gamma_{347}w_{9} + \gamma_{348}w_{10} + \gamma_{19}w_{11} + \gamma_{357}w_{12} + \gamma_{358}w_{13}\\
&\quad + \gamma_{359}w_{14} + \gamma_{34}w_{15} + \gamma_{375}w_{16} + \gamma_{376}w_{17} + \gamma_{35}w_{18} + \gamma_{383}w_{19}\\
&\quad + \gamma_{384}w_{20} + \gamma_{38}w_{21} + \gamma_{389}w_{22} + \gamma_{390}w_{23} + \gamma_{391}w_{24} + \gamma_{392}w_{25}\\
&\quad + \gamma_{39}w_{26} + \gamma_{403}w_{27} + \gamma_{404}w_{28} + \gamma_{405}w_{29} + \gamma_{406}w_{30} + \gamma_{415}w_{31}\\
&\quad + \gamma_{416}w_{32} + \gamma_{417}w_{33} + \gamma_{424}w_{34} + \gamma_{425}w_{35} + \gamma_{49}w_{36} + \gamma_{438}w_{37}\\
&\quad + \gamma_{439}w_{38} + \gamma_{440}w_{39} + \gamma_{50}w_{40} + \gamma_{449}w_{41} + \gamma_{450}w_{42} + \gamma_{451}w_{43}\\
&\quad + \gamma_{452}w_{44} + \gamma_{461}w_{45} + \gamma_{462}w_{46} + \gamma_{467}w_{47} + \gamma_{470}w_{48} + \gamma_{471}w_{49}\\
&\quad + \gamma_{476}w_{50} + \gamma_{55}w_{51} + \gamma_{56}w_{52} + \gamma_{57}w_{53} + \gamma_{64}w_{54} + \gamma_{65}w_{55}\\
&\quad + \gamma_{70}w_{56} \equiv 0,\\
p_{(1;4)}(\theta) &\equiv \gamma_{\{12,174\}}w_{1} + \gamma_{\{337,810\}}w_{2} + \gamma_{16}w_{3} + \gamma_{\{341,1179\}}w_{4} + \gamma_{343}w_{5}\\
&\quad + \gamma_{\{345,1181\}}w_{6} + \gamma_{20}w_{7} + \gamma_{350}w_{8} + \gamma_{352}w_{9} + \gamma_{354}w_{10} + \gamma_{356}w_{11}\\
&\quad + \gamma_{22}w_{12} + \gamma_{24}w_{13} + \gamma_{26}w_{14} + \gamma_{36}w_{15} + \gamma_{\{378,1215\}}w_{16} + \gamma_{380}w_{17}\\
&\quad + \gamma_{\{382,1217\}}w_{18} + \gamma_{\{385,931\}}w_{19} + \gamma_{\{387,933\}}w_{20} + \gamma_{40}w_{21} + \gamma_{394}w_{22}\\
&\quad + \gamma_{396}w_{23} + \gamma_{398}w_{24} + \gamma_{400}w_{25} + \gamma_{402}w_{26} + \gamma_{407}w_{27} + \gamma_{409}w_{28}\\
&\quad + \gamma_{411}w_{29} + \gamma_{413}w_{30} + \gamma_{418}w_{31} + \gamma_{420}w_{32} + \gamma_{422}w_{33} + \gamma_{42}w_{34}\\
&\quad + \gamma_{44}w_{35} + \gamma_{51}w_{36} + \gamma_{442}w_{37} + \gamma_{444}w_{38} + \gamma_{446}w_{39} + \gamma_{448}w_{40}\\
&\quad + \gamma_{453}w_{41} + \gamma_{455}w_{42} + \gamma_{457}w_{43} + \gamma_{459}w_{44} + \gamma_{463}w_{45} + \gamma_{465}w_{46}\\
&\quad + \gamma_{468}w_{47} + \gamma_{472}w_{48} + \gamma_{474}w_{49} + \gamma_{53}w_{50} + \gamma_{58}w_{51} + \gamma_{60}w_{52}\\
&\quad + \gamma_{62}w_{53} + \gamma_{66}w_{54} + \gamma_{68}w_{55} + \gamma_{71}w_{56} \equiv 0,\\
p_{(1;5)}(\theta) &\equiv \gamma_{\{336,809\}}w_{1} + \gamma_{\{13,175\}}w_{2} + \gamma_{\{340,823\}}w_{3} + \gamma_{342}w_{4} + \gamma_{\{344,826\}}w_{5}\\
&\quad + \gamma_{17}w_{6} + \gamma_{\{349,842\}}w_{7} + \gamma_{351}w_{8} + \gamma_{353}w_{9} + \gamma_{355}w_{10} + \gamma_{21}w_{11}\\
&\quad + \gamma_{\{23,181\}}w_{12} + \gamma_{25}w_{13} + \gamma_{27}w_{14} + \gamma_{\{377,924\}}w_{15} + \gamma_{379}w_{16}\\
&\quad + \gamma_{\{381,927\}}w_{17} + \gamma_{37}w_{18} + \gamma_{386}w_{19} + \gamma_{388}w_{20} + \gamma_{\{393,942\}}w_{21}\\
&\quad + \gamma_{395}w_{22} + \gamma_{397}w_{23} + \gamma_{399}w_{24} + \gamma_{401}w_{25} + \gamma_{41}w_{26} + \gamma_{408}w_{27}\\
&\quad + \gamma_{410}w_{28} + \gamma_{412}w_{29} + \gamma_{414}w_{30} + \gamma_{\{419,956\}}w_{31} + \gamma_{421}w_{32} + \gamma_{423}w_{33}\\
&\quad + \gamma_{43}w_{34} + \gamma_{45}w_{35} + \gamma_{\{441,1017\}}w_{36} + \gamma_{443}w_{37} + \gamma_{445}w_{38}\\
&\quad + \gamma_{447}w_{39} + \gamma_{52}w_{40} + \gamma_{454}w_{41} + \gamma_{456}w_{42} + \gamma_{458}w_{43} + \gamma_{460}w_{44}\\
&\quad + \gamma_{464}w_{45} + \gamma_{466}w_{46} + \gamma_{469}w_{47} + \gamma_{473}w_{48} + \gamma_{475}w_{49} + \gamma_{54}w_{50}\\
&\quad + \gamma_{\{59,209\}}w_{51} + \gamma_{61}w_{52} + \gamma_{63}w_{53} + \gamma_{67}w_{54} + \gamma_{69}w_{55} + \gamma_{72}w_{56} \equiv 0.
\end{align*}	
From above relations we have
\begin{equation}\label{c61}
\gamma_j = 0 \mbox{ for } t \in \mathbb J_1,\ \gamma_i= \gamma_j \mbox{ for } (i,j) \in \mathbb K_1,
\end{equation}
where $\mathbb J_1 = \{$1, 2, 3, 4, 5, 6, 7, 8, 9, 10, 11, 14, 15, 16, 17, 18, 19, 20, 21, 22, 24, 25, 26, 27, 28, 29, 30, 31, 32, 33, 34, 35, 36, 37, 38, 39, 40, 41, 42, 43, 44, 45, 46, 47, 48, 49, 50, 51, 52, 53, 54, 55, 56, 57, 58, 60, 61, 62, 63, 64, 65, 66, 67, 68, 69, 70, 71, 72, 331, 332, 333, 334, 335, 338, 339, 342, 343, 346, 347, 348, 350, 351, 352, 353, 354, 355, 356, 357, 358, 359, 360, 361, 362, 363, 364, 365, 366, 367, 368, 369, 370, 371, 372, 373, 374, 375, 376, 379, 380, 383, 384, 386, 388, 389, 390, 391, 392, 394, 395, 396, 397, 398, 399, 400, 401, 402, 403, 404, 405, 406, 407, 408, 409, 410, 411, 412, 413, 414, 415, 416, 417, 418, 420, 421, 422, 423, 424, 425, 426, 427, 428, 429, 430, 431, 432, 433, 434, 435, 436, 437, 438, 439, 440, 442, 443, 444, 445, 446, 447, 448, 449, 450, 451, 452, 453, 454, 455, 456, 457, 458, 459, 460, 461, 462, 463, 464, 465, 466, 467, 468, 469, 470, 471, 472, 473, 474, 475, 476, 477, 478, 479, 480, 481, 482 $\}$ and $\mathbb K_1 = \{$((12,174), (13,175), (23,181), (59,209), (336,809), (337,810), (340,823), (341,1179), (344,826), (345,1181), (349,842), (377,924), (378,1215), (381,927), (382,1217), (385,931),\linebreak (387,933), (393,942), (419,956), (441,1017)$\}$.

Apply the homomorphisms $p_{(2;3)}$, $p_{(2;4)}$, $p_{(2;5)}$ to \eqref{ctd612} using \eqref{c61}, to obtain
\begin{align*}
p_{(2;3)}(\theta) &\equiv \gamma_{145}w_{1} + \gamma_{146}w_{2} + \gamma_{149}w_{3} + \gamma_{749}w_{4} + \gamma_{750}w_{5} + \gamma_{150}w_{6} + \gamma_{153}w_{7}\\
&\quad + \gamma_{757}w_{8} + \gamma_{758}w_{9} + \gamma_{759}w_{10} + \gamma_{154}w_{11} + \gamma_{768}w_{12} + \gamma_{769}w_{13}\\
&\quad + \gamma_{770}w_{14} + \gamma_{217}w_{15} + \gamma_{1053}w_{16} + \gamma_{1054}w_{17} + \gamma_{218}w_{18} + \gamma_{1061}w_{19}\\
&\quad + \gamma_{1062}w_{20} + \gamma_{221}w_{21} + \gamma_{1067}w_{22} + \gamma_{1068}w_{23} + \gamma_{1069}w_{24} + \gamma_{1070}w_{25}\\
&\quad + \gamma_{222}w_{26} + \gamma_{1081}w_{27} + \gamma_{1082}w_{28} + \gamma_{1083}w_{29} + \gamma_{1084}w_{30} + \gamma_{1093}w_{31}\\
&\quad + \gamma_{1094}w_{32} + \gamma_{1095}w_{33} + \gamma_{1102}w_{34} + \gamma_{1103}w_{35} + \gamma_{249}w_{36} + \gamma_{1308}w_{37}\\
&\quad + \gamma_{1309}w_{38} + \gamma_{1310}w_{39} + \gamma_{250}w_{40} + \gamma_{1319}w_{41} + \gamma_{1320}w_{42} + \gamma_{1321}w_{43}\\
&\quad + \gamma_{1322}w_{44} + \gamma_{1331}w_{45} + \gamma_{1332}w_{46} + \gamma_{1337}w_{47} + \gamma_{1340}w_{48} + \gamma_{1341}w_{49}\\
&\quad + \gamma_{1346}w_{50} + \gamma_{265}w_{51} + \gamma_{266}w_{52} + \gamma_{267}w_{53} + \gamma_{274}w_{54} + \gamma_{275}w_{55}\\
&\quad + \gamma_{280}w_{56} \equiv 0,\\
p_{(2;4)}(\theta) &\equiv \gamma_{\{12,147\}}w_{1} + \gamma_{\{337,748\}}w_{2} + \gamma_{151}w_{3} + \gamma_{\{385,752\}}w_{4} + \gamma_{754}w_{5}\\
&\quad + \gamma_{\{387,756\}}w_{6} + \gamma_{155}w_{7} + \gamma_{761}w_{8} + \gamma_{763}w_{9} + \gamma_{765}w_{10} + \gamma_{767}w_{11}\\
&\quad + \gamma_{157}w_{12} + \gamma_{159}w_{13} + \gamma_{161}w_{14} + \gamma_{219}w_{15} + \gamma_{\{378,1056\}}w_{16} + \gamma_{1058}w_{17}\\
&\quad + \gamma_{\{382,1060\}}w_{18} + \gamma_{\{341,1063\}}w_{19} + \gamma_{\{345,1065\}}w_{20} + \gamma_{223}w_{21} + \gamma_{1072}w_{22}\\
&\quad + \gamma_{1074}w_{23} + \gamma_{1076}w_{24} + \gamma_{1078}w_{25} + \gamma_{1080}w_{26} + \gamma_{1085}w_{27} + \gamma_{1087}w_{28}\\
&\quad + \gamma_{1089}w_{29} + \gamma_{1091}w_{30} + \gamma_{1096}w_{31} + \gamma_{1098}w_{32} + \gamma_{1100}w_{33} + \gamma_{225}w_{34}\\
&\quad + \gamma_{227}w_{35} + \gamma_{251}w_{36} + \gamma_{1312}w_{37} + \gamma_{1314}w_{38} + \gamma_{1316}w_{39} + \gamma_{1318}w_{40}\\
&\quad + \gamma_{1323}w_{41} + \gamma_{1325}w_{42} + \gamma_{1327}w_{43} + \gamma_{1329}w_{44} + \gamma_{1333}w_{45} + \gamma_{1335}w_{46}\\
&\quad + \gamma_{1338}w_{47} + \gamma_{1342}w_{48} + \gamma_{1344}w_{49} + \gamma_{253}w_{50} + \gamma_{268}w_{51} + \gamma_{270}w_{52}\\
&\quad + \gamma_{272}w_{53} + \gamma_{276}w_{54} + \gamma_{278}w_{55} + \gamma_{281}w_{56} \equiv 0,\\
p_{(2;5)}(\theta) &\equiv \gamma_{\{336,747\}}w_{1} + \gamma_{\{13,148\}}w_{2} + \gamma_{\{340,751\}}w_{3} + \gamma_{753}w_{4} + \gamma_{\{344,755\}}w_{5}\\
&\quad + \gamma_{152}w_{6} + \gamma_{\{349,760\}}w_{7} + \gamma_{762}w_{8} + \gamma_{764}w_{9} + \gamma_{766}w_{10} + \gamma_{156}w_{11}\\
&\quad + \gamma_{\{23,158\}}w_{12} + \gamma_{160}w_{13} + \gamma_{162}w_{14} + \gamma_{\{1055,1172\}}w_{15} + \gamma_{1057}w_{16}\\
&\quad + \gamma_{\{1059,1175\}}w_{17} + \gamma_{220}w_{18} + \gamma_{1064}w_{19} + \gamma_{1066}w_{20} + \gamma_{\{1071,1190\}}w_{21}\\
&\quad + \gamma_{1073}w_{22} + \gamma_{1075}w_{23} + \gamma_{1077}w_{24} + \gamma_{1079}w_{25} + \gamma_{224}w_{26} + \gamma_{1086}w_{27}\\
&\quad + \gamma_{1088}w_{28} + \gamma_{1090}w_{29} + \gamma_{1092}w_{30} + \gamma_{\{1097,1204\}}w_{31} + \gamma_{1099}w_{32}\\
&\quad + \gamma_{1101}w_{33} + \gamma_{226}w_{34} + \gamma_{228}w_{35} + \gamma_{\{1311,1403\}}w_{36} + \gamma_{1313}w_{37}\\
&\quad + \gamma_{1315}w_{38} + \gamma_{1317}w_{39} + \gamma_{252}w_{40} + \gamma_{1324}w_{41} + \gamma_{1326}w_{42} + \gamma_{1328}w_{43}\\
&\quad + \gamma_{1330}w_{44} + \gamma_{1334}w_{45} + \gamma_{1336}w_{46} + \gamma_{1339}w_{47} + \gamma_{1343}w_{48} + \gamma_{1345}w_{49}\\
&\quad + \gamma_{254}w_{50} + \gamma_{\{269,311\}}w_{51} + \gamma_{271}w_{52} + \gamma_{273}w_{53} + \gamma_{277}w_{54}\\
&\quad + \gamma_{279}w_{55} + \gamma_{282}w_{56} \equiv 0.
\end{align*}
Computing from these equalities gives
\begin{equation}\label{c62}
\gamma_j = 0 \mbox{ for } j \in \mathbb J_2,\ \gamma_i= \gamma_j \mbox{ for } (i,j) \in \mathbb K_2,
\end{equation}
where $\mathbb J_2 = \{$145, 146, 149, 150, 151, 152, 153, 154, 155, 156, 157, 159, 160, 161, 162, 217, 218, 219, 220, 221, 222, 223, 224, 225, 226, 227, 228, 249, 250, 251, 252, 253, 254, 265, 266, 267, 268, 270, 271, 272, 273, 274, 275, 276, 277, 278, 279, 280, 281, 282, 749, 750, 753, 754, 757, 758, 759, 761, 762, 763, 764, 765, 766, 767, 768, 769, 770, 1053, 1054, 1057, 1058, 1061, 1062, 1064, 1066, 1067, 1068, 1069, 1070, 1072, 1073, 1074, 1075, 1076, 1077, 1078, 1079, 1080, 1081, 1082, 1083, 1084, 1085, 1086, 1087, 1088, 1089, 1090, 1091, 1092, 1093, 1094, 1095, 1096, 1098, 1099, 1100, 1101, 1102, 1103, 1308, 1309, 1310, 1312, 1313, 1314, 1315, 1316, 1317, 1318, 1319, 1320, 1321, 1322, 1323, 1324, 1325, 1326, 1327, 1328, 1329, 1330, 1331, 1332, 1333, 1334, 1335, 1336, 1337, 1338, 1339, 1340, 1341, 1342, 1343, 1344, 1345, 1346$\}$ and $\mathbb K_2 = \{$(12,147), (13,148), (23,158), (269,311), (336,747), (337,748), (340,751), (341,1063), (344,755), (345,1065), (349,760), (378,1056), (382,1060), (385,752), (387,756),\linebreak (1055,1172), (1059,1175), (1071,1190), (1097,1204), (1311,1403)$\}$.

Applying the homomorphisms $p_{(3;4)}$, $p_{(3;5)}$ and $p_{(4;5)}$ to \eqref{ctd612} and using Theorem \ref{dlsig}, the relations \eqref{c61}, \eqref{c62} give
\begin{align*}
p_{(3;4)}(\theta) &\equiv \gamma_{\{12,172,176,178\}}w_{1} + \gamma_{\{337,805,815,847\}}w_{2} + \gamma_{190}w_{3} + \gamma_{\{385,914\}}w_{4}\\
&\quad + \gamma_{916}w_{5} + \gamma_{\{387,918\}}w_{6} + \gamma_{196}w_{7} + \gamma_{1004}w_{8} + \gamma_{1006}w_{9} + \gamma_{1008}w_{10}\\
&\quad + \gamma_{1010}w_{11} + \gamma_{199}w_{12} + \gamma_{201}w_{13} + \gamma_{203}w_{14} + \gamma_{235}w_{15} + \gamma_{\{341,1162\}}w_{16}\\
&\quad + \gamma_{1164}w_{17} + \gamma_{\{345,1166\}}w_{18} + \gamma_{\{378,1209,1219\}}w_{19} + \gamma_{\{382,1211,1223,1232\}}w_{20}\\
&\quad + \gamma_{238}w_{21} + \gamma_{1239}w_{22} + \gamma_{1241}w_{23} + \gamma_{1243}w_{24} + \gamma_{1245}w_{25} + \gamma_{1247}w_{26}\\
&\quad + \gamma_{1274}w_{27} + \gamma_{1276}w_{28} + \gamma_{1278}w_{29} + \gamma_{1280}w_{30} + \gamma_{1290}w_{31} + \gamma_{1292}w_{32}\\
&\quad + \gamma_{1294}w_{33} + \gamma_{241}w_{34} + \gamma_{243}w_{35} + \gamma_{258}w_{36} + \gamma_{1390}w_{37} + \gamma_{1392}w_{38}\\
&\quad + \gamma_{1394}w_{39} + \gamma_{1396}w_{40} + \gamma_{1421}w_{41} + \gamma_{1423}w_{42} + \gamma_{1425}w_{43} + \gamma_{1427}w_{44}\\
&\quad + \gamma_{1437}w_{45} + \gamma_{1439}w_{46} + \gamma_{1445}w_{47} + \gamma_{1449}w_{48} + \gamma_{1451}w_{49} + \gamma_{261}w_{50}\\
&\quad + \gamma_{301}w_{51} + \gamma_{303}w_{52} + \gamma_{305}w_{53} + \gamma_{319}w_{54} + \gamma_{321}w_{55} + \gamma_{327}w_{56} \equiv 0,\\
p_{(3;5)}(\theta) &\equiv \gamma_{\{336,804,814,832,852\}}w_{1} + \gamma_{\{13,173,177,179\}}w_{2} + \gamma_{\{377,913\}}w_{3} + \gamma_{915}w_{4}\\
&\quad + \gamma_{\{381,917\}}w_{5} + \gamma_{191}w_{6} + \gamma_{\{441,1003\}}w_{7} + \gamma_{1005}w_{8} + \gamma_{1007}w_{9} + \gamma_{1009}w_{10}\\
&\quad + \gamma_{197}w_{11} + \gamma_{\{59,200\}}w_{12} + \gamma_{202}w_{13} + \gamma_{204}w_{14} + \gamma_{\{1055,1161\}}w_{15}\\
&\quad + \gamma_{1163}w_{16} + \gamma_{\{1059,1165\}}w_{17} + \gamma_{236}w_{18} + \gamma_{\{1210,1222,1234\}}w_{19}\\
&\quad + \gamma_{\{1212,1226\}}w_{20} + \gamma_{\{1238,1255\}}w_{21} + \gamma_{1240}w_{22} + \gamma_{1242}w_{23} + \gamma_{1244}w_{24}\\
&\quad + \gamma_{1246}w_{25} + \gamma_{239}w_{26} + \gamma_{1275}w_{27} + \gamma_{1277}w_{28} + \gamma_{1279}w_{29} + \gamma_{1281}w_{30}\\
&\quad + \gamma_{\{1291,1300\}}w_{31} + \gamma_{1293}w_{32} + \gamma_{1295}w_{33} + \gamma_{242}w_{34} + \gamma_{244}w_{35}\\
&\quad + \gamma_{\{1311,1389\}}w_{36} + \gamma_{1391}w_{37} + \gamma_{1393}w_{38} + \gamma_{1395}w_{39} + \gamma_{259}w_{40}\\
&\quad + \gamma_{1422}w_{41} + \gamma_{1424}w_{42} + \gamma_{1426}w_{43} + \gamma_{1428}w_{44} + \gamma_{1438}w_{45} + \gamma_{1440}w_{46}\\
&\quad + \gamma_{1446}w_{47} + \gamma_{1450}w_{48} + \gamma_{1452}w_{49} + \gamma_{262}w_{50} + \gamma_{\{269,302\}}w_{51}\\
&\quad + \gamma_{304}w_{52} + \gamma_{306}w_{53} + \gamma_{320}w_{54} + \gamma_{322}w_{55} + \gamma_{328}w_{56} \equiv 0,\\
p_{(4;5)}(\theta) &\equiv \gamma_{\{349,841,843,844,845,846\}}w_{1} + \gamma_{\{23,180,182,183\}}w_{2}\\
&\quad + \gamma_{\{393,941,943,944,945,946\}}w_{3} + \gamma_{951}w_{4} + \gamma_{\{419,955,957,958\}}w_{5} + \gamma_{192}w_{6}\\
&\quad + \gamma_{\{441,1016,1018,1019,1020,1021\}}w_{7} + \gamma_{1026}w_{8} + \gamma_{1029}w_{9} + \gamma_{1033}w_{10}\\
&\quad + \gamma_{198}w_{11} + \gamma_{\{59,208,210,211\}}w_{12} + \gamma_{214}w_{13} + \gamma_{216}w_{14}\\
&\quad + \gamma_{\{1071,1189,1191,1192,1193,1194\}}w_{15} + \gamma_{1199}w_{16} + \gamma_{\{1097,1203,1205,1206\}}w_{17}\\
&\quad + \gamma_{237}w_{18} + \gamma_{1231}w_{19} + \gamma_{1237}w_{20} + \gamma_{\{1254,1255,1256,1257,1258,1259\}}w_{21}\\
&\quad + \gamma_{1264}w_{22} + \gamma_{1267}w_{23} + \gamma_{1269}w_{24} + \gamma_{1272}w_{25} + \gamma_{240}w_{26} + \gamma_{1286}w_{27}\\
&\quad + \gamma_{1287}w_{28} + \gamma_{1288}w_{29} + \gamma_{1289}w_{30} + \gamma_{\{1299,1300,1301,1302\}}w_{31} + \gamma_{1305}w_{32}\\
&\quad + \gamma_{1307}w_{33} + \gamma_{247}w_{34} + \gamma_{248}w_{35} + \gamma_{\{1311,1402,1404,1405,1406,1407\}}w_{36}\\
&\quad + \gamma_{1412}w_{37} + \gamma_{1415}w_{38} + \gamma_{1419}w_{39} + \gamma_{260}w_{40} + \gamma_{1433}w_{41} + \gamma_{1434}w_{42}\\
&\quad + \gamma_{1435}w_{43} + \gamma_{1436}w_{44} + \gamma_{1443}w_{45} + \gamma_{1444}w_{46} + \gamma_{1448}w_{47} + \gamma_{1455}w_{48}\\
&\quad + \gamma_{1456}w_{49} + \gamma_{264}w_{50} + \gamma_{\{269,310,312,313\}}w_{51} + \gamma_{316}w_{52}\\
&\quad + \gamma_{318}w_{53} + \gamma_{325}w_{54} + \gamma_{326}w_{55} + \gamma_{330}w_{56} \equiv 0.
\end{align*}
From the above equalities it implies
\begin{equation}\label{c63}
\gamma_j = 0 \mbox{ for } j \in \mathbb J_3,\ \gamma_i= \gamma_j \mbox{ for } (i,j) \in \mathbb K_3,
\end{equation}
where $\mathbb J_3 = \{$190, 191, 192, 196, 197, 198, 199, 201, 202, 203, 204, 214, 216, 235, 236, 237, 238, 239, 240, 241, 242, 243, 244, 247, 248, 258, 259, 260, 261, 262, 264, 301, 303, 304, 305, 306, 316, 318, 319, 320, 321, 322, 325, 326, 327, 328, 330, 915, 916, 951, 1004, 1005, 1006, 1007, 1008, 1009, 1010, 1026, 1029, 1033, 1163, 1164, 1199, 1231, 1237, 1239, 1240, 1241, 1242, 1243, 1244, 1245, 1246, 1247, 1264, 1267, 1269, 1272, 1274, 1275, 1276, 1277, 1278, 1279, 1280, 1281, 1286, 1287, 1288, 1289, 1290, 1292, 1293, 1294, 1295, 1305, 1307, 1390, 1391, 1392, 1393, 1394, 1395, 1396, 1412, 1415, 1419, 1421, 1422, 1423, 1424, 1425, 1426, 1427, 1428, 1433, 1434, 1435, 1436, 1437, 1438, 1439, 1440, 1443, 1444, 1445, 1446, 1448, 1449, 1450, 1451, 1452, 1455, 1456$\}$ and $\mathbb K_3 = \{$(59,200), (269,302), (341,1162), (345,1166), (377,913), (381,917), (385,914), (387,918), (441,1003), (1055,1161), (1059,1165), (1212,1226), (1238,1255), (1291,1300), (1311,1389)$\}$.

By applying the homomorphisms $p_{(1;(2,3))}$, $p_{(1;(2,4))}$, $p_{(1;(3,4))}$ to \eqref{ctd612} and using \eqref{c61}, \eqref{c62}, \eqref{c63}, we get
\begin{align*}
p_{(1;(2,3))}(\theta) &\equiv \gamma_{73}w_{1} + \gamma_{74}w_{2} + \gamma_{77}w_{3} + \gamma_{485}w_{4} + \gamma_{486}w_{5} + \gamma_{78}w_{6} + \gamma_{81}w_{7}\\
&\quad + \gamma_{493}w_{8} + \gamma_{494}w_{9} + \gamma_{495}w_{10} + \gamma_{82}w_{11} + \gamma_{504}w_{12} + \gamma_{505}w_{13}\\
&\quad + \gamma_{506}w_{14} + \gamma_{100}w_{15} + \gamma_{540}w_{16} + \gamma_{541}w_{17} + \gamma_{101}w_{18} + \gamma_{548}w_{19}\\
&\quad + \gamma_{549}w_{20} + \gamma_{104}w_{21} + \gamma_{554}w_{22} + \gamma_{555}w_{23} + \gamma_{556}w_{24} + \gamma_{557}w_{25}\\
&\quad + \gamma_{105}w_{26} + \gamma_{568}w_{27} + \gamma_{569}w_{28} + \gamma_{570}w_{29} + \gamma_{571}w_{30} + \gamma_{580}w_{31}\\
&\quad + \gamma_{581}w_{32} + \gamma_{582}w_{33} + \gamma_{589}w_{34} + \gamma_{590}w_{35} + \gamma_{118}w_{36} + \gamma_{648}w_{37}\\
&\quad + \gamma_{649}w_{38} + \gamma_{650}w_{39} + \gamma_{119}w_{40} + \gamma_{659}w_{41} + \gamma_{660}w_{42} + \gamma_{661}w_{43}\\
&\quad + \gamma_{662}w_{44} + \gamma_{671}w_{45} + \gamma_{672}w_{46} + \gamma_{677}w_{47} + \gamma_{680}w_{48} + \gamma_{681}w_{49}\\
&\quad + \gamma_{686}w_{50} + \gamma_{729}w_{51} + \gamma_{730}w_{52} + \gamma_{731}w_{53} + \gamma_{738}w_{54} + \gamma_{739}w_{55}\\
&\quad + \gamma_{744}w_{56} \equiv 0,\\
p_{(1;(2,4))}(\theta) &\equiv \gamma_{\{12,75\}}w_{1} + \gamma_{\{337,484\}}w_{2} + \gamma_{79}w_{3} + \gamma_{\{341,385,488\}}w_{4} + \gamma_{490}w_{5}\\
&\quad + \gamma_{\{345,387,492\}}w_{6} + \gamma_{83}w_{7} + \gamma_{497}w_{8} + \gamma_{499}w_{9} + \gamma_{501}w_{10}\\
&\quad + \gamma_{503}w_{11} + \gamma_{85}w_{12} + \gamma_{87}w_{13} + \gamma_{89}w_{14} + \gamma_{102}w_{15} + \gamma_{\{378,543\}}w_{16}\\
&\quad + \gamma_{545}w_{17} + \gamma_{\{382,547\}}w_{18} + \gamma_{550}w_{19} + \gamma_{552}w_{20} + \gamma_{106}w_{21}\\
&\quad + \gamma_{559}w_{22} + \gamma_{561}w_{23} + \gamma_{563}w_{24} + \gamma_{565}w_{25} + \gamma_{567}w_{26} + \gamma_{572}w_{27}\\
&\quad + \gamma_{574}w_{28} + \gamma_{576}w_{29} + \gamma_{578}w_{30} + \gamma_{583}w_{31} + \gamma_{585}w_{32} + \gamma_{587}w_{33}\\
&\quad + \gamma_{108}w_{34} + \gamma_{110}w_{35} + \gamma_{120}w_{36} + \gamma_{652}w_{37} + \gamma_{654}w_{38} + \gamma_{656}w_{39}\\
&\quad + \gamma_{658}w_{40} + \gamma_{663}w_{41} + \gamma_{665}w_{42} + \gamma_{667}w_{43} + \gamma_{669}w_{44} + \gamma_{673}w_{45}\\
&\quad + \gamma_{675}w_{46} + \gamma_{678}w_{47} + \gamma_{682}w_{48} + \gamma_{684}w_{49} + \gamma_{122}w_{50} + \gamma_{732}w_{51}\\
&\quad + \gamma_{734}w_{52} + \gamma_{736}w_{53} + \gamma_{740}w_{54} + \gamma_{742}w_{55} + \gamma_{745}w_{56} \equiv 0,\\
p_{(1;(3,4))}(\theta) &\equiv \gamma_{\{91,172\}}w_{1} + \gamma_{\{508,805\}}w_{2} + \gamma_{\{93,176\}}w_{3} + \gamma_{\{341,512,813\}}w_{4}\\
&\quad + \gamma_{\{514,815\}}w_{5} + \gamma_{\{345,516,817\}}w_{6} + \gamma_{\{95,178\}}w_{7} + \gamma_{\{522,829\}}w_{8}\\
&\quad + \gamma_{\{524,831\}}w_{9} + \gamma_{\{526,833\}}w_{10} + \gamma_{\{528,835\}}w_{11} + \gamma_{\{534,847\}}w_{12}\\
&\quad + \gamma_{\{536,849\}}w_{13} + \gamma_{\{538,851\}}w_{14} + \gamma_{112}w_{15} + \gamma_{\{385,592,1209\}}w_{16}\\
&\quad + \gamma_{594}w_{17} + \gamma_{\{387,596,1211\}}w_{18} + \gamma_{\{378,601,1209\}}w_{19}\\
&\quad + \gamma_{\{382,603,1211\}}w_{20} + \gamma_{114}w_{21} + \gamma_{\{608,1219\}}w_{22} + \gamma_{\{610,1221\}}w_{23}\\
&\quad + \gamma_{612}w_{24} + \gamma_{\{614,1223\}}w_{25} + \gamma_{\{616,1225\}}w_{26} + \gamma_{\{623,1219\}}w_{27}\\
&\quad + \gamma_{\{625,1221\}}w_{28} + \gamma_{\{627,1223\}}w_{29} + \gamma_{\{629,1225\}}w_{30} + \gamma_{635}w_{31}\\
&\quad + \gamma_{\{637,1232\}}w_{32} + \gamma_{639}w_{33} + \gamma_{\{644,1232\}}w_{34} + \gamma_{646}w_{35} + \gamma_{124}w_{36}\\
&\quad + \gamma_{688}w_{37} + \gamma_{690}w_{38} + \gamma_{692}w_{39} + \gamma_{694}w_{40} + \gamma_{700}w_{41} + \gamma_{702}w_{42}\\
&\quad + \gamma_{704}w_{43} + \gamma_{706}w_{44} + \gamma_{712}w_{45} + \gamma_{714}w_{46} + \gamma_{718}w_{47} + \gamma_{721}w_{48}\\
&\quad + \gamma_{723}w_{49} + \gamma_{727}w_{50} + \gamma_{127}w_{51} + \gamma_{129}w_{52} + \gamma_{131}w_{53} + \gamma_{136}w_{54}\\
&\quad + \gamma_{138}w_{55} + \gamma_{142}w_{56} \equiv 0.  
\end{align*}
By a direct calculation using the above equalities we have
\begin{equation}\label{c64}
\gamma_j = 0 \mbox{ for } j \in \mathbb J_4,\ , \gamma_i= \gamma_j \mbox{ for } (i,j) \in \mathbb K_4,
\end{equation}
where $\mathbb J_4 = \{$73, 74, 77, 78, 79, 81, 82, 83, 85, 87, 89, 100, 101, 102, 104, 105, 106, 108, 110, 112, 114, 118, 119, 120, 122, 124, 127, 129, 131, 136, 138, 142, 485, 486, 490, 493, 494, 495, 497, 499, 501, 503, 504, 505, 506, 540, 541, 545, 548, 549, 550, 552, 554, 555, 556, 557, 559, 561, 563, 565, 567, 568, 569, 570, 571, 572, 574, 576, 578, 580, 581, 582, 583, 585, 587, 589, 590, 594, 612, 635, 639, 646, 648, 649, 650, 652, 654, 656, 658, 659, 660, 661, 662, 663, 665, 667, 669, 671, 672, 673, 675, 677, 678, 680, 681, 682, 684, 686, 688, 690, 692, 694, 700, 702, 704, 706, 712, 714, 718, 721, 723, 727, 729, 730, 731, 732, 734, 736, 738, 739, 740, 742, 744, 745$\}$ and $\mathbb K_4 = \{$(12,75), (91,172), (93,176), (95,178), (337,484), (378,543), (382,547), (508,805), (514,815), (522,829), (524,831), (526,833), (528,835), (534,847), (536,849), (538,851), (608,1219), (608,623), (610,1221), (610,625), (614,1223), (614,627), (616,1225),\linebreak (616,629), (637,1232), (637,644)$\}$.

Apply the homomorphisms $p_{(1;(2,5))}$, $p_{(1;(3,5))}$, $p_{(1;(4,5))}$ to \eqref{ctd612} using \eqref{c61}, \eqref{c62}, \eqref{c63}, \eqref{c64} to obtain
\begin{align*}
&p_{(1;(2,5))}(\theta) \equiv \gamma_{\{336,483\}}w_{1} + \gamma_{\{13,76\}}w_{2} + \gamma_{\{340,487\}}w_{3} + \gamma_{489}w_{4} + \gamma_{\{344,491\}}w_{5}\\
&\quad + \gamma_{80}w_{6} + \gamma_{\{349,496\}}w_{7} + \gamma_{498}w_{8} + \gamma_{500}w_{9} + \gamma_{502}w_{10} + \gamma_{84}w_{11} + \gamma_{\{23,86\}}w_{12}\\
&\quad + \gamma_{88}w_{13} + \gamma_{90}w_{14} + \gamma_{542}w_{15} + \gamma_{544}w_{16} + \gamma_{546}w_{17} + \gamma_{103}w_{18} + \gamma_{551}w_{19}\\
&\quad + \gamma_{553}w_{20} + \gamma_{558}w_{21} + \gamma_{560}w_{22} + \gamma_{562}w_{23} + \gamma_{564}w_{24} + \gamma_{566}w_{25} + \gamma_{107}w_{26}\\
&\quad + \gamma_{573}w_{27} + \gamma_{575}w_{28} + \gamma_{577}w_{29} + \gamma_{579}w_{30} + \gamma_{584}w_{31} + \gamma_{586}w_{32} + \gamma_{588}w_{33}\\
&\quad + \gamma_{109}w_{34} + \gamma_{111}w_{35} + \gamma_{\{651,1238\}}w_{36} + \gamma_{653}w_{37} + \gamma_{655}w_{38} + \gamma_{657}w_{39}\\
&\quad + \gamma_{121}w_{40} + \gamma_{664}w_{41} + \gamma_{666}w_{42} + \gamma_{668}w_{43} + \gamma_{670}w_{44} + \gamma_{674}w_{45} + \gamma_{676}w_{46}\\
&\quad + \gamma_{679}w_{47} + \gamma_{683}w_{48} + \gamma_{685}w_{49} + \gamma_{123}w_{50} + \gamma_{\{733,1291\}}w_{51} + \gamma_{735}w_{52}\\
&\quad + \gamma_{737}w_{53} + \gamma_{741}w_{54} + \gamma_{743}w_{55} + \gamma_{746}w_{56} \equiv 0,\\
&p_{(1;(3,5))}(\theta) \equiv \gamma_{\{507,804\}}w_{1} + \gamma_{\{92,173\}}w_{2} + \gamma_{\{511,812\}}w_{3} + \gamma_{\{513,814\}}w_{4}\\
&\quad + \gamma_{\{515,816\}}w_{5} + \gamma_{\{94,177\}}w_{6} + \gamma_{\{521,828,1071\}}w_{7} + \gamma_{\{523,830\}}w_{8}\\
&\quad + \gamma_{\{525,832\}}w_{9} + \gamma_{\{527,834\}}w_{10} + \gamma_{\{96,179\}}w_{11} + \gamma_{\{535,848,1097\}}w_{12}\\
&\quad + \gamma_{\{537,850\}}w_{13} + \gamma_{\{539,852\}}w_{14} + \gamma_{\{377,591,1210\}}w_{15} + \gamma_{593}w_{16}\\
&\quad + \gamma_{\{381,595,1212\}}w_{17} + \gamma_{113}w_{18} + \gamma_{\{602,1210,1222\}}w_{19} + \gamma_{604}w_{20}\\
&\quad + \gamma_{\{607,1220\}}w_{21} + \gamma_{\{609,1222\}}w_{22} + \gamma_{611}w_{23} + \gamma_{\{613,1224\}}w_{24}\\
&\quad + \gamma_{\{615,1212\}}w_{25} + \gamma_{115}w_{26} + \gamma_{\{624,1220\}}w_{27} + \gamma_{\{626,1222,1234\}}w_{28}\\
&\quad + \gamma_{\{628,1224,1234\}}w_{29} + \gamma_{\{630,1212\}}w_{30} + \gamma_{\{636,1233\}}w_{31} + \gamma_{\{638,1234\}}w_{32}\\
&\quad + \gamma_{640}w_{33} + \gamma_{\{645,1233\}}w_{34} + \gamma_{\{647,1234\}}w_{35} + \gamma_{\{441,687\}}w_{36} + \gamma_{689}w_{37}\\
&\quad + \gamma_{691}w_{38} + \gamma_{693}w_{39} + \gamma_{125}w_{40} + \gamma_{701}w_{41} + \gamma_{703}w_{42} + \gamma_{705}w_{43} + \gamma_{707}w_{44}\\
&\quad + \gamma_{713}w_{45} + \gamma_{715}w_{46} + \gamma_{719}w_{47} + \gamma_{722}w_{48} + \gamma_{724}w_{49} + \gamma_{728}w_{50}\\
&\quad + \gamma_{\{59,128\}}w_{51} + \gamma_{130}w_{52} + \gamma_{132}w_{53} + \gamma_{137}w_{54} + \gamma_{139}w_{55} + \gamma_{143}w_{56} \equiv 0,\\
&p_{(1;(4,5))}(\theta) \equiv \gamma_{\{269,310,312,313,509,808,1055,1171,1173,1311,1402,1404,1405\}}w_{1}\\
&\quad + \gamma_{\{269,310,312,313,510,811,1059,1174,1176,1406,1407\}}w_{2}\\
&\quad + \gamma_{\{341,517,822,1071,1180,1189,1191\}}w_{3} + \gamma_{\{518,824,1097,1180,1192,1203,1205\}}w_{4}\\
&\quad + \gamma_{\{345,519,825,1097,1182,1193,1203,1205\}}w_{5} + \gamma_{\{520,827,1182,1194\}}w_{6} + \gamma_{\{529,841\}}w_{7}\\
&\quad + \gamma_{\{530,843,1097,1203,1205\}}w_{8} + \gamma_{\{531,844,1097,1203,1205,1206\}}w_{9} + \gamma_{\{532,845,1206\}}w_{10}\\
&\quad + \gamma_{\{533,846\}}w_{11} + \gamma_{\{97,180\}}w_{12} + \gamma_{\{98,182\}}w_{13} + \gamma_{\{99,183\}}w_{14}\\
&\quad + \gamma_{\{378,597,923,1216,1238,1254,1256\}}w_{15} + \gamma_{\{598,925,1216,1257,1291,1299,1301\}}w_{16}\\
&\quad + \gamma_{\{382,599,926,1218,1258,1291,1299,1301\}}w_{17} + \gamma_{\{600,928,1218,1259\}}w_{18}\\
&\quad + \gamma_{\{605,932\}}w_{19} + \gamma_{\{606,934\}}w_{20} + \gamma_{\{617,941\}}w_{21} + \gamma_{\{618,943\}}w_{22}\\
&\quad + \gamma_{\{619,944\}}w_{23} + \gamma_{620}w_{24} + \gamma_{\{621,945\}}w_{25} + \gamma_{\{622,946\}}w_{26} + \gamma_{631}w_{27}\\
&\quad + \gamma_{632}w_{28} + \gamma_{633}w_{29} + \gamma_{634}w_{30} + \gamma_{\{641,955\}}w_{31} + \gamma_{\{642,957\}}w_{32}\\
&\quad + \gamma_{\{643,958\}}w_{33} + \gamma_{116}w_{34} + \gamma_{117}w_{35} + \gamma_{\{695,1016\}}w_{36}\\
&\quad + \gamma_{\{696,1018,1291,1299,1301\}}w_{37} + \gamma_{\{697,1019,1291,1299,1301,1302\}}w_{38} \\
&\quad+ \gamma_{\{698,1020,1302\}}w_{39} + \gamma_{\{699,1021\}}w_{40} + \gamma_{708}w_{41} + \gamma_{709}w_{42} + \gamma_{710}w_{43}\\
&\quad + \gamma_{711}w_{44} + \gamma_{716}w_{45} + \gamma_{717}w_{46} + \gamma_{720}w_{47} + \gamma_{725}w_{48} + \gamma_{726}w_{49}\\
&\quad + \gamma_{126}w_{50} + \gamma_{\{133,208\}}w_{51} + \gamma_{\{134,210\}}w_{52} + \gamma_{\{135,211\}}w_{53}\\
&\quad + \gamma_{140}w_{54} + \gamma_{141}w_{55} + \gamma_{144}w_{56} \equiv 0.
\end{align*}
From the above relations we obtain
\begin{equation}\label{c65}
\gamma_j = 0 \mbox{ for } j \in \mathbb J_5,\ \gamma_i= \gamma_j \mbox{ for } (i,j) \in \mathbb K_5,
\end{equation}
where $\mathbb J_5 = \{$80, 84, 88, 90, 103, 107, 109, 111, 113, 115, 116, 117, 121, 123, 125, 126, 130, 132, 137, 139, 140, 141, 143, 144, 489, 498, 500, 502, 542, 544, 546, 551, 553, 558, 560, 562, 564, 566, 573, 575, 577, 579, 584, 586, 588, 593, 604, 611, 620, 631, 632, 633, 634, 640, 653, 655, 657, 664, 666, 668, 670, 674, 676, 679, 683, 685, 689, 691, 693, 701, 703, 705, 707, 708, 709, 710, 711, 713, 715, 716, 717, 719, 720, 722, 724, 725, 726, 728, 735, 737, 741, 743, 746$\}$ and $\mathbb K_5 = \{$(13,76), (23,86), (59,128), (92,173), (94,177), (96,179), (97,180), (98,182), (99,183), (133,208), (134,210), (135,211), (336,483), (340,487), (344,491), (349,496), (441,687), (507,804), (511,812), (513,814), (515,816), (523,830), (525,832), (527,834), (529,841), (533,846), (537,850), (539,852), (605,932), (606,934), (607,624),\linebreak (607,1220), (609,1222), (613,1224), (615,630), (615,1212), (617,941), (618,943),\linebreak (619,944), (621,945), (622,946), (636,645), (636,1233), (638,647), (638,1234),\linebreak (641,955), (642,957), (643,958), (651,1238), (695,1016), (699,1021), (733,1291)$\}$.

Applying the homomorphisms $p_{(s;(u,v))}$, $2\leqslant s < u < v \leqslant 5$ to \eqref{ctd612}, using Theorem \ref{dlsig} and the relations \eqref{c61}-\eqref{c65} give
\begin{align*}
&p_{(2;(3,4))}(\theta) \equiv \gamma_{\{91,163,184,193\}}w_{1} + \gamma_{\{508,772,859,1035\}}w_{2} + \gamma_{\{93,165,184,186\}}w_{3}\\
&\quad + \gamma_{\{385,776,813,857,866,877,900,962,974\}}w_{4} + \gamma_{\{514,778,859,900\}}w_{5}\\
&\quad + \gamma_{\{387,780,817,861,868,966,978,1037,1044\}}w_{6} + \gamma_{\{95,167,186,193\}}w_{7}\\
&\quad + \gamma_{\{522,786,873,888,900,962,974\}}w_{8}\\
&\quad + \gamma_{\{524,788,875,877,890,900,964,976,986,992,995,1035,1037,1044\}}w_{9}\\
&\quad + \gamma_{\{526,790,879,892,966,978,992,995\}}w_{10} + \gamma_{\{528,792,881,894,968,980,988,1039,1046,1050\}}w_{11}\\
&\quad + \gamma_{\{534,798,900,1035\}}w_{12} + \gamma_{\{536,800,902,909,992,995,1037,1044\}}w_{13}\\
&\quad + \gamma_{\{538,802,904,911,997,1001,1039,1046,1050\}}w_{14} + \gamma_{229}w_{15} + \gamma_{\{341,1105,1209\}}w_{16}\\
&\quad + \gamma_{1107}w_{17} + \gamma_{\{345,1109,1211\}}w_{18} + \gamma_{\{378,1114,1209\}}w_{19} + \gamma_{\{382,1116,1211\}}w_{20}\\
&\quad + \gamma_{231}w_{21} + \gamma_{\{608,1121\}}w_{22} + \gamma_{\{610,1123\}}w_{23} + \gamma_{1125}w_{24} + \gamma_{\{614,1127\}}w_{25}\\
&\quad + \gamma_{\{616,1129\}}w_{26} + \gamma_{\{608,1136\}}w_{27} + \gamma_{\{610,1138\}}w_{28} + \gamma_{\{614,1140\}}w_{29}\\
&\quad + \gamma_{\{616,1142\}}w_{30} + \gamma_{1148}w_{31} + \gamma_{\{637,1150\}}w_{32} + \gamma_{1152}w_{33} + \gamma_{\{637,1157\}}w_{34}\\
&\quad + \gamma_{1159}w_{35} + \gamma_{255}w_{36} + \gamma_{1348}w_{37} + \gamma_{1350}w_{38} + \gamma_{1352}w_{39} + \gamma_{1354}w_{40}\\
&\quad + \gamma_{1360}w_{41} + \gamma_{1362}w_{42} + \gamma_{1364}w_{43} + \gamma_{1366}w_{44} + \gamma_{1372}w_{45} + \gamma_{1374}w_{46}\\
&\quad + \gamma_{1378}w_{47} + \gamma_{1381}w_{48} + \gamma_{1383}w_{49} + \gamma_{1387}w_{50} + \gamma_{283}w_{51} + \gamma_{285}w_{52}\\
&\quad + \gamma_{287}w_{53} + \gamma_{292}w_{54} + \gamma_{294}w_{55} + \gamma_{298}w_{56} \equiv 0,\\
&p_{(2;(3,5))}(\theta) \equiv \gamma_{\{507,771,858,965,1040\}}w_{1} + \gamma_{\{92,164,185,194\}}w_{2}\\
&\quad + \gamma_{\{511,775,856,867,874,963,977,1038,1047\}}w_{3}\\
&\quad + \gamma_{\{513,777,858,876,878,893,903,979,989,993,998,1038,1040,1047\}}w_{4}\\
&\quad + \gamma_{\{515,779,860,869,878,880,893,903,967,979,981,989,993,998,1038,1040,1047\}}w_{5}\\
&\quad + \gamma_{\{94,166,185,187\}}w_{6} + \gamma_{\{393,785,828,872,889,961,975,987,1036,1045,1051\}}w_{7}\\
&\quad + \gamma_{\{523,787,874,891,903,963,977,993,998\}}w_{8}\\
&\quad + \gamma_{\{525,789,876,878,893,903,905,965,979,989,993,998,1038,1040,1047\}}w_{9}\\
&\quad + \gamma_{\{527,791,880,895,905,967,981\}}w_{10} + \gamma_{\{96,168,187,194\}}w_{11}\\
&\quad + \gamma_{\{419,799,848,901,910,996,1002,1036,1045,1051\}}w_{12}\\
&\quad + \gamma_{\{537,801,903,912,993,998,1038,1047\}}w_{13} + \gamma_{\{539,803,905,1040\}}w_{14}\\
&\quad + \gamma_{\{1055,1104,1210\}}w_{15} + \gamma_{1106}w_{16} + \gamma_{\{615,1059,1108\}}w_{17} + \gamma_{230}w_{18}\\
&\quad + \gamma_{\{609,1115,1210\}}w_{19} + \gamma_{1117}w_{20} + \gamma_{\{607,1120\}}w_{21} + \gamma_{\{609,1122\}}w_{22} + \gamma_{1124}w_{23}\\
&\quad + \gamma_{\{613,1126\}}w_{24} + \gamma_{\{615,1128\}}w_{25} + \gamma_{232}w_{26} + \gamma_{\{607,1137\}}w_{27}\\
&\quad + \gamma_{\{609,638,1139\}}w_{28} + \gamma_{\{613,638,1141\}}w_{29} + \gamma_{\{615,1143\}}w_{30} + \gamma_{\{636,1149\}}w_{31}\\
&\quad + \gamma_{\{638,1151\}}w_{32} + \gamma_{1153}w_{33} + \gamma_{\{636,1158\}}w_{34} + \gamma_{\{638,1160\}}w_{35}\\
&\quad + \gamma_{\{1311,1347\}}w_{36} + \gamma_{1349}w_{37} + \gamma_{1351}w_{38} + \gamma_{1353}w_{39} + \gamma_{256}w_{40}\\
&\quad + \gamma_{1361}w_{41} + \gamma_{1363}w_{42} + \gamma_{1365}w_{43} + \gamma_{1367}w_{44} + \gamma_{1373}w_{45} + \gamma_{1375}w_{46}\\
&\quad + \gamma_{1379}w_{47} + \gamma_{1382}w_{48} + \gamma_{1384}w_{49} + \gamma_{1388}w_{50} + \gamma_{\{269,284\}}w_{51} + \gamma_{286}w_{52}\\
&\quad + \gamma_{288}w_{53} + \gamma_{293}w_{54} + \gamma_{295}w_{55} + \gamma_{299}w_{56} \equiv 0,\\
&p_{(2;(4,5))}(\theta) \equiv \gamma_{\{59,133,134,135,377,441,695,773,808,862,863,923,925,969,970,971,1018,1019,1041,1042,1043\}}w_{1}\\
&\quad + \gamma_{\{59,133,134,135,381,699,774,811,864,865,926,928,972,973,1020,1041,1042,1043\}}w_{2}\\
&\quad + \gamma_{\{385,393,605,617,618,781,822,870,882,883,982,983,1048,1049\}}w_{3}\\
&\quad + \gamma_{\{419,605,619,641,642,782,824,870,884,898,906,907,982,983,991,994,999,1000,1048,1049,1052\}}w_{4}\\
&\quad + \gamma_{\{387,419,606,621,641,642,783,825,871,885,886,898,906,907,984,985,991,994,999,1000,1048,1049,1052\}}w_{5}\\
&\quad + \gamma_{\{606,622,784,827,871,887,984,985,1048,1049\}}w_{6} + \gamma_{\{529,793,896,990,1052\}}w_{7}\\
&\quad + \gamma_{\{419,641,642,794,843,896,897,906,907,999,1000\}}w_{8}\\
&\quad + \gamma_{\{419,641,642,643,795,844,897,898,906,907,908,990,991\}}w_{9}\\
&\quad + \gamma_{\{643,796,845,898,899,908,999,1000\}}w_{10} + \gamma_{\{533,797,899,991,1052\}}w_{11}\\
&\quad + \gamma_{\{97,169,188,195\}}w_{12} + \gamma_{\{98,170,188,189\}}w_{13} + \gamma_{\{99,171,189,195\}}w_{14}\\
&\quad + \gamma_{\{378,651,733,1110,1171,1216,1254,1299\}}w_{15} + \gamma_{\{1111,1173,1216,1256,1257,1301,1302\}}w_{16}\\
&\quad + \gamma_{\{382,733,1112,1174,1218,1299\}}w_{17} + \gamma_{\{1113,1176,1218,1258,1259,1301,1302\}}w_{18}\\
&\quad + \gamma_{\{1118,1180\}}w_{19} + \gamma_{\{1119,1182\}}w_{20} + \gamma_{\{1130,1189\}}w_{21} + \gamma_{\{1131,1191\}}w_{22}\\
&\quad + \gamma_{\{1132,1192\}}w_{23} + \gamma_{1133}w_{24} + \gamma_{\{1134,1193\}}w_{25} + \gamma_{\{1135,1194\}}w_{26} + \gamma_{1144}w_{27}\\
&\quad + \gamma_{1145}w_{28} + \gamma_{1146}w_{29} + \gamma_{1147}w_{30} + \gamma_{\{1154,1203\}}w_{31} + \gamma_{\{1155,1205\}}w_{32}\\
&\quad + \gamma_{\{1156,1206\}}w_{33} + \gamma_{233}w_{34} + \gamma_{234}w_{35} + \gamma_{\{1355,1402\}}w_{36}\\
&\quad + \gamma_{\{1356,1404\}}w_{37} + \gamma_{\{1357,1405\}}w_{38} + \gamma_{\{1358,1406\}}w_{39} + \gamma_{\{1359,1407\}}w_{40}\\
&\quad + \gamma_{1368}w_{41} + \gamma_{1369}w_{42} + \gamma_{1370}w_{43} + \gamma_{1371}w_{44} + \gamma_{1376}w_{45} + \gamma_{1377}w_{46}\\
&\quad + \gamma_{1380}w_{47} + \gamma_{1385}w_{48} + \gamma_{1386}w_{49} + \gamma_{257}w_{50} + \gamma_{\{289,310\}}w_{51}\\
&\quad + \gamma_{\{290,312\}}w_{52} + \gamma_{\{291,313\}}w_{53} + \gamma_{296}w_{54} + \gamma_{297}w_{55} + \gamma_{300}w_{56} \equiv 0,\\
&p_{(3;(4,5))}(\theta) \equiv (\gamma_{\{12,23,91,93,95,97,98,99,336,340,349,507,511,513,522,523,524,525,529,534,536\}}\\
&\quad +\gamma_{\{537, 538, 539, 806, 808, 813, 818, 819, 822, 824, 828, 836, 837, 838, 843, 844, 848, 853, 854, 855\}})w_{1}\\
&\quad + \gamma_{\{13,23,92,94,96,97,98,99,337,344,508,514,515,526,527,528,533,534,536,537,538\}}\\
&\quad +\gamma_{\{539, 807, 811, 817, 820, 821, 825, 827, 839, 840, 845, 848, 853, 854, 855\}})w_{2}\\
&\quad + \gamma_{\{385,393,419,605,617,641,919,923,929,935,947,948,952,959,960\}}w_{3}\\
&\quad + \gamma_{\{605,618,619,642,643,920,925,929,936,937,947,948,953,954,959,960,1024,1028,1031,1034\}}w_{4}\\
&\quad + \gamma_{\{387,419,606,641,921,926,930,938,949,950,952,959,960,1024,1028,1031,1034\}}w_{5}\\
&\quad + \gamma_{\{606,621,622,642,643,922,928,930,939,940,949,950,953,954,959,960\}}w_{6}\\
&\quad + \gamma_{\{695,1011,1022,1027,1031,1034\}}w_{7} + \gamma_{\{1012,1018,1022,1023,1030,1031,1032\}}w_{8}\\
&\quad + \gamma_{\{1013,1019,1023,1024,1027,1028,1031,1032\}}w_{9} + \gamma_{\{1014,1020,1024,1025,1030,1031,1032\}}w_{10}\\
&\quad + \gamma_{\{699,1015,1025,1028,1032,1034\}}w_{11} + \gamma_{\{133,205,212,215\}}w_{12} + \gamma_{\{134,206,212,213\}}w_{13}\\
&\quad + \gamma_{\{135,207,213,215\}}w_{14}\\
&\quad + \gamma_{\{341,1071,1097,1167,1171,1177,1180,1183,1189,1195,1196,1200,1203,1207,1208\}}w_{15}\\
&\quad + \gamma_{\{1168,1173,1177,1180,1184,1185,1191,1192,1195,1196,1201,1202\}}\\
&\quad +\gamma_{\{1205,1206,1207,1208,1410,1414,1417,1420\}})w_{16}\\
&\quad + \gamma_{\{345,1097,1169,1174,1178,1182,1186,1197,1198,1200,1203,1207,1208,1410,1414,1417,1420\}}w_{17}\\
&\quad + \gamma_{\{1170,1176,1178,1182,1187,1188,1193,1194,1197,1198,1201,1202,1205,1206,1207,1208\}}w_{18}\\
&\quad + \gamma_{\{378,607,608,609,610,636,637,638,1209,1210,1213,1216,1227,1228\}}\\
&\quad +\gamma_{\{1235,1236,1261,1265,1271,1273\}})w_{19}\\
&\quad + \gamma_{\{382,613,614,616,636,637,638,1211,1214,1218,1229,1230,1235,1236,1263,1266,1271,1273\}}w_{20}\\
&\quad + \gamma_{\{1248,1254,1260,1265,1270,1273\}}w_{21} + \gamma_{\{1249,1256,1260,1261,1268,1270,1271\}}w_{22}\\
&\quad + \gamma_{\{1250,1257,1261,1265,1271,1273\}}w_{23} + \gamma_{\{1251,1262,1266,1270,1273\}}w_{24}\\
&\quad + \gamma_{\{1252,1258,1262,1263,1268,1270,1271\}}w_{25} + \gamma_{\{1253,1259,1263,1266,1271,1273\}}w_{26}\\
&\quad + \gamma_{1282}w_{27} + \gamma_{\{1283,1304,1306\}}w_{28} + \gamma_{\{1284,1304,1306\}}w_{29} + \gamma_{1285}w_{30}\\
&\quad + \gamma_{\{1296,1299,1303,1306\}}w_{31} + \gamma_{\{1297,1301,1303,1304\}}w_{32} + \gamma_{\{1298,1302,1304,1306\}}w_{33}\\
&\quad + \gamma_{245}w_{34} + \gamma_{246}w_{35} + \gamma_{\{1397,1402,1408,1413,1417,1420\}}w_{36}\\
&\quad + \gamma_{\{1398,1404,1408,1409,1416,1417,1418\}}w_{37}\\
&\quad + \gamma_{\{1399,1405,1409,1410,1413,1414,1417,1418\}}w_{38}\\
&\quad + \gamma_{\{1400,1406,1410,1411,1416,1417,1418\}}w_{39} + \gamma_{\{1401,1407,1411,1414,1418,1420\}}w_{40}\\
&\quad + \gamma_{1429}w_{41} + \gamma_{1430}w_{42} + \gamma_{1431}w_{43} + \gamma_{1432}w_{44} + \gamma_{1441}w_{45} + \gamma_{1442}w_{46}\\
&\quad + \gamma_{1447}w_{47} + \gamma_{1453}w_{48} + \gamma_{1454}w_{49} + \gamma_{263}w_{50} + \gamma_{\{307,310,314,317\}}w_{51}\\
&\quad + \gamma_{\{308,312,314,315\}}w_{52} + \gamma_{\{309,313,315,317\}}w_{53} + \gamma_{323}w_{54}\\
&\quad + \gamma_{324}w_{55} + \gamma_{329}w_{56} \equiv 0.
\end{align*}
By computing from these relations we obtain
\begin{equation}\label{c66}
\gamma_j = 0 \mbox{ for } j \in \mathbb J_6,\ \gamma_i = \gamma_j \mbox{ for } (i,j) \in \mathbb K_6,
\end{equation}
where $\mathbb J_6 = \{$229, 230, 231, 232, 233, 234, 245, 246, 255, 256, 257, 263, 283, 285, 286, 287, 288, 292, 293, 294, 295, 296, 297, 298, 299, 300, 323, 324, 329, 1106, 1107, 1117, 1124, 1125, 1133, 1144, 1145, 1146, 1147, 1148, 1152, 1153, 1159, 1282, 1285, 1348, 1349, 1350, 1351, 1352, 1353, 1354, 1360, 1361, 1362, 1363, 1364, 1365, 1366, 1367, 1368, 1369, 1370, 1371, 1372, 1373, 1374, 1375, 1376, 1377, 1378, 1379, 1380, 1381, 1382, 1383, 1384, 1385, 1386, 1387, 1388, 1429, 1430, 1431, 1432, 1441, 1442, 1447, 1453, 1454$\}$ and $\mathbb K_6 = \{$(269,284), (289,310), (290,312), (291,313), (607,1120), (607,1137), (608,1121), (608,1136), (609,1122), (610,1123), (610,1138), (613,1126), (614,1127), (614,1140), (615,1128), (615,1143), (616,1129), (616,1142), (636,1149), (636,1158), (637,1150), (637,1157), (638,1151), (638,1160), (1118,1180), (1119,1182), (1130,1189), (1131,1191), (1132,1192), (1134,1193), (1135,1194),\linebreak (1154,1203), (1155,1205), (1156,1206), (1311,1347), (1355,1402), (1356,1404),\linebreak (1357,1405), (1358,1406), (1359,1407)$\}$.

By applying the homomorphisms $p_{(1;(2,3,4))}$ and $p_{(1;(2,3,5))}$   to \eqref{ctd612} and using \eqref{c61}-\eqref{c66} we get
\begin{align*}
&p_{(1;(2,3,4))}(\theta) \equiv \gamma_{163}w_{1} + \gamma_{772}w_{2} + \gamma_{165}w_{3} + \gamma_{\{378,385,488,512,776,813,1105,1114\}}w_{4}\\
&\quad + \gamma_{778}w_{5} + \gamma_{\{382,387,492,516,780,817,1109,1116\}}w_{6} + \gamma_{167}w_{7} + \gamma_{786}w_{8} + \gamma_{788}w_{9}\\
&\quad + \gamma_{790}w_{10} + \gamma_{792}w_{11} + \gamma_{798}w_{12} + \gamma_{800}w_{13} + \gamma_{802}w_{14} + \gamma_{184}w_{15}\\
&\quad + \gamma_{\{341,385,592,857,1105,1209\}}w_{16} + \gamma_{859}w_{17} + \gamma_{\{345,387,596,861,1109,1211\}}w_{18}\\
&\quad + \gamma_{\{378,601,866,1114,1209\}}w_{19} + \gamma_{\{382,603,868,1116,1211\}}w_{20} + \gamma_{186}w_{21}\\
&\quad + \gamma_{\{608,873\}}w_{22} + \gamma_{\{610,875\}}w_{23} + \gamma_{877}w_{24} + \gamma_{\{614,879\}}w_{25} + \gamma_{\{616,881\}}w_{26}\\
&\quad + \gamma_{\{608,888\}}w_{27} + \gamma_{\{610,890\}}w_{28} + \gamma_{\{614,892\}}w_{29} + \gamma_{\{616,894\}}w_{30} + \gamma_{900}w_{31}\\
&\quad + \gamma_{\{637,902\}}w_{32} + \gamma_{904}w_{33} + \gamma_{\{637,909\}}w_{34} + \gamma_{911}w_{35} + \gamma_{193}w_{36} + \gamma_{962}w_{37}\\
&\quad + \gamma_{964}w_{38} + \gamma_{966}w_{39} + \gamma_{968}w_{40} + \gamma_{974}w_{41} + \gamma_{976}w_{42} + \gamma_{978}w_{43} + \gamma_{980}w_{44}\\
&\quad + \gamma_{986}w_{45} + \gamma_{988}w_{46} + \gamma_{992}w_{47} + \gamma_{995}w_{48} + \gamma_{997}w_{49} + \gamma_{1001}w_{50} + \gamma_{1035}w_{51}\\
&\quad + \gamma_{1037}w_{52} + \gamma_{1039}w_{53} + \gamma_{1044}w_{54} + \gamma_{1046}w_{55} + \gamma_{1050}w_{56} \equiv 0,\\
&p_{(1;(2,3,5))}(\theta) \equiv \gamma_{771}w_{1} + \gamma_{164}w_{2} + \gamma_{\{609,775,1055,1104,1115\}}w_{3} + \gamma_{\{613,638,777,1141\}}w_{4}\\
&\quad + \gamma_{\{613,615,638,779,1059,1108,1141\}}w_{5} + \gamma_{166}w_{6} + \gamma_{\{521,785,828\}}w_{7}\\
&\quad + \gamma_{\{609,638,787,1139\}}w_{8} + \gamma_{\{613,638,789,1141\}}w_{9} + \gamma_{791}w_{10} + \gamma_{168}w_{11}\\
&\quad + \gamma_{\{535,799,848\}}w_{12} + \gamma_{801}w_{13} + \gamma_{803}w_{14} + \gamma_{\{377,591,856,1055,1104,1210\}}w_{15}\\
&\quad + \gamma_{858}w_{16} + \gamma_{\{381,595,615,860,1059,1108\}}w_{17} + \gamma_{185}w_{18} + \gamma_{\{602,609,867,1115,1210\}}w_{19}\\
&\quad + \gamma_{869}w_{20} + \gamma_{\{607,872\}}w_{21} + \gamma_{\{609,874\}}w_{22} + \gamma_{876}w_{23} + \gamma_{\{613,878\}}w_{24}\\
&\quad + \gamma_{\{615,880\}}w_{25} + \gamma_{187}w_{26} + \gamma_{\{607,889\}}w_{27} + \gamma_{\{609,626,638,891,1139\}}w_{28}\\
&\quad + \gamma_{\{613,628,638,893,1141\}}w_{29} + \gamma_{\{615,895\}}w_{30} + \gamma_{\{636,901\}}w_{31} + \gamma_{\{638,903\}}w_{32}\\
&\quad + \gamma_{905}w_{33} + \gamma_{\{636,910\}}w_{34} + \gamma_{\{638,912\}}w_{35} + \gamma_{\{651,961\}}w_{36} + \gamma_{963}w_{37} + \gamma_{965}w_{38}\\
&\quad + \gamma_{967}w_{39} + \gamma_{194}w_{40} + \gamma_{975}w_{41} + \gamma_{977}w_{42} + \gamma_{979}w_{43} + \gamma_{981}w_{44} + \gamma_{987}w_{45}\\
&\quad + \gamma_{989}w_{46} + \gamma_{993}w_{47} + \gamma_{996}w_{48} + \gamma_{998}w_{49} + \gamma_{1002}w_{50} + \gamma_{\{733,1036\}}w_{51}\\
&\quad + \gamma_{1038}w_{52} + \gamma_{1040}w_{53} + \gamma_{1045}w_{54} + \gamma_{1047}w_{55} + \gamma_{1051}w_{56} \equiv 0.
\end{align*}
By computing from these relations we obtain
\begin{equation}\label{c68}
\gamma_j = 0 \mbox{ for } j \in \mathbb J_7,\ \gamma_i= \gamma_j \mbox{ for } (i,j) \in \mathbb K_7,
\end{equation}
where $\mathbb J_7 = \{$12, 13, 75, 76, 91, 92, 93, 94, 95, 96, 147, 148, 163, 164, 165, 166, 167, 168, 172, 173, 174, 175, 176, 177, 178, 179, 184, 185, 186, 187, 193, 194, 337, 484, 507, 508, 514, 522, 524, 526, 527, 528, 534, 536, 537, 538, 539, 748, 771, 772, 778, 786, 788, 790, 791, 792, 798, 800, 801, 802, 803, 804, 805, 810, 815, 829, 831, 833, 834, 835, 847, 849, 850, 851, 852, 858, 859, 869, 876, 877, 900, 904, 905, 911, 962, 963, 964, 965, 966, 967, 968, 974, 975, 976, 977, 978, 979, 980, 981, 986, 987, 988, 989, 992, 993, 995, 996, 997, 998, 1001, 1002, 1035, 1037, 1038, 1039, 1040, 1044, 1045, 1046, 1047, 1050, 1051$\}$ and $\mathbb K_7 = \{$(607,872), (607,889), (608,873), (608,888), (609,874), (610,875), (610,890), (613,878), (614,879), (614,892), (615,880), (615,895), (616,881), (616,894), (636,901), (636,910), (637,902), (637,909), (638,903), (638,912), (651,961), (733,1036)$\}$.

Apply the homomorphisms $p_{(1;(2,4,5))}$ and $p_{(1;(3,4,5))}$   to \eqref{ctd612} using \eqref{c61}-\eqref{c68} to obtain
\begin{align*}
&p_{(1;(2,4,5))}(\theta) \equiv \gamma_{\{378,509,773,808,1110,1111,1171,1173\}}w_{1}\\
&\quad + \gamma_{\{382,510,774,811,1112,1113,1174,1176\}}w_{2} + \gamma_{\{517,781,822\}}w_{3}\\
&\quad + \gamma_{\{341,385,488,518,782,824\}}w_{4} + \gamma_{\{519,783,825\}}w_{5} + \gamma_{\{345,387,492,520,784,827\}}w_{6}\\
&\quad + \gamma_{793}w_{7} + \gamma_{\{530,794,843\}}w_{8} + \gamma_{\{531,795,844\}}w_{9} + \gamma_{\{532,796,845\}}w_{10}\\
&\quad + \gamma_{797}w_{11} + \gamma_{169}w_{12} + \gamma_{170}w_{13} + \gamma_{171}w_{14}\\
&\quad + \gamma_{\{378,597,651,862,923,1110,1171,1216,1254,1256\}}w_{15}\\
&\quad + \gamma_{\{598,733,863,925,1111,1173,1216,1257,1299,1301\}}w_{16}\\
&\quad + \gamma_{\{382,599,733,864,926,1112,1174,1218,1258,1299,1301\}}w_{17}\\
&\quad + \gamma_{\{600,865,928,1113,1176,1218,1259\}}w_{18} + \gamma_{870}w_{19} + \gamma_{871}w_{20} + \gamma_{882}w_{21} + \gamma_{883}w_{22}\\
&\quad + \gamma_{884}w_{23} + \gamma_{885}w_{24} + \gamma_{886}w_{25} + \gamma_{887}w_{26} + \gamma_{896}w_{27} + \gamma_{897}w_{28} + \gamma_{898}w_{29}\\
&\quad + \gamma_{899}w_{30} + \gamma_{906}w_{31} + \gamma_{907}w_{32} + \gamma_{908}w_{33} + \gamma_{188}w_{34} + \gamma_{189}w_{35} + \gamma_{\{969,1254\}}w_{36}\\
&\quad + \gamma_{\{696,733,970,1018,1256,1299,1301\}}w_{37} + \gamma_{\{697,733,971,1019,1257,1299,1301,1302\}}w_{38}\\
&\quad + \gamma_{\{698,972,1020,1258,1302\}}w_{39} + \gamma_{\{973,1259\}}w_{40} + \gamma_{982}w_{41} + \gamma_{983}w_{42} + \gamma_{984}w_{43}\\
&\quad + \gamma_{985}w_{44} + \gamma_{990}w_{45} + \gamma_{991}w_{46} + \gamma_{994}w_{47} + \gamma_{999}w_{48} + \gamma_{1000}w_{49} + \gamma_{195}w_{50}\\
&\quad + \gamma_{\{1041,1299\}}w_{51} + \gamma_{\{1042,1301\}}w_{52} + \gamma_{\{1043,1302\}}w_{53} + \gamma_{1048}w_{54} + \gamma_{1049}w_{55}\\
&\quad + \gamma_{1052}w_{56} \equiv 0,\\
&p_{(1;(3,4,5))}(\theta) \equiv (\gamma_{\{289,290,291,307,308,309,341,509,806,808\}}\\
&\quad \hskip 3cm + \gamma_{\{1167,1168,1171,1173,1355,1356,1357,1397,1398,1399\}})w_{1}\\
&\quad + \gamma_{\{289,290,291,307,308,309,345,510,807,811,1169,1170,1174,1176,1358,1359,1400,1401\}}w_{2}\\
&\quad + (\gamma_{\{290,308,314,315,341,517,818,822,1071,1118,1130,1131\}}\\
&\quad \hskip 3cm + \gamma_{\{1167,1171,1177,1183,1184,1356,1398,1408,1409\}})w_{3}\\
&\quad + (\gamma_{\{290,291,308,309,314,317,341,512,518,813,819,824,1097,1118,1132,1154,1155,1168,1173\}}\\
&\quad \hskip 1cm + \gamma_{\{1177,1185,1186, 1197, 1200, 1201, 1356, 1398, 1408, 1409, 1410, 1414, 1417, 1418\}})w_{4}\\
&\quad + (\gamma_{\{290,291,308,309,314,317,345,519,820,825,1097,1119,1134,1154,1155,1169,1174\}}\\
&\quad \hskip 2.5cm + \gamma_{\{1178,1187,1197,1200, 1201, 1358, 1400, 1411, 1414, 1417, 1418 \}})w_{5}\\
&\quad + (\gamma_{\{290,308,314,315,345,516,520,817,821,827,1119,1135\}}\\
&\quad \hskip 3cm + \gamma_{\{1170,1176,1178,1188,1358,1400,1410,1411\}})w_{6}\\
&\quad + \gamma_{\{289,307,314,317,521,828,836,1071,1130,1183,1195,1355,1397,1408,1413\}}w_{7}\\
&\quad + (\gamma_{\{530,837,843,1097,1131,1154,1155,1184,1195,1196\}}\\
&\quad \hskip 3cm + \gamma_{\{1200,1201,1356,1398,1408,1409,1416,1417,1418\}})w_{8}\\
&\quad + (\gamma_{\{289,291,307,309,314,315,531,838,844,1097,1132,1154,1155,1156\}}\\
&\quad \hskip 1.5cm + \gamma_{\{1185,1186,1196,1197,1200,1201,1202,1357, 1399, 1409, 1410, 1413, 1414\}})w_{9}\\
&\quad + \gamma_{\{532,839,845,1134,1156,1187,1197,1198,1202,1358,1400,1410,1411,1416,1417,1418\}}w_{10}\\
&\quad + \gamma_{\{291,309,315,317,840,1135,1188,1198,1359,1401,1411,1414\}}w_{11}\\
&\quad + \gamma_{\{289,307,314,317,535,848,853,1097,1154,1200,1207,1417,1420\}}w_{12}\\
&\quad + \gamma_{\{290,308,314,315,854,1155,1201,1207,1208,1416,1417,1418\}}w_{13}\\
&\quad + \gamma_{\{291,309,315,317,855,1156,1202,1208,1418,1420\}}w_{14}\\
&\quad + \gamma_{\{377,378,591,597,919,923,1209,1210,1213,1216,1248,1249,1254,1256\}}w_{15}\\
&\quad + \gamma_{\{385,592,598,920,925,1209,1210,1213,1216,1250,1257,1284,1296,1297,1299,1301\}}w_{16}\\
&\quad + \gamma_{\{381,382,595,599,615,921,926,1211,1214,1218,1251,1252,1258,1284,1296,1297,1299,1301\}}w_{17}\\
&\quad + \gamma_{\{387,596,600,615,922,928,1211,1214,1218,1253,1259\}}w_{18}\\
&\quad + \gamma_{\{378,601,602,609,610,929,1209,1210,1213,1216,1228,1249,1256,1260,1268\}}w_{19}\\
&\quad + \gamma_{\{382,603,616,930,1211,1214,1218,1230,1252,1258,1262\}}w_{20}\\
&\quad + \gamma_{\{608,935,1227,1248,1254,1260\}}w_{21} + \gamma_{\{607,610,936,1227,1228,1249,1256,1260,1261,1268\}}w_{22}\\
&\quad + \gamma_{\{609,937,1228,1250,1257,1261\}}w_{23} + \gamma_{\{614,938,1229,1251,1262,1268\}}w_{24}\\
&\quad + \gamma_{\{613,616,939,1229,1230,1252,1258,1262,1263\}}w_{25} + \gamma_{\{615,940,1230,1253,1259,1263\}}w_{26}\\
&\quad + \gamma_{\{947,1227,1265,1296,1299,1303\}}w_{27} + \gamma_{\{609,626,638,948,1228,1236,1265,1268,1270\}}w_{28}\\
&\quad + \gamma_{\{613,628,638,949,1229,1236,1266,1268,1270\}}w_{29} + \gamma_{\{950,1230,1266,1298,1302,1304\}}w_{30}\\
&\quad + \gamma_{\{637,952,1235,1270,1296,1299,1303\}}w_{31} + \gamma_{\{636,953,1235,1236,1268,1270,1271\}}w_{32}\\
&\quad + \gamma_{\{638,954,1236,1271,1298,1302,1304\}}w_{33} + \gamma_{\{959,1235,1273\}}w_{34} + \gamma_{\{960,1236,1273\}}w_{35}\\
&\quad + \gamma_{1011}w_{36} + \gamma_{\{696,1012,1018,1283,1296,1297,1299,1301\}}w_{37}\\
&\quad + \gamma_{\{697,1013,1019,1283,1284,1296,1297,1298,1299,1301,1302\}}w_{38}\\
&\quad + \gamma_{\{698,1014,1020,1284,1298,1302\}}w_{39} + \gamma_{1015}w_{40} + \gamma_{\{1022,1296,1299,1303\}}w_{41}\\
&\quad + \gamma_{\{1023,1283,1304\}}w_{42} + \gamma_{\{1024,1284,1297,1301,1303\}}w_{43} + \gamma_{\{1025,1298,1302,1304\}}w_{44}\\
&\quad + \gamma_{\{1027,1296,1299,1303,1306\}}w_{45} + \gamma_{\{1028,1298,1302,1304,1306\}}w_{46}\\
&\quad + \gamma_{\{1030,1297,1301,1303,1304\}}w_{47} + \gamma_{\{1031,1306\}}w_{48} + \gamma_{\{1032,1306\}}w_{49} + \gamma_{1034}w_{50}\\
&\quad + \gamma_{205}w_{51} + \gamma_{206}w_{52} + \gamma_{207}w_{53} + \gamma_{212}w_{54} + \gamma_{213}w_{55} + \gamma_{215}w_{56} \equiv 0.
\end{align*}
By computing from these relations we obtain
\begin{equation}\label{c69}
\gamma_j = 0 \mbox{ for } j \in \mathbb J_8,\ \gamma_i= \gamma_j \mbox{ for } (i,j) \in \mathbb K_8,
\end{equation}
where $\mathbb J_8 = \{$23, 59, 86, 97, 98, 99, 128, 133, 134, 135, 158, 169, 170, 171, 180, 181, 182, 183, 188, 189, 195, 200, 205, 206, 207, 208, 209, 210, 211, 212, 213, 215, 529, 533, 793, 797, 841, 846, 870, 871, 882, 883, 884, 885, 886, 887, 896, 897, 898, 899, 906, 907, 908, 982, 983, 984, 985, 990, 991, 994, 999, 1000, 1011, 1015, 1034, 1048, 1049, 1052$\}$ and $\mathbb K_8 = \{$(969,1254), (973,1259), (1031,1032), (1031,1306), (1041,1299), (1042,1301), (1043,1302)$\}$.

Apply the homomorphism $p_{(2;(3,4,5))}$ to \eqref{ctd612} and using \eqref{c61}-\eqref{c69} to get
\begin{align*}
&p_{(2;(3,4,5))}(\theta) \equiv (\gamma_{\{385,651,695,733,773,806,808,856,857,862,863,919,920\}}\\
&\quad \hskip 2.5cm + \gamma_{\{923,925,969,970,971,1012,1013,1018,1019, 1041, 1042, 1043\}})w_{1}\\
&\quad + (\gamma_{\{387,699,733,774,807,811,860,861,864,865,921,922,926\}}\\
&\quad \hskip 2.5cm + \gamma_{\{928,972,973,1014,1020,1041,1042,1043\}})w_{2}\\
&\quad + (\gamma_{\{385,393,511,605,607,608,609,617,618,775,781,818,822,856,862,866\}}\\
&\quad \hskip 2.5cm + \gamma_{\{867,919,923,929,935,936, 970, 1012, 1018, 1022, 1023, 1042\}})w_{3}\\
&\quad + (\gamma_{\{385,419,513,605,610,613,614,619,636,637,638,641,642,776,777,782,813,819,824,857,863,866\}}\\
&\quad  + \gamma_{\{867, 893, 920, 925, 929, 937, 938, 949, 952, 953, 970, 1012, 1018, 1022, 1023, 1024, 1028, 1042, 1043\}})w_{4}\\
&\quad + (\gamma_{\{387,419,515,606,615,621,636,637,638,641,642,779,783,820,825,860,864,868\}}\\
&\quad \hskip 1cm + \gamma_{\{893,921,926,930, 939, 949, 952, 953, 972, 1014, 1020, 1025, 1028, 1042, 1043\}})w_{5}\\
&\quad + (\gamma_{\{387,606,616,622,780,784,817,821,827,861,865,868,922\}}\\
&\quad \hskip 3cm + \gamma_{\{928,930,940,972,1014,1020,1024,1025,1042\}})w_{6}\\
&\quad + \gamma_{\{393,608,617,651,695,733,785,828,836,935,947,969,1022,1027,1041\}}w_{7}\\
&\quad + (\gamma_{\{419,523,607,609,610,618,636,637,638,641,642,787,794,837,843\}}\\
&\quad \hskip 3cm + \gamma_{\{891,936,947,948,952,953,970, 1012, 1018, 1022, 1023, 1030\}})w_{8}\\
&\quad + (\gamma_{\{419,525,613,614,619,636,637,638,641,642,643,733,789,795,838,844,891,893\}}\\
&\quad \hskip 1cm + \gamma_{\{937,938,948,949, 952, 953, 954, 971, 1013, 1019, 1023, 1024, 1027, 1028, 1041, 1043\}})w_{9}\\
&\quad + \gamma_{\{616,621,643,796,839,845,893,939,949,950,954,972,1014,1020,1024,1025,1030\}}w_{10}\\
&\quad + \gamma_{\{615,622,699,840,940,950,973,1025,1028,1043\}}w_{11}\\
&\quad + \gamma_{\{419,637,641,733,799,848,853,952,959,1031,1041\}}w_{12}\\
&\quad + \gamma_{\{636,642,854,953,959,960,1030,1042\}}w_{13} + \gamma_{\{638,643,855,954,960,1031,1043\}}w_{14}\\
&\quad + \gamma_{\{378,969,1041,1055,1104,1110,1167,1171,1209,1210,1213,1216,1248,1283,1296\}}w_{15}\\
&\quad + (\gamma_{\{341,1042,1043,1105,1111,1168,1173,1209,1210,1213,1216\}}\\
&\quad \hskip 3cm + \gamma_{\{1249,1250,1256,1257,1283,1297,1298\}})w_{16}\\
&\quad + \gamma_{\{382,615,1041,1059,1108,1112,1169,1174,1211,1214,1218,1251,1284,1296\}}w_{17}\\
&\quad + (\gamma_{\{345,615,973,1042,1043,1109,1113,1170,1176,1211\}}\\
&\quad \hskip 3cm + \gamma_{\{1214,1218,1252,1253,1258,1284,1297,1298\}})w_{18}\\
&\quad + \gamma_{\{378,609,610,1041,1114,1115,1177,1209,1210,1213,1216,1228,1260,1261,1296,1303,1304\}}w_{19}\\
&\quad + \gamma_{\{382,616,1116,1178,1211,1214,1218,1230,1262,1263,1284,1303,1304\}}w_{20}\\
&\quad + \gamma_{\{608,969,1041,1183,1227,1248,1296\}}w_{21}\\
&\quad + \gamma_{\{607,610,1041,1042,1184,1227,1228,1249,1256,1283,1296,1297\}}w_{22}\\
&\quad + \gamma_{\{609,1042,1185,1228,1250,1257,1268,1283,1297\}}w_{23} + \gamma_{\{614,1186,1229,1251,1284\}}w_{24}\\
&\quad + \gamma_{\{613,616,1043,1187,1229,1230,1252,1258,1268,1284,1298\}}w_{25}\\
&\quad + \gamma_{\{615,973,1043,1188,1230,1253,1298\}}w_{26} + \gamma_{\{1031,1195,1227,1260,1265,1303\}}w_{27}\\
&\quad + \gamma_{\{609,638,1031,1139,1196,1228,1236,1261,1265,1270,1271,1283,1303\}}w_{28}\\
&\quad + \gamma_{\{613,638,1031,1141,1197,1229,1236,1262,1266,1270,1271,1284,1304\}}w_{29}\\
&\quad + \gamma_{\{1031,1198,1230,1263,1266,1304\}}w_{30} + \gamma_{\{637,1041,1200,1235,1296\}}w_{31}\\
&\quad + \gamma_{\{636,1042,1201,1235,1236,1268,1297\}}w_{32} + \gamma_{\{638,1043,1202,1236,1298\}}w_{33}\\
&\quad + \gamma_{\{1031,1207,1235,1270,1273,1303\}}w_{34} + \gamma_{\{1031,1208,1236,1271,1273,1304\}}w_{35}\\
&\quad + \gamma_{1397}w_{36} + \gamma_{1398}w_{37} + \gamma_{1399}w_{38} + \gamma_{1400}w_{39} + \gamma_{1401}w_{40} + \gamma_{1408}w_{41}\\
&\quad + \gamma_{1409}w_{42} + \gamma_{1410}w_{43} + \gamma_{1411}w_{44} + \gamma_{1413}w_{45} + \gamma_{1414}w_{46} + \gamma_{1416}w_{47}\\
&\quad + \gamma_{1417}w_{48} + \gamma_{1418}w_{49} + \gamma_{1420}w_{50} + \gamma_{307}w_{51} + \gamma_{308}w_{52} + \gamma_{309}w_{53}\\
&\quad + \gamma_{314}w_{54} + \gamma_{315}w_{55} + \gamma_{317}w_{56} \equiv 0.
\end{align*}
\begin{align}\label{c70}
	\gamma_j = 0 \mbox{ for } j \in \mathbb J_9,
\end{align}
where $\mathbb J_9 = \{$269, 284, 289, 290, 291, 302, 307, 308, 309, 310, 311, 312, 313, 314, 315, 317, 1311, 1347, 1355, 1356, 1357, 1358, 1359, 1389, 1397, 1398, 1399, 1400, 1401, 1402, 1403, 1404, 1405, 1406, 1407, 1408, 1409, 1410, 1411, 1413, 1414, 1416, 1417, 1418, 1420$\}$.

By applying the homomorphism $p_{(1;(2,3,4,5))}$ to \eqref{ctd612}, using Theorem \ref{dlsig} and the relation \eqref{c61}-\eqref{c70}, we get
\begin{align*}
&p_{(1;(2,3,4,5))}(\theta) \equiv \gamma_{\{341,378,509,773,806,808,1055,1104,1105,1110,1111,1167,1168,1171,1173\}}w_{1}\\
&\quad + \gamma_{\{345,382,510,774,807,811,1059,1108,1109,1112,1113,1169,1170,1174,1176\}}w_{2}\\
&\quad + \gamma_{\{517,607,608,609,775,781,818,822,1055,1104,1110,1114,1115,1167,1171,1177,1183,1184\}}w_{3}\\
&\quad + (\gamma_{\{378,385,488,512,518,610,613,614,636,637,638,776,777,782,813,819,824,1105\}}\\
&\quad \hskip 2cm + \gamma_{\{1111,1114,1115,1141, 1168, 1173, 1177, 1185, 1186, 1197, 1200, 1201\}})w_{4}\\
&\quad + (\gamma_{\{519,615,636,637,638,779,783,820,825,1059,1108,1112,1116\}}\\
&\quad \hskip 2cm + \gamma_{\{1141,1169,1174,1178,1187,1197,1200,1201\}})w_{5}\\
&\quad + \gamma_{\{382,387,492,516,520,616,780,784,817,821,827,1109,1113,1116,1170,1176,1178,1188\}}w_{6}\\
&\quad + \gamma_{\{521,608,785,828,836,1183,1195\}}w_{7}\\
&\quad + \gamma_{\{530,607,609,610,636,637,638,787,794,837,843,1139,1184,1195,1196,1200,1201\}}w_{8}\\
&\quad + \gamma_{\{531,613,614,636,637,638,789,795,838,844,1139,1141,1185,1186,1196,1197,1200,1201,1202\}}w_{9}\\
&\quad + \gamma_{\{532,616,796,839,845,1141,1187,1197,1198,1202\}}w_{10}\\
&\quad + \gamma_{\{615,840,1188,1198\}}w_{11} + \gamma_{\{535,637,799,848,853,1200,1207\}}w_{12}\\
&\quad + \gamma_{\{636,854,1201,1207,1208\}}w_{13} + \gamma_{\{638,855,1202,1208\}}w_{14}\\
&\quad + (\gamma_{\{377,378,591,597,856,862,919,923,969,1055,1104,1110,1167\}}\\
&\quad \hskip 2cm + \gamma_{\{1171,1209,1210,1213,1216,1248,1249,1256\}})w_{15}\\
&\quad + (\gamma_{\{341,385,592,598,857,863,920,925,1041,1042,1105,1111,1168\}}\\
&\quad \hskip 2cm + \gamma_{\{1173,1209,1210,1213,1216,1250,1251,1257,1284, 1296, 1297\}})w_{16}\\
&\quad + (\gamma_{\{381,382,595,599,615,860,864,921,926,1041,1042,1059,1108\}}\\
&\quad \hskip 2cm + \gamma_{\{1112,1169,1174,1211,1214,1218,1252,1258,1284, 1296, 1297\}})w_{17}\\
&\quad + \gamma_{\{345,387,596,600,615,861,865,922,928,973,1109,1113,1170,1176,1211,1214,1218,1253\}}w_{18}\\
&\quad + (\gamma_{\{378,601,602,609,610,866,867,929,1114,1115,1177\}}\\
&\quad \hskip 2cm + \gamma_{\{1209,1210,1213,1216,1228,1249,1256,1260,1268\}})w_{19}\\
&\quad + \gamma_{\{382,603,616,868,930,1116,1178,1211,1214,1218,1230,1252,1258,1262\}}w_{20}\\
&\quad + \gamma_{\{608,935,969,1183,1227,1248,1260\}}w_{21}\\
&\quad + \gamma_{\{607,610,936,1184,1227,1228,1249,1256,1260,1261,1268\}}w_{22}\\
&\quad + \gamma_{\{609,937,1185,1228,1250,1257,1261\}}w_{23} + \gamma_{\{614,938,1186,1229,1251,1262,1268\}}w_{24}\\
&\quad + \gamma_{\{613,616,939,1187,1229,1230,1252,1258,1262,1263\}}w_{25}\\
&\quad + \gamma_{\{615,940,973,1188,1230,1253,1263\}}w_{26} + \gamma_{\{947,1041,1195,1227,1265,1296,1303\}}w_{27}\\
&\quad + \gamma_{\{609,626,638,891,948,1139,1196,1228,1236,1265,1268,1270\}}w_{28}\\
&\quad + \gamma_{\{613,628,638,893,949,1141,1197,1229,1236,1266,1268,1270\}}w_{29}\\
&\quad + \gamma_{\{950,1043,1198,1230,1266,1298,1304\}}w_{30} + \gamma_{\{637,952,1041,1200,1235,1270,1296,1303\}}w_{31}\\
&\quad + \gamma_{\{636,953,1201,1235,1236,1268,1270,1271\}}w_{32}\\
&\quad + \gamma_{\{638,954,1043,1202,1236,1271,1298,1304\}}w_{33}\\
&\quad + \gamma_{\{959,1207,1235,1273\}}w_{34} + \gamma_{\{960,1208,1236,1273\}}w_{35} + \gamma_{1248}w_{36}\\
&\quad + \gamma_{\{696,970,1012,1018,1041,1042,1249,1256,1283,1296,1297\}}w_{37}\\
&\quad + \gamma_{\{697,971,1013,1019,1041,1042,1043,1250,1251,1257,1283,1284,1296,1297,1298\}}w_{38}\\
&\quad + \gamma_{\{698,972,1014,1020,1043,1252,1258,1284,1298\}}w_{39}\\
&\quad + \gamma_{1253}w_{40} + \gamma_{\{1022,1041,1260,1296,1303\}}w_{41} + \gamma_{\{1023,1261,1283,1304\}}w_{42}\\
&\quad + \gamma_{\{1024,1042,1262,1284,1297,1303\}}w_{43} + \gamma_{\{1025,1043,1263,1298,1304\}}w_{44}\\
&\quad + \gamma_{\{1027,1031,1041,1265,1296,1303\}}w_{45} + \gamma_{\{1028,1031,1043,1266,1298,1304\}}w_{46}\\
&\quad + \gamma_{\{1030,1042,1268,1297,1303,1304\}}w_{47} + \gamma_{1270}w_{48} + \gamma_{1271}w_{49} + \gamma_{1273}w_{50}\\
&\quad + \gamma_{1296}w_{51} + \gamma_{1297}w_{52} + \gamma_{1298}w_{53} + \gamma_{1303}w_{54} + \gamma_{1304}w_{55} + \gamma_{1031}w_{56} \equiv 0.
\end{align*}

By a direct computation from the above equalities we obtain $\gamma_t = 0$ for $1\leqslant t \leqslant 1456$. The theorem is proved.

Let

\section{Appendix}\label{s5}
\setcounter{equation}{0} 

\subsection{A set of $\mathcal A$-generators for $P_5$ in degree 108}\label{s51}\

\medskip
It is well known that
$$B_5(108) = B_5((4)|^3|(2)|^2|(1))\bigcup B_5((4)|^4|(1)|^2)\bigcup B_5((4)|^4|(3)).$$

From Corollary \ref{hq}, we have $$B_5((4)|^4|(1)|^2) = \Phi(B_4((4)|^4|(1)|^2)), \ B_5((4)|^4|(3)) = \Phi(B_4((4)|^4|(3))).$$
By \cite{su2}, $B_4((4)|^4|(3)) = \{x_1^{15}x_2^{31}x_3^{31}x_4^{31},\ x_1^{31}x_2^{15}x_3^{31}x_4^{31},\ x_1^{31}x_2^{31}x_3^{15}x_4^{31},\ x_1^{31}x_2^{31}x_3^{31}x_4^{15}\}$ and $B_4((4)|^4|(1)|^2)$ is the set of 10 monomials which are determined as follows:

\medskip
\centerline{
\begin{tabular}{lllll}
$1.\ x_1^{15}x_2^{15}x_3^{15}x_4^{63}$& $2.\ x_1^{15}x_2^{15}x_3^{31}x_4^{47}$& $3.\ x_1^{15}x_2^{15}x_3^{63}x_4^{15}$& $4.\ x_1^{15}x_2^{31}x_3^{15}x_4^{47} $\cr  $5.\ x_1^{15}x_2^{31}x_3^{47}x_4^{15}$& $6.\ x_1^{15}x_2^{63}x_3^{15}x_4^{15}$& $7.\ x_1^{31}x_2^{15}x_3^{15}x_4^{47}$& $8.\ x_1^{31}x_2^{15}x_3^{47}x_4^{15} $\cr  $9.\ x_1^{31}x_2^{47}x_3^{15}x_4^{15}$& $10.\ x_1^{63}x_2^{15}x_3^{15}x_4^{15}$.& &\cr 
\end{tabular}} 

\medskip
Hence, we get $|B_5((4)|^4|(1)|^2)| = 310$ and $|B_5((4)|^4|(3))| = 124$.
We have 
$$B_5((4)|^3|(2)|^2|(1)) =  B_5^0((4)|^3|(2)|^2|(1))\bigcup B_5^+((4)|^3|(2)|^2|(1)),$$ 
where $B_5^0((4)|^3|(2)|^2|(1)) = \Phi^0(B_4((4)|^3|(2)|^2|(1))))$. By \cite{su2}, $B_4((4)|^3|(2)|^2|(1)))$ is the set of $56$ monomials $w_t,\ 1 \leqslant t \leqslant 56$, which are determined as follows:

\medskip
\centerline{\begin{tabular}{lll}
$w_{1} =  x_1^{7}x_2^{7}x_3^{31}x_4^{63}$& $w_{2} =  x_1^{7}x_2^{7}x_3^{63}x_4^{31}$& $w_{3} =  x_1^{7}x_2^{15}x_3^{23}x_4^{63} $\cr  $w_{4} =  x_1^{7}x_2^{15}x_3^{31}x_4^{55}$& $w_{5} =  x_1^{7}x_2^{15}x_3^{55}x_4^{31}$& $w_{6} =  x_1^{7}x_2^{15}x_3^{63}x_4^{23} $\cr  $w_{7} =  x_1^{7}x_2^{31}x_3^{7}x_4^{63}$& $w_{8} =  x_1^{7}x_2^{31}x_3^{15}x_4^{55}$& $w_{9} =  x_1^{7}x_2^{31}x_3^{31}x_4^{39} $\cr  $w_{10} =  x_1^{7}x_2^{31}x_3^{47}x_4^{23}$& $w_{11} =  x_1^{7}x_2^{31}x_3^{63}x_4^{7}$& $w_{12} =  x_1^{7}x_2^{63}x_3^{7}x_4^{31} $\cr  $w_{13} =  x_1^{7}x_2^{63}x_3^{15}x_4^{23}$& $w_{14} =  x_1^{7}x_2^{63}x_3^{31}x_4^{7}$& $w_{15} =  x_1^{15}x_2^{7}x_3^{23}x_4^{63} $\cr  $w_{16} =  x_1^{15}x_2^{7}x_3^{31}x_4^{55}$& $w_{17} =  x_1^{15}x_2^{7}x_3^{55}x_4^{31}$& $w_{18} =  x_1^{15}x_2^{7}x_3^{63}x_4^{23} $\cr  $w_{19} =  x_1^{15}x_2^{15}x_3^{23}x_4^{55}$& $w_{20} =  x_1^{15}x_2^{15}x_3^{55}x_4^{23}$& $w_{21} =  x_1^{15}x_2^{23}x_3^{7}x_4^{63} $\cr  $w_{22} =  x_1^{15}x_2^{23}x_3^{15}x_4^{55}$& $w_{23} =  x_1^{15}x_2^{23}x_3^{31}x_4^{39}$& $w_{24} =  x_1^{15}x_2^{23}x_3^{39}x_4^{31} $\cr
\end{tabular}} 

\medskip
We have $|B_5^0((4)|^3|(2)|^2|(1))| = 5\times 56 = 280$.

By \cite{su2} and Proposition \ref{mdmo}, we have $|B_4^+((3)|^3|(1)|^2)| = 66$ and
$$\mathcal B := \left\{x_i^{63}\theta_{J_i}(x): x \in B_4^+((3)|^3|(1)|^2), 1 \leqslant i \leqslant 5\right\}\subset B_5^+((4)|^3|(2)|^2|(1)).$$
Then, $\mathcal B$ is the set of 330 monomials $a_t,\, 1 \leqslant t \leqslant 330$, which are determined as follows:

\medskip
\centerline{\begin{tabular}{lll}
$1.\ x_1x_2^{6}x_3^{7}x_4^{31}x_5^{63}$& $2.\ x_1x_2^{6}x_3^{7}x_4^{63}x_5^{31}$& $3.\ x_1x_2^{6}x_3^{15}x_4^{23}x_5^{63} $\cr  $4.\ x_1x_2^{6}x_3^{15}x_4^{63}x_5^{23}$& $5.\ x_1x_2^{6}x_3^{31}x_4^{7}x_5^{63}$& $6.\ x_1x_2^{6}x_3^{31}x_4^{63}x_5^{7} $\cr  $7.\ x_1x_2^{6}x_3^{63}x_4^{7}x_5^{31}$& $8.\ x_1x_2^{6}x_3^{63}x_4^{15}x_5^{23}$& $9.\ x_1x_2^{6}x_3^{63}x_4^{31}x_5^{7} $\cr   $10.\ x_1x_2^{7}x_3^{6}x_4^{31}x_5^{63}$& $11.\ x_1x_2^{7}x_3^{6}x_4^{63}x_5^{31}$& $12.\ x_1x_2^{7}x_3^{7}x_4^{30}x_5^{63} $\cr  $13.\ x_1x_2^{7}x_3^{7}x_4^{63}x_5^{30}$& $14.\ x_1x_2^{7}x_3^{14}x_4^{23}x_5^{63}$& $15.\ x_1x_2^{7}x_3^{14}x_4^{63}x_5^{23} $\cr  $16.\ x_1x_2^{7}x_3^{15}x_4^{22}x_5^{63}$& $17.\ x_1x_2^{7}x_3^{15}x_4^{63}x_5^{22}$& $18.\ x_1x_2^{7}x_3^{30}x_4^{7}x_5^{63} $\cr   
\end{tabular}}
\centerline{\begin{tabular}{lll}
$19.\ x_1x_2^{7}x_3^{30}x_4^{63}x_5^{7}$& $20.\ x_1x_2^{7}x_3^{31}x_4^{6}x_5^{63}$& $21.\ x_1x_2^{7}x_3^{31}x_4^{63}x_5^{6} $\cr  $22.\ x_1x_2^{7}x_3^{63}x_4^{6}x_5^{31}$& $23.\ x_1x_2^{7}x_3^{63}x_4^{7}x_5^{30}$& $24.\ x_1x_2^{7}x_3^{63}x_4^{14}x_5^{23} $\cr   $25.\ x_1x_2^{7}x_3^{63}x_4^{15}x_5^{22}$& $26.\ x_1x_2^{7}x_3^{63}x_4^{30}x_5^{7}$& $27.\ x_1x_2^{7}x_3^{63}x_4^{31}x_5^{6} $\cr   $28.\ x_1x_2^{14}x_3^{7}x_4^{23}x_5^{63}$& $29.\ x_1x_2^{14}x_3^{7}x_4^{63}x_5^{23}$& $30.\ x_1x_2^{14}x_3^{23}x_4^{7}x_5^{63} $\cr   $31.\ x_1x_2^{14}x_3^{23}x_4^{63}x_5^{7}$& $32.\ x_1x_2^{14}x_3^{63}x_4^{7}x_5^{23}$& $33.\ x_1x_2^{14}x_3^{63}x_4^{23}x_5^{7} $\cr  $34.\ x_1x_2^{15}x_3^{6}x_4^{23}x_5^{63}$& $35.\ x_1x_2^{15}x_3^{6}x_4^{63}x_5^{23}$& $36.\ x_1x_2^{15}x_3^{7}x_4^{22}x_5^{63} $\cr  $37.\ x_1x_2^{15}x_3^{7}x_4^{63}x_5^{22}$& $38.\ x_1x_2^{15}x_3^{22}x_4^{7}x_5^{63}$& $39.\ x_1x_2^{15}x_3^{22}x_4^{63}x_5^{7} $\cr  $40.\ x_1x_2^{15}x_3^{23}x_4^{6}x_5^{63}$& $41.\ x_1x_2^{15}x_3^{23}x_4^{63}x_5^{6}$& $42.\ x_1x_2^{15}x_3^{63}x_4^{6}x_5^{23} $\cr
\end{tabular}}
\centerline{\begin{tabular}{lll}
$163.\ x_1^{7}x_2^{3}x_3^{5}x_4^{30}x_5^{63}$& $164.\ x_1^{7}x_2^{3}x_3^{5}x_4^{63}x_5^{30}$& $165.\ x_1^{7}x_2^{3}x_3^{13}x_4^{22}x_5^{63} $\cr  $166.\ x_1^{7}x_2^{3}x_3^{13}x_4^{63}x_5^{22}$& $167.\ x_1^{7}x_2^{3}x_3^{29}x_4^{6}x_5^{63}$& $168.\ x_1^{7}x_2^{3}x_3^{29}x_4^{63}x_5^{6} $\cr  $169.\ x_1^{7}x_2^{3}x_3^{63}x_4^{5}x_5^{30}$& $170.\ x_1^{7}x_2^{3}x_3^{63}x_4^{13}x_5^{22}$& $171.\ x_1^{7}x_2^{3}x_3^{63}x_4^{29}x_5^{6} $\cr  $172.\ x_1^{7}x_2^{7}x_3x_4^{30}x_5^{63}$& $173.\ x_1^{7}x_2^{7}x_3x_4^{63}x_5^{30}$& $174.\ x_1^{7}x_2^{7}x_3^{7}x_4^{24}x_5^{63} $\cr $175.\ x_1^{7}x_2^{7}x_3^{7}x_4^{63}x_5^{24}$& $176.\ x_1^{7}x_2^{7}x_3^{9}x_4^{22}x_5^{63}$& $177.\ x_1^{7}x_2^{7}x_3^{9}x_4^{63}x_5^{22} $\cr  $178.\ x_1^{7}x_2^{7}x_3^{25}x_4^{6}x_5^{63}$& $179.\ x_1^{7}x_2^{7}x_3^{25}x_4^{63}x_5^{6}$& $180.\ x_1^{7}x_2^{7}x_3^{63}x_4x_5^{30} $\cr  $181.\ x_1^{7}x_2^{7}x_3^{63}x_4^{7}x_5^{24}$& $182.\ x_1^{7}x_2^{7}x_3^{63}x_4^{9}x_5^{22}$& $183.\ x_1^{7}x_2^{7}x_3^{63}x_4^{25}x_5^{6} $\cr  $184.\ x_1^{7}x_2^{11}x_3^{5}x_4^{22}x_5^{63}$& $185.\ x_1^{7}x_2^{11}x_3^{5}x_4^{63}x_5^{22}$& $186.\ x_1^{7}x_2^{11}x_3^{21}x_4^{6}x_5^{63} $\cr
\end{tabular}}
\centerline{\begin{tabular}{lll}
$310.\ x_1^{63}x_2^{7}x_3^{7}x_4x_5^{30}$& $311.\ x_1^{63}x_2^{7}x_3^{7}x_4^{7}x_5^{24}$& $312.\ x_1^{63}x_2^{7}x_3^{7}x_4^{9}x_5^{22} $\cr  $313.\ x_1^{63}x_2^{7}x_3^{7}x_4^{25}x_5^{6}$& $314.\ x_1^{63}x_2^{7}x_3^{11}x_4^{5}x_5^{22}$& $315.\ x_1^{63}x_2^{7}x_3^{11}x_4^{21}x_5^{6} $\cr  $316.\ x_1^{63}x_2^{7}x_3^{15}x_4x_5^{22}$& $317.\ x_1^{63}x_2^{7}x_3^{27}x_4^{5}x_5^{6}$& $318.\ x_1^{63}x_2^{7}x_3^{31}x_4x_5^{6} $\cr   $319.\ x_1^{63}x_2^{15}x_3x_4^{6}x_5^{23}$& $320.\ x_1^{63}x_2^{15}x_3x_4^{7}x_5^{22}$& $321.\ x_1^{63}x_2^{15}x_3x_4^{22}x_5^{7} $\cr   $322.\ x_1^{63}x_2^{15}x_3x_4^{23}x_5^{6}$& $323.\ x_1^{63}x_2^{15}x_3^{3}x_4^{5}x_5^{22}$& $324.\ x_1^{63}x_2^{15}x_3^{3}x_4^{21}x_5^{6} $\cr  $325.\ x_1^{63}x_2^{15}x_3^{7}x_4x_5^{22}$& $326.\ x_1^{63}x_2^{15}x_3^{23}x_4x_5^{6}$& $327.\ x_1^{63}x_2^{31}x_3x_4^{6}x_5^{7} $\cr  $328.\ x_1^{63}x_2^{31}x_3x_4^{7}x_5^{6}$& $329.\ x_1^{63}x_2^{31}x_3^{3}x_4^{5}x_5^{6}$& $330.\ x_1^{63}x_2^{31}x_3^{7}x_4x_5^{6} $\cr   
\end{tabular}} 

\medskip
$B_5^+((4)|^3|(2)|^2|(1))\setminus \mathcal B\subset \{[a_t] : 331 \leqslant t \leqslant 1470\}$. Here, the monomials $a_t,\, 331 \leqslant t \leqslant 1470$ are determined as follows:

\medskip
\centerline{\begin{tabular}{lll}
$331.\  x_1x_2^{6}x_3^{15}x_4^{31}x_5^{55}$& $332.\  x_1x_2^{6}x_3^{15}x_4^{55}x_5^{31}$& $333.\  x_1x_2^{6}x_3^{31}x_4^{15}x_5^{55} $\cr  $334.\  x_1x_2^{6}x_3^{31}x_4^{31}x_5^{39}$& $335.\  x_1x_2^{6}x_3^{31}x_4^{47}x_5^{23}$& $336.\  x_1x_2^{7}x_3^{7}x_4^{31}x_5^{62} $\cr  $337.\  x_1x_2^{7}x_3^{7}x_4^{62}x_5^{31}$& $338.\  x_1x_2^{7}x_3^{14}x_4^{31}x_5^{55}$& $339.\  x_1x_2^{7}x_3^{14}x_4^{55}x_5^{31} $\cr  $340.\  x_1x_2^{7}x_3^{15}x_4^{23}x_5^{62}$& $341.\  x_1x_2^{7}x_3^{15}x_4^{30}x_5^{55}$& $342.\  x_1x_2^{7}x_3^{15}x_4^{31}x_5^{54} $\cr  $343.\  x_1x_2^{7}x_3^{15}x_4^{54}x_5^{31}$& $344.\  x_1x_2^{7}x_3^{15}x_4^{55}x_5^{30}$& $345.\  x_1x_2^{7}x_3^{15}x_4^{62}x_5^{23} $\cr  $346.\  x_1x_2^{7}x_3^{30}x_4^{15}x_5^{55}$& $347.\  x_1x_2^{7}x_3^{30}x_4^{31}x_5^{39}$& $348.\  x_1x_2^{7}x_3^{30}x_4^{47}x_5^{23} $\cr  $349.\  x_1x_2^{7}x_3^{31}x_4^{7}x_5^{62}$& $350.\  x_1x_2^{7}x_3^{31}x_4^{14}x_5^{55}$& $351.\  x_1x_2^{7}x_3^{31}x_4^{15}x_5^{54} $\cr  $352.\  x_1x_2^{7}x_3^{31}x_4^{30}x_5^{39}$& $353.\  x_1x_2^{7}x_3^{31}x_4^{31}x_5^{38}$& $354.\  x_1x_2^{7}x_3^{31}x_4^{46}x_5^{23} $\cr
\end{tabular}}
\centerline{\begin{tabular}{lll}
$448.\  x_1x_2^{31}x_3^{7}x_4^{62}x_5^{7}$& $449.\  x_1x_2^{31}x_3^{14}x_4^{7}x_5^{55}$& $450.\  x_1x_2^{31}x_3^{14}x_4^{23}x_5^{39} $\cr   $451.\  x_1x_2^{31}x_3^{14}x_4^{39}x_5^{23}$& $452.\  x_1x_2^{31}x_3^{14}x_4^{55}x_5^{7}$& $453.\  x_1x_2^{31}x_3^{15}x_4^{6}x_5^{55} $\cr  $454.\  x_1x_2^{31}x_3^{15}x_4^{7}x_5^{54}$& $455.\  x_1x_2^{31}x_3^{15}x_4^{22}x_5^{39}$& $456.\  x_1x_2^{31}x_3^{15}x_4^{23}x_5^{38} $\cr   $457.\  x_1x_2^{31}x_3^{15}x_4^{38}x_5^{23}$& $458.\  x_1x_2^{31}x_3^{15}x_4^{39}x_5^{22}$& $459.\  x_1x_2^{31}x_3^{15}x_4^{54}x_5^{7} $\cr$460.\  x_1x_2^{31}x_3^{15}x_4^{55}x_5^{6}$& $461.\  x_1x_2^{31}x_3^{30}x_4^{7}x_5^{39}$& $462.\  x_1x_2^{31}x_3^{30}x_4^{39}x_5^{7} $\cr  $463.\  x_1x_2^{31}x_3^{31}x_4^{6}x_5^{39}$& $464.\  x_1x_2^{31}x_3^{31}x_4^{7}x_5^{38}$& $465.\  x_1x_2^{31}x_3^{31}x_4^{38}x_5^{7} $\cr  $466.\  x_1x_2^{31}x_3^{31}x_4^{39}x_5^{6}$& $467.\  x_1x_2^{31}x_3^{38}x_4^{15}x_5^{23}$& $468.\  x_1x_2^{31}x_3^{39}x_4^{14}x_5^{23} $\cr  $469.\  x_1x_2^{31}x_3^{39}x_4^{15}x_5^{22}$& $470.\  x_1x_2^{31}x_3^{46}x_4^{7}x_5^{23}$& $471.\  x_1x_2^{31}x_3^{46}x_4^{23}x_5^{7} $\cr
\end{tabular}}
\centerline{\begin{tabular}{lll}
$592.\  x_1^{3}x_2^{15}x_3^{5}x_4^{30}x_5^{55}$& $593.\  x_1^{3}x_2^{15}x_3^{5}x_4^{31}x_5^{54}$& $594.\  x_1^{3}x_2^{15}x_3^{5}x_4^{54}x_5^{31} $\cr  $595.\  x_1^{3}x_2^{15}x_3^{5}x_4^{55}x_5^{30}$& $596.\  x_1^{3}x_2^{15}x_3^{5}x_4^{62}x_5^{23}$& $597.\  x_1^{3}x_2^{15}x_3^{7}x_4^{21}x_5^{62} $\cr  $598.\  x_1^{3}x_2^{15}x_3^{7}x_4^{29}x_5^{54}$& $599.\  x_1^{3}x_2^{15}x_3^{7}x_4^{53}x_5^{30}$& $600.\  x_1^{3}x_2^{15}x_3^{7}x_4^{61}x_5^{22} $\cr   $601.\  x_1^{3}x_2^{15}x_3^{13}x_4^{22}x_5^{55}$& $602.\  x_1^{3}x_2^{15}x_3^{13}x_4^{23}x_5^{54}$& $603.\  x_1^{3}x_2^{15}x_3^{13}x_4^{54}x_5^{23} $\cr  $604.\  x_1^{3}x_2^{15}x_3^{13}x_4^{55}x_5^{22}$& $605.\  x_1^{3}x_2^{15}x_3^{15}x_4^{21}x_5^{54}$& $606.\  x_1^{3}x_2^{15}x_3^{15}x_4^{53}x_5^{22} $\cr  $607.\  x_1^{3}x_2^{15}x_3^{21}x_4^{7}x_5^{62}$& $608.\  x_1^{3}x_2^{15}x_3^{21}x_4^{14}x_5^{55}$& $609.\  x_1^{3}x_2^{15}x_3^{21}x_4^{15}x_5^{54} $\cr  $610.\  x_1^{3}x_2^{15}x_3^{21}x_4^{30}x_5^{39}$& $611.\  x_1^{3}x_2^{15}x_3^{21}x_4^{31}x_5^{38}$& $612.\  x_1^{3}x_2^{15}x_3^{21}x_4^{38}x_5^{31} $\cr  $613.\  x_1^{3}x_2^{15}x_3^{21}x_4^{39}x_5^{30}$& $614.\  x_1^{3}x_2^{15}x_3^{21}x_4^{46}x_5^{23}$& $615.\  x_1^{3}x_2^{15}x_3^{21}x_4^{47}x_5^{22} $\cr
\end{tabular}}
\centerline{\begin{tabular}{lll}
$739.\  x_1^{3}x_2^{61}x_3^{14}x_4^{23}x_5^{7}$& $740.\  x_1^{3}x_2^{61}x_3^{15}x_4^{6}x_5^{23}$& $741.\  x_1^{3}x_2^{61}x_3^{15}x_4^{7}x_5^{22} $\cr  $742.\  x_1^{3}x_2^{61}x_3^{15}x_4^{22}x_5^{7}$& $743.\  x_1^{3}x_2^{61}x_3^{15}x_4^{23}x_5^{6}$& $744.\  x_1^{3}x_2^{61}x_3^{30}x_4^{7}x_5^{7} $\cr   $745.\  x_1^{3}x_2^{61}x_3^{31}x_4^{6}x_5^{7}$& $746.\  x_1^{3}x_2^{61}x_3^{31}x_4^{7}x_5^{6}$& $747.\  x_1^{7}x_2x_3^{7}x_4^{31}x_5^{62} $\cr  $748.\  x_1^{7}x_2x_3^{7}x_4^{62}x_5^{31}$& $749.\  x_1^{7}x_2x_3^{14}x_4^{31}x_5^{55}$& $750.\  x_1^{7}x_2x_3^{14}x_4^{55}x_5^{31} $\cr  $751.\  x_1^{7}x_2x_3^{15}x_4^{23}x_5^{62}$& $752.\  x_1^{7}x_2x_3^{15}x_4^{30}x_5^{55}$& $753.\  x_1^{7}x_2x_3^{15}x_4^{31}x_5^{54} $\cr  $754.\  x_1^{7}x_2x_3^{15}x_4^{54}x_5^{31}$& $755.\  x_1^{7}x_2x_3^{15}x_4^{55}x_5^{30}$& $756.\  x_1^{7}x_2x_3^{15}x_4^{62}x_5^{23} $\cr  $757.\  x_1^{7}x_2x_3^{30}x_4^{15}x_5^{55}$& $758.\  x_1^{7}x_2x_3^{30}x_4^{31}x_5^{39}$& $759.\  x_1^{7}x_2x_3^{30}x_4^{47}x_5^{23} $\cr  $760.\  x_1^{7}x_2x_3^{31}x_4^{7}x_5^{62}$& $761.\  x_1^{7}x_2x_3^{31}x_4^{14}x_5^{55}$& $762.\  x_1^{7}x_2x_3^{31}x_4^{15}x_5^{54} $\cr
\end{tabular}}
\centerline{\begin{tabular}{lll}
$886.\  x_1^{7}x_2^{11}x_3^{23}x_4^{45}x_5^{22}$& $887.\  x_1^{7}x_2^{11}x_3^{23}x_4^{61}x_5^{6}$& $888.\  x_1^{7}x_2^{11}x_3^{29}x_4^{6}x_5^{55} $\cr   $889.\  x_1^{7}x_2^{11}x_3^{29}x_4^{7}x_5^{54}$& $890.\  x_1^{7}x_2^{11}x_3^{29}x_4^{22}x_5^{39}$& $891.\  x_1^{7}x_2^{11}x_3^{29}x_4^{23}x_5^{38} $\cr  $892.\  x_1^{7}x_2^{11}x_3^{29}x_4^{38}x_5^{23}$& $893.\  x_1^{7}x_2^{11}x_3^{29}x_4^{39}x_5^{22}$& $894.\  x_1^{7}x_2^{11}x_3^{29}x_4^{54}x_5^{7} $\cr  $895.\  x_1^{7}x_2^{11}x_3^{29}x_4^{55}x_5^{6}$& $896.\  x_1^{7}x_2^{11}x_3^{31}x_4^{5}x_5^{54}$& $897.\  x_1^{7}x_2^{11}x_3^{31}x_4^{21}x_5^{38} $\cr  $898.\  x_1^{7}x_2^{11}x_3^{31}x_4^{37}x_5^{22}$& $899.\  x_1^{7}x_2^{11}x_3^{31}x_4^{53}x_5^{6}$& $900.\  x_1^{7}x_2^{11}x_3^{53}x_4^{6}x_5^{31} $\cr  $901.\  x_1^{7}x_2^{11}x_3^{53}x_4^{7}x_5^{30}$& $902.\  x_1^{7}x_2^{11}x_3^{53}x_4^{14}x_5^{23}$& $903.\  x_1^{7}x_2^{11}x_3^{53}x_4^{15}x_5^{22} $\cr  $904.\  x_1^{7}x_2^{11}x_3^{53}x_4^{30}x_5^{7}$& $905.\  x_1^{7}x_2^{11}x_3^{53}x_4^{31}x_5^{6}$& $906.\  x_1^{7}x_2^{11}x_3^{55}x_4^{5}x_5^{30} $\cr  $907.\  x_1^{7}x_2^{11}x_3^{55}x_4^{13}x_5^{22}$& $908.\  x_1^{7}x_2^{11}x_3^{55}x_4^{29}x_5^{6}$& $909.\  x_1^{7}x_2^{11}x_3^{61}x_4^{6}x_5^{23} $\cr
\end{tabular}}
\centerline{\begin{tabular}{lll}
$1036.\  x_1^{7}x_2^{59}x_3^{5}x_4^{7}x_5^{30}$& $1037.\  x_1^{7}x_2^{59}x_3^{5}x_4^{14}x_5^{23}$& $1038.\  x_1^{7}x_2^{59}x_3^{5}x_4^{15}x_5^{22} $\cr  $1039.\  x_1^{7}x_2^{59}x_3^{5}x_4^{30}x_5^{7}$& $1040.\  x_1^{7}x_2^{59}x_3^{5}x_4^{31}x_5^{6}$& $1041.\  x_1^{7}x_2^{59}x_3^{7}x_4^{5}x_5^{30} $\cr  $1042.\  x_1^{7}x_2^{59}x_3^{7}x_4^{13}x_5^{22}$& $1043.\  x_1^{7}x_2^{59}x_3^{7}x_4^{29}x_5^{6}$& $1044.\  x_1^{7}x_2^{59}x_3^{13}x_4^{6}x_5^{23} $\cr  $1045.\  x_1^{7}x_2^{59}x_3^{13}x_4^{7}x_5^{22}$& $1046.\  x_1^{7}x_2^{59}x_3^{13}x_4^{22}x_5^{7}$& $1047.\  x_1^{7}x_2^{59}x_3^{13}x_4^{23}x_5^{6} $\cr  $1048.\  x_1^{7}x_2^{59}x_3^{15}x_4^{5}x_5^{22}$& $1049.\  x_1^{7}x_2^{59}x_3^{15}x_4^{21}x_5^{6}$& $1050.\  x_1^{7}x_2^{59}x_3^{29}x_4^{6}x_5^{7} $\cr  $1051.\  x_1^{7}x_2^{59}x_3^{29}x_4^{7}x_5^{6}$& $1052.\  x_1^{7}x_2^{59}x_3^{31}x_4^{5}x_5^{6}$& $1053.\  x_1^{15}x_2x_3^{6}x_4^{31}x_5^{55} $\cr  $1054.\  x_1^{15}x_2x_3^{6}x_4^{55}x_5^{31}$& $1055.\  x_1^{15}x_2x_3^{7}x_4^{23}x_5^{62}$& $1056.\  x_1^{15}x_2x_3^{7}x_4^{30}x_5^{55} $\cr  $1057.\  x_1^{15}x_2x_3^{7}x_4^{31}x_5^{54}$& $1058.\  x_1^{15}x_2x_3^{7}x_4^{54}x_5^{31}$& $1059.\  x_1^{15}x_2x_3^{7}x_4^{55}x_5^{30} $\cr
\end{tabular}}
\centerline{\begin{tabular}{lll}
$1183.\  x_1^{15}x_2^{7}x_3^{19}x_4^{5}x_5^{62}$& $1184.\  x_1^{15}x_2^{7}x_3^{19}x_4^{13}x_5^{54}$& $1185.\  x_1^{15}x_2^{7}x_3^{19}x_4^{29}x_5^{38} $\cr  $1186.\  x_1^{15}x_2^{7}x_3^{19}x_4^{37}x_5^{30}$& $1187.\  x_1^{15}x_2^{7}x_3^{19}x_4^{45}x_5^{22}$& $1188.\  x_1^{15}x_2^{7}x_3^{19}x_4^{61}x_5^{6} $\cr  $1189.\  x_1^{15}x_2^{7}x_3^{23}x_4x_5^{62}$& $1190.\  x_1^{15}x_2^{7}x_3^{23}x_4^{7}x_5^{56}$& $1191.\  x_1^{15}x_2^{7}x_3^{23}x_4^{9}x_5^{54} $\cr  $1192.\  x_1^{15}x_2^{7}x_3^{23}x_4^{25}x_5^{38}$& $1193.\  x_1^{15}x_2^{7}x_3^{23}x_4^{41}x_5^{22}$& $1194.\  x_1^{15}x_2^{7}x_3^{23}x_4^{57}x_5^{6} $\cr  $1195.\  x_1^{15}x_2^{7}x_3^{27}x_4^{5}x_5^{54}$& $1196.\  x_1^{15}x_2^{7}x_3^{27}x_4^{21}x_5^{38}$& $1197.\  x_1^{15}x_2^{7}x_3^{27}x_4^{37}x_5^{22} $\cr  $1198.\  x_1^{15}x_2^{7}x_3^{27}x_4^{53}x_5^{6}$& $1199.\  x_1^{15}x_2^{7}x_3^{31}x_4x_5^{54}$& $1200.\  x_1^{15}x_2^{7}x_3^{51}x_4^{5}x_5^{30} $\cr  $1201.\  x_1^{15}x_2^{7}x_3^{51}x_4^{13}x_5^{22}$& $1202.\  x_1^{15}x_2^{7}x_3^{51}x_4^{29}x_5^{6}$& $1203.\  x_1^{15}x_2^{7}x_3^{55}x_4x_5^{30} $\cr  $1204.\  x_1^{15}x_2^{7}x_3^{55}x_4^{7}x_5^{24}$& $1205.\  x_1^{15}x_2^{7}x_3^{55}x_4^{9}x_5^{22}$& $1206.\  x_1^{15}x_2^{7}x_3^{55}x_4^{25}x_5^{6} $\cr
\end{tabular}}
\centerline{\begin{tabular}{lll}
$1330.\  x_1^{31}x_2x_3^{15}x_4^{55}x_5^{6}$& $1331.\  x_1^{31}x_2x_3^{30}x_4^{7}x_5^{39}$& $1332.\  x_1^{31}x_2x_3^{30}x_4^{39}x_5^{7} $\cr  $1333.\  x_1^{31}x_2x_3^{31}x_4^{6}x_5^{39}$& $1334.\  x_1^{31}x_2x_3^{31}x_4^{7}x_5^{38}$& $1335.\  x_1^{31}x_2x_3^{31}x_4^{38}x_5^{7} $\cr  $1336.\  x_1^{31}x_2x_3^{31}x_4^{39}x_5^{6}$& $1337.\  x_1^{31}x_2x_3^{38}x_4^{15}x_5^{23}$& $1338.\  x_1^{31}x_2x_3^{39}x_4^{14}x_5^{23} $\cr  $1339.\  x_1^{31}x_2x_3^{39}x_4^{15}x_5^{22}$& $1340.\  x_1^{31}x_2x_3^{46}x_4^{7}x_5^{23}$& $1341.\  x_1^{31}x_2x_3^{46}x_4^{23}x_5^{7} $\cr  $1342.\  x_1^{31}x_2x_3^{47}x_4^{6}x_5^{23}$& $1343.\  x_1^{31}x_2x_3^{47}x_4^{7}x_5^{22}$& $1344.\  x_1^{31}x_2x_3^{47}x_4^{22}x_5^{7} $\cr  $1345.\  x_1^{31}x_2x_3^{47}x_4^{23}x_5^{6}$& $1346.\  x_1^{31}x_2x_3^{62}x_4^{7}x_5^{7}$& $1347.\  x_1^{31}x_2^{3}x_3^{5}x_4^{7}x_5^{62} $\cr  $1348.\  x_1^{31}x_2^{3}x_3^{5}x_4^{14}x_5^{55}$& $1349.\  x_1^{31}x_2^{3}x_3^{5}x_4^{15}x_5^{54}$& $1350.\  x_1^{31}x_2^{3}x_3^{5}x_4^{30}x_5^{39} $\cr  $1351.\  x_1^{31}x_2^{3}x_3^{5}x_4^{31}x_5^{38}$& $1352.\  x_1^{31}x_2^{3}x_3^{5}x_4^{46}x_5^{23}$& $1353.\  x_1^{31}x_2^{3}x_3^{5}x_4^{47}x_5^{22} $\cr
\end{tabular}} 

\medskip
\subsection{A set of generators for $(\widetilde{SF}_5)_{108}$}\label{s52}\

\medskip
We have $(\widetilde{SF}_5)_{108} = \langle [p_j] : 1 \leqslant j \leqslant 14\rangle$, where the polynomials $p_j$, $1 \leqslant j \leqslant 14$, are determined as follows:
%

\eject
%===============================
{}

\bigskip
\end{document}